\documentclass[11pt,letterpaper]{article}

\usepackage{graphicx,subfigure}

\def\jamesmode{0}
\def\arxivmode{0}
\def\fastmode{1}

\def\showauthornotes{0}

\def\showkeys{0}
\def\showdraftbox{1}
\def\showcolorlinks{1}
\def\usemicrotype{1}
\def\showfixme{1}

\ifnum\arxivmode=0 

\else

\newcommand{\vvmathbb}{\mathbb}
\fi

\newcommand{\capacity}{\mathrm{cap}^G}

\newcommand{\invcheeger}{\Phi^{\mathrm{inv}}}

\newcommand{\F}{\Phi}

\newcommand{\dloc}{\vvmathbb{d}_{\mathrm{loc}}}
\newcommand{\graphs}{\cG}
\newcommand{\rgraphs}{\cG_{\bullet}}
\newcommand{\rrgraphs}{\cG_{\bullet \bullet}}
\newcommand{\cgraphs}{\graphs^*}
\newcommand{\crgraphs}{\rgraphs^*}
\newcommand{\crrgraphs}{\rrgraphs^*}

\newcommand{\todl}{\Rightarrow}

\newcommand{\smashed}[1]{\sqrt{\smash[b]{#1}}}
\newcommand{\avgd}{\bar{d}}
\newcommand{\dimspec}{\mathrm{dim}_{\mathsf{sp}}}
\newcommand{\dimspecover}{\obar{\mathrm{dim}}_{\mathsf{sp}}}
\newcommand{\dimspecunder}{\ubar{\mathrm{dim}}_{\mathsf{sp}}}
\newcommand{\dimconf}{\mathrm{dim}_{\mathsf{cg}}}
\newcommand{\dimconfunder}{\ubar{\mathrm{dim}}_{\mathsf{cg}}}
\newcommand{\dimconfover}{\obar{\mathrm{dim}}_{\mathsf{cg}}}

\newcommand{\Reff}{R_{\mathrm{eff}}}

\ifnum\fastmode=0
\usepackage{etex}
\fi

\ifnum\fastmode=0
\usepackage[l2tabu, orthodox]{nag}
\fi

\usepackage{xspace,enumerate}

\usepackage[dvipsnames]{xcolor}

\ifnum\fastmode=0
\usepackage[T1]{fontenc}
\usepackage[full]{textcomp}
\fi

\usepackage[american]{babel}

\usepackage{mathtools}

\usepackage{amsthm}

\newtheorem{theorem}{Theorem}[section]
\newtheorem*{theorem*}{Theorem}

\newtheorem*{proposition*}{Proposition}
\newtheorem{lemma}[theorem]{Lemma}
\newtheorem*{lemma*}{Lemma}
\newtheorem{corollary}[theorem]{Corollary}
\newtheorem*{conjecture*}{Conjecture}

\newtheorem*{fact*}{Fact}

\newtheorem*{exercise*}{Exercise}

\newtheorem*{hypothesis*}{Hypothesis}

\theoremstyle{definition}
\newtheorem{definition}[theorem]{Definition}

\newtheorem{assumption}[theorem]{Assumption}
\newtheorem{exercise-easy}[theorem]{Exercise}
\newtheorem{exercise-med}[theorem]{Exercise}
\newtheorem{exercise-hard}[theorem]{Exercise$^\star$}
\newtheorem{claim}[theorem]{Claim}
\newtheorem*{claim*}{Claim}

\newtheorem{remark}[theorem]{Remark}
\newtheorem*{remark*}{Remark}

\newtheorem*{observation*}{Observation}

\usepackage[
letterpaper,
top=0.7in,
bottom=0.9in,
left=1in,
right=1in]{geometry}

\ifnum\arxivmode=0
\usepackage{newpxtext} %
\usepackage{textcomp} %
\usepackage[varg,bigdelims]{newpxmath}
\usepackage[scr=rsfso]{mathalfa}%
\usepackage{bm} %
\let\mathbb\varmathbb
\fi

\ifnum\arxivmode=1
\usepackage[varg]{pxfonts} %
\usepackage{textcomp} %
\usepackage[scr=rsfso]{mathalfa}%
\usepackage{bm} %
\fi

\ifnum\showkeys=1
\usepackage[color]{showkeys}
\fi

\makeatletter
\@ifpackageloaded{hyperref}
{}
{\usepackage[pagebackref]{hyperref}}

\definecolor{bleudefrance}{rgb}{0.01, 0.1, 1.0}
\definecolor{azure}{rgb}{0.0, 0.5, 1.0}

\ifnum\showcolorlinks=1
\hypersetup{
pagebackref=true,
colorlinks=true,
urlcolor=blue,
linkcolor=blue,
citecolor=OliveGreen}
\fi

\ifnum\showcolorlinks=0
\hypersetup{
pagebackref=true,
colorlinks=false,
pdfborder={0 0 0}}
\fi

\usepackage{prettyref}

\newcommand{\savehyperref}[2]{\texorpdfstring{\hyperref[#1]{#2}}{#2}}

\newrefformat{eq}{\savehyperref{#1}{\textup{(\ref*{#1})}}}
\newrefformat{lem}{\savehyperref{#1}{Lemma~\ref*{#1}}}
\newrefformat{def}{\savehyperref{#1}{Definition~\ref*{#1}}}
\newrefformat{thm}{\savehyperref{#1}{Theorem~\ref*{#1}}}
\newrefformat{cor}{\savehyperref{#1}{Corollary~\ref*{#1}}}
\newrefformat{cha}{\savehyperref{#1}{Chapter~\ref*{#1}}}
\newrefformat{sec}{\savehyperref{#1}{Section~\ref*{#1}}}
\newrefformat{app}{\savehyperref{#1}{Appendix~\ref*{#1}}}
\newrefformat{tab}{\savehyperref{#1}{Table~\ref*{#1}}}
\newrefformat{fig}{\savehyperref{#1}{Figure~\ref*{#1}}}
\newrefformat{hyp}{\savehyperref{#1}{Hypothesis~\ref*{#1}}}
\newrefformat{alg}{\savehyperref{#1}{Algorithm~\ref*{#1}}}
\newrefformat{rem}{\savehyperref{#1}{Remark~\ref*{#1}}}
\newrefformat{item}{\savehyperref{#1}{Item~\ref*{#1}}}
\newrefformat{step}{\savehyperref{#1}{step~\ref*{#1}}}
\newrefformat{conj}{\savehyperref{#1}{Conjecture~\ref*{#1}}}
\newrefformat{fact}{\savehyperref{#1}{Fact~\ref*{#1}}}
\newrefformat{prop}{\savehyperref{#1}{Proposition~\ref*{#1}}}
\newrefformat{prob}{\savehyperref{#1}{Problem~\ref*{#1}}}
\newrefformat{claim}{\savehyperref{#1}{Claim~\ref*{#1}}}
\newrefformat{relax}{\savehyperref{#1}{Relaxation~\ref*{#1}}}
\newrefformat{red}{\savehyperref{#1}{Reduction~\ref*{#1}}}
\newrefformat{part}{\savehyperref{#1}{Part~\ref*{#1}}}
\newrefformat{ex}{\savehyperref{#1}{Exercise~\ref*{#1}}}
\newrefformat{property}{\savehyperref{#1}{Property~\ref*{#1}}}
\newrefformat{assume}{\savehyperref{#1}{Assumption~\ref*{#1}}}

\newcommand{\Sref}[1]{\hyperref[#1]{\S\ref*{#1}}}

\usepackage{nicefrac}

\ifnum\fastmode=0
\ifnum\usemicrotype=1
\usepackage{microtype}
\fi
\fi

\ifnum\showauthornotes=2
\newcommand{\mynotes}[1]{{\sffamily\small\color{teal}{#1}}\medskip}
\newcommand{\Authornote}[2]{{\sffamily\small\color{brickred}{[#1: #2]}}\medskip}
\newcommand{\Authornotecolored}[3]{{\sffamily\small\color{#1}{[#2: #3]}}}
\newcommand{\Authorcomment}[2]{{\sffamily\small\color{BrickRed}{[#1: #2]}}}
\newcommand{\Authorstartcomment}[1]{\sffamily\small\color{gray}[#1: }

\newcommand{\Authorfnote}[2]{\footnote{\color{red}{#1: #2}}}
\newcommand{\Authorfixme}[1]{\Authornote{#1}{\textbf{??}}}
\newcommand{\Authormarginmark}[1]{\marginpar{\textcolor{red}{\fbox{\Large #1:!}}}}
\newcommand{\myexplain}[1]{{\sffamily\small\color{red}{\noindent [Explanation:\medskip\newline \begin{quote}#1\hfill]\end{quote}}}\medskip}

\else

\newcommand{\mynotes}[1]{}
\newcommand{\Authornote}[2]{}
\newcommand{\Authornotecolored}[3]{}
\newcommand{\Authorcomment}[2]{}
\newcommand{\Authorstartcomment}[1]{}

\newcommand{\Authorfnote}[2]{}
\newcommand{\Authorfixme}[1]{}
\newcommand{\Authormarginmark}[1]{}
\newcommand{\myexplain}[1]{}

\ifnum\showauthornotes=1
\renewcommand{\Authornote}[2]{{\sffamily\small\color{brickred}{[#1: #2]}}\medskip}
\renewcommand{\Authornotecolored}[3]{{\sffamily\small\color{#1}{[#2: #3]}}}
\newcommand{\Authorfnote}[2]{\footnote{\color{brickred}{#1: #2}}}
\fi

\fi

\newcommand{\jnote}{\Authornote{J}}

\ifnum\showfixme=0

\fi

\usepackage{boxedminipage}

\newcommand{\Esymb}{\mathbb{E}}
\newcommand{\Psymb}{\mathbb{P}}

\DeclareMathOperator*{\E}{\Esymb}

\DeclareMathOperator*{\ProbOp}{\Psymb}

\renewcommand{\Pr}{\ProbOp}

\newcommand{\textparen}[1]{\text{(#1)}}

\ifx\because\undefined
\newcommand{\because}[1]{\textparen{because #1}}
\else
\renewcommand{\because}[1]{\textparen{because #1}}
\fi

\newcommand{\seteq}{\mathrel{\mathop:}=}

\newcommand{\bigmid}{~\big|~}

\newcommand\bdot\bullet

\ifx\mathds\undefined %
\newcommand{\Ind}{\mathbb I}
\else
\newcommand{\Ind}{\mathds 1}
\fi

\DeclareMathOperator{\vol}{vol}

\DeclareMathOperator{\supp}{supp}
\DeclareMathOperator{\dist}{dist}

\newcommand{\Z}{\mathbb Z}
\newcommand{\N}{\mathbb N}
\newcommand{\R}{\mathbb R}

\newcommand{\cB}{\mathcal B}
\newcommand{\cC}{\mathcal C}

\newcommand{\cE}{\mathcal E}
\newcommand{\cF}{\mathcal F}
\newcommand{\cG}{\mathcal G}

\newcommand{\cL}{\mathcal L}

\newcommand{\cP}{\mathcal P}
\newcommand{\cQ}{\mathcal Q}
\newcommand{\cR}{\mathcal R}

\newcommand{\cT}{\mathcal T}

\renewcommand{\leq}{\leqslant}

\renewcommand{\geq}{\geqslant}

\ifnum\showdraftbox=1

\else

\fi

\let\epsilon=\varepsilon

\numberwithin{equation}{section}

\newcommand\MYcurrentlabel{xxx}

\newcommand{\MYstore}[2]{%
  \global\expandafter \def \csname MYMEMORY #1 \endcsname{#2}%
}

\newcommand{\MYload}[1]{%
  \csname MYMEMORY #1 \endcsname%
}

\newcommand{\MYnewlabel}[1]{%
  \renewcommand\MYcurrentlabel{#1}%
  \MYoldlabel{#1}%
}

\newcommand{\MYdummylabel}[1]{}

\newcommand{\torestate}[1]{%
  \let\MYoldlabel\label%
  \let\label\MYnewlabel%
  #1%
  \MYstore{\MYcurrentlabel}{#1}%
  \let\label\MYoldlabel%
}

\newcommand{\restatetheorem}[1]{%
  \let\MYoldlabel\label
  \let\label\MYdummylabel
  \begin{theorem*}[Restatement of \prettyref{#1}]
    \MYload{#1}
  \end{theorem*}
  \let\label\MYoldlabel
}

\newcommand{\restatelemma}[1]{%
  \let\MYoldlabel\label
  \let\label\MYdummylabel
  \begin{lemma*}[Restatement of \prettyref{#1}]
    \MYload{#1}
  \end{lemma*}
  \let\label\MYoldlabel
}

\newcommand{\restateprop}[1]{%
  \let\MYoldlabel\label
  \let\label\MYdummylabel
  \begin{proposition*}[Restatement of \prettyref{#1}]
    \MYload{#1}
  \end{proposition*}
  \let\label\MYoldlabel
}

\newcommand{\restatefact}[1]{%
  \let\MYoldlabel\label
  \let\label\MYdummylabel
  \begin{fact*}[Restatement of \prettyref{#1}]
    \MYload{#1}
  \end{fact*}
  \let\label\MYoldlabel
}

\newcommand{\restate}[1]{%
  \let\MYoldlabel\label
  \let\label\MYdummylabel
  \MYload{#1}
  \let\label\MYoldlabel
}

\newcommand{\addreferencesection}{
  \phantomsection
\ifnum\stocmode=0
  \addcontentsline{toc}{section}{References}
\else
  \addcontentsline{toc}{section}{References \hspace*{1in} --------- End of extended abstract ---------}
\fi

}

\newcommand{\e}{\epsilon}

\let\origparagraph\paragraph
\renewcommand{\paragraph}[1]{\vspace*{-20pt}\hspace*{-5pt}\origparagraph{#1.}}

\allowdisplaybreaks

\usepackage{paralist}

\usepackage{comment}

\let\pref=\prettyref

\newcommand{\diam}{\mathsf{diam}}

\DeclareMathOperator{\tr}{tr}

\renewcommand{\Ind}{\vvmathbb{1}}
\ifnum\jamesmode=0
\fi
\let\1\Ind

\newcommand\f{\varphi}

\usepackage{accents}

\usepackage{enumitem}
\usepackage[normalem]{ulem}

\newcommand\myuline{\bgroup\markoverwith{\rule[-0.4ex]{2pt}{0.2mm}}\ULon}
\newcommand{\obar}[1]{\smash{\mkern3mu\overline{\mkern-3mu \vphantom{\scalebox{0.85}{\ensuremath{#1}}} \smash{#1}\mkern-3mu}\mkern3mu}}
\newcommand{\ubar}[1]{\smash{\mkern2mu\myuline{\mkern-2mu \smash{#1}\mkern-2mu}\mkern2mu}}

\newlist{assumptions}{enumerate}{10}
\setlist[assumptions]{label*=\arabic*}

\renewcommand{\deg}{\mathrm{deg}}

\newcommand{\reff}{R_{\mathrm{eff}}}
\newcommand{\dmax}{d_{\max}}
\newcommand{\len}{\mathrm{len}}
\newcommand{\area}{\mathcal{A}}
\newcommand{\ao}{\area_{\omega}}

\newcommand{\rig}{\mathsf{rig}}

\renewcommand{\Z}{\vvmathbb{Z}}
\renewcommand{\R}{\vvmathbb{R}}

\renewcommand{\mathbb}{\vvmathbb}

\begin{document}

\title{Conformal growth rates and spectral geometry \\ on distributional limits of graphs}
   \author{James R. Lee \\ {\small University of Washington}}
\date{}

\maketitle

\begin{abstract}
   For a unimodular random graph $(G,\rho)$, we consider deformations
   of its intrinsic path metric by a (random) weighting of its vertices.
   This leads to the notion of the {\em conformal growth exponent of $(G,\rho)$},
   which is the best asymptotic degree of volume growth of balls that can be
   achieved by such a reweighting.
   Under moment conditions on the degree of the root,
   we show that the conformal growth exponent of a unimodular random graph
   bounds its almost sure spectral dimension.
   This has interesting consequences for many low-dimensional models.

   The consequences in dimension two are particularly strong.
   It establishes that models like the
   uniform infinite planar triangulation (UIPT) and
   quadrangulation (UIPQ) almost surely have spectral dimension at most two.
   It also establishes a conjecture of Benjamini and Schramm (2001)
   by extending their Recurrence Theorem from planar graphs
   to arbitrary families of $H$-minor free graphs.
   More generally, it strengthens the work of Gurel-Gurevich and Nachmias (2013)
   who established recurrence for distributional limits of planar graphs
   when the degree of the root has exponential tails.

   We further present a general method for proving subdiffusivity of the
   random walk on a large class of models, including UIPT and UIPQ,
   using only the volume growth profile of balls in the intrinsic metric.
\end{abstract}

\newpage

\begingroup
\hypersetup{linktocpage=false}
\renewcommand{\baselinestretch}{0.9}\normalsize
\tableofcontents
\renewcommand{\baselinestretch}{1.0}\normalsize
\endgroup

\newpage

\section{Introduction}

\newcommand\myov{\bgroup\markoverwith{\rule[8.5pt]{0.1pt}{0.5pt}}\ULon}
\newcommand{\myubar}[1]{\mkern2mu\myov{\mkern-2mu #1\mkern-2mu}\mkern2mu}

   Motivated by the study of random surfaces in Quantum Geometry \cite{ADJ97},
   Benjamini and Schramm \cite{BS01} sought to understand the behavior
   of random planar triangulations.
   Toward this end, they
   introduced the notion of the {\em distributional limit} of a sequence
   of finite graphs $\{G_n\}$.  This limit is a random rooted infinite graph $(G,\rho)$
   with the property that the laws of neighborhoods of a randomly chosen vertex in $G_n$
   converge, as $n \to \infty$, to the laws of neighborhoods of $\rho$ in $G$.
   When the limit exists, it is a {\em unimodular random graph} in the sense of Aldous and Lyons \cite{aldous-lyons}.
   (See \pref{sec:unimodular} for a discussion of the weak local topology and unimodular random graphs.)

   An example of central importance is the {\em uniform infinite planar triangulation (UIPT)} of
   Angel and Schramm \cite{AS03} which is obtained by taking the distributional limit
   of a uniform random triangulation of the $2$-sphere with $n$ vertices.
   A well-studied variant is the {\em uniform infinite planar quadrangulation (UIPQ)}
   constructed by Krikun \cite{Krikun05}.
   More recently, Benjamini and Curien \cite{BC11} sought to extend
   these studies to graphs that can be sphere-packed in $\R^d$ for $d \geq 3$,
   but noted that the defining natural models in higher dimensions
   is a subtle issue.

   In general, the goal of this line of work is to understand the almost sure
   geometric properties of the limit object, where often interesting phenomena emerge.
   For instance, Angel \cite{Angel03} has shown that almost surely balls
   of radius $R$ in UIPT have volume $R^{4+o(1)}$, but such a ball can be separated
   from infinity by removing only $R^{1+o(1)}$ vertices.
   This reflects the fractal geometry of UIPT and leads one
   to suspect, for instance, that the random walk should be recurrent,
   and the speed of the walk should be subdiffusive.

   Indeed, Benjamini and Curien \cite{BC13} 
   proved that the random walk in UIPQ is
   almost surely subdiffusive:
   The average distance from the starting point is at most $T^{1/3+o(1)}$
   after $T$ steps; the correct exponent is conjectured to be $1/4$,
   as predicted by the KPZ relations (see the discussion in \cite{BC13}).
   More recently, Gurel-Gurevich and Nachmias \cite{GN13}
   established that the random walk on UIPT and UIPQ is almost surely recurrent.

   While the theory developed here applies to a wide range of distributional limits,
   our methods yield new proofs of these preceding results in substantially more general settings,
   and more detailed
   information even for the specific models of UIPT and UIPQ.
   For instance, we establish that the spectral dimension of UIPT/UIPQ
   is almost surely at most two, confirming a long-held belief.
   In fact, this holds for any distributional limit of finite planar graphs
   when the degree of the root has sufficiently nice tails.
   For UIPT (and a family of other planar models), a matching lower bound 
   was recently established by Gwynne and Miller \cite{GM17}.
   We are additionally able to strengthen the almost sure recurrence for UIPT/UIPQ to the
   conclusion that almost surely the number of returns to the root
   by time $T$ grows asymptotically faster than $\log \log T$.

   We show that the $T^{1/3+o(1)}$ speed bound for UIPT/UIPQ actually
   holds for any unimodular random planar graph with quartic volume growth,
   or alternately in any unimodular random graph where there
   is a large enough discrepancy between the volume growth
   and isoperimetric profile.
   After initial circulation of this manuscript, this method has been
   used to establish that the speed in UIPT is $T^{1/4+o(1)}$ \cite{GH18};
   this estimate is sharp, as it matches the lower bound proved in \cite{GM17}.
   Moreover, the method of conformal weights was used recently to show that random walk on
   the 2D incipient infinite cluster is almost surely subdiffusive
   in the chemical distance \cite{GL20}.

      Previous work on distributional limits of planar graphs relies
   heavily on the analysis of circle packings, which can
   be thought of as ambient representations that conformally uniformize the geometry of the underlying graph.
   Here we take an intrinsic approach, deforming the graph geometry directly
   using a family of discrete graph metrics.
   This makes our methods significantly more flexible
   and applicable to a much broader family of graphs.
   The connection between discrete uniformization and spectral geometry of graphs
   is present in
   earlier joint works with Biswal and Rao \cite{BLR08} and Kelner, Price, and Teng \cite{KLPT09}, where
   we showed how such metrics can be used to control the spectrum of the Laplacian
   in bounded-degree graphs.

\subsection{Discrete conformal metrics and the growth exponent}

   Consider a locally finite, connected graph $G$.
A {\em conformal metric on $G$} is a map $\omega : V(G) \to \R_+$.
The metric endows $G$ with a graph distance as follows:  Give to every edge $\{u,v\} \in E(G)$ a length
$\len_{\omega}(\{u,v\}) \seteq \frac12 (\omega(u)+\omega(v))$.
This prescribes to every path $\gamma = \{v_0, v_1, v_2, \ldots\}$ in $G$ the induced length
\[
   \len_{\omega}(\gamma) \seteq \sum_{k \geq 0} \len_{\omega}(\{v_k, v_{k+1}\})\,.
\]
Now for $u,v \in V(G)$,
one defines the path metric $\dist_{\omega}(u,v)$ as the infimum of the lengths of all $u$-$v$ paths in $G$.
Denote the closed ball
\[
   B_{\omega}(x,R) = \left\{ y \in V(G) : \dist_{\omega}(x,y) \leq R \right\}\,.
\]

If $(G,\rho)$ is a unimodular random graph, then a {\em conformal metric on $(G,\rho)$}
is a (marked) unimodular random graph $(G',\omega,\rho')$ with $\omega : V(G) \to \R_+$
such that $(G,\rho)$ and $(G',\rho')$
have the same law.  We say that the conformal weight is {\em normalized} if $\E\left[\omega(\rho)^2\right] = 1$.
See \pref{sec:unimodular} for precise definitions.

One thinks of such a metric $\omega : V(G) \to \R_+$ as deforming the geometry
of the underlying graph.  It will turn out that normalized conformal metrics
with nice geometric properties form a powerful tool in understanding
the structure of $(G,\rho)$.
A basic property one might hope for is controlled volume growth of balls:
$|B_{\omega}(\rho,R)| \leq O(R^d)$ for some fixed $d > 0$.
As we will see,
the best exponent $d$ one can achieve
controls the spectral dimension of $G$ from above.

\medskip
\noindent
{\bf Spectral dimension vs. conformal growth exponent.}
Consider a unimodular random graph $(G,\rho)$.
We define the {\em upper and lower conformal growth exponents of $(G,\rho)$}, respectively, by
\begin{align*}
   \dimconfover(G,\rho) & \seteq \inf_{\omega} \limsup_{R \to \infty} \frac{\log \|\#B_{\omega}(\rho,R)\|_{L^\infty}}{\log R}\,, \\
   \dimconfunder(G,\rho) & \seteq \inf_{\omega} \liminf_{R \to \infty} \frac{\log \|\#B_{\omega}(\rho,R)\|_{L^\infty}}{\log R}\,,
\end{align*}
and the infimum is over all normalized conformal metrics on $(G,\rho)$, and 
we use
\[
   \|X\|_{L^{\infty}} \seteq \sup \left\{ \lambda > 0 : \Pr(X < \lambda) = 1 \right\}
\] 
to denote the essential supremum of a random variable $X$,
and $\# S$ to denote the cardinality of a finite set $S$.

When $\dimconfover(G,\rho) = \dimconfunder(G,\rho)$, we define the {\em conformal growth exponent}
by
\[
   \dimconf(G,\rho) \seteq \dimconfover(G,\rho) = \dimconfunder(G,\rho)\,.
\]
Note that the quantities $\dimconfover,\dimconfunder,\dimconf$ are functions of the law of $(G,\rho)$;
they are not defined on (fixed) rooted graphs.

As an indication that the conformal growth exponent can be bounded
in interesting settings, let us 
state the next theorem which is proved in the companion paper \cite{Lee17b}.
We use $\todl$ to denote convergence
in the distributional sense; see \pref{sec:unimodular}.
Say that a graph $G$ is {\em sphere-packed in $\R^d$} if $G$
is the tangency graph of a collection of interior-disjoint Euclidean balls
in $\R^d$.

\begin{theorem}\label{thm:Rd-packings}
   If $\{G_n\}$ are finite graphs such that each $G_n$ is sphere-packed in $\R^d$, and
   $\{G_n\} \todl (G,\rho)$, then $\dimconfover(G,\rho) \leq d$.
\end{theorem}
\pref{thm:Rd-packings} is proved in somewhat greater generality:
One can replace $\R^d$ by any Ahlfors $d$-regular metric measure space
and relax the notion of ``packing'' to allow bounded-multiplicity
overlap of balls.
We refer to \cite{Lee17b}
for details.

      For a locally finite, connected graph $G$, denote the discrete-time
      heat kernel
      \[p^G_T(x,y) \seteq \Pr[X_T=y \mid X_0=x]\,, \qquad x,y \in V(G)\,,\]
      where $\{X_n\}$ is the standard random walk on $G$.
      We recall the {\em spectral dimension of $G$:}
      \[
         \dimspec(G) \seteq \lim_{n \to \infty} \frac{-2 \log p^G_{2n}(x,x)}{\log n}\,,
      \]
      whenever the limit exists.  If the limit does exist, then it is the same for all $x \in V(G)$.

      The spectral dimension is considered an important quantity
      in the study of quantum gravity, since it can be defined in a
      reparameterization-invariant way \cite{ANRBW98,AAJW98}.
      It has long been conjectured that the spectral dimension of 2D quantum gravity is equal to two.
      Our results confirm that the almost sure spectral dimension is at most two for these models.

      Define also the {\em upper and lower spectral dimension of $G$}, respectively:
      \begin{align*}
         \dimspecover(G) &\seteq \limsup_{n \to \infty}\frac{-2 \log p^G_{2n}(x,x)}{\log n}\,,\\
         \dimspecunder(G) &\seteq \liminf_{n \to \infty}\frac{-2 \log p^G_{2n}(x,x)}{\log n}\,.
      \end{align*}
      It turns out that conformal growth exponent bounds the spectral dimension 
      in somewhat general settings.

      Say that a real-valued random variable $X$ has {\em negligible tails}
      if its tails decay faster than any inverse polynomial:
      \begin{equation}\label{eq:neg-tails-def}
         \lim_{n \to \infty} \frac{\log n}{|\!\log \Pr[|X| > n]|} = 0\,,
      \end{equation}
      where we take $\log(0)=-\infty$ in the preceding definition
      (in the case that $X$ is essentially bounded).

      For the sake of clarity in the next statement, we use $(G,\rho)$ to denote the law $\mu$ of $(G,\rho)$,
      and $(\bm{G},\bm{\rho})$ to denote the random variable with law $\mu$.

      \begin{theorem}\label{thm:spec-conf}
         If $(G,\rho)$ is a unimodular random graph and $\deg_G(\rho)$ has negligible tails, then 
         almost surely:
         \begin{align*}
            \dimspecover(\bm{G}) &\leq \dimconfover(G,\rho)\,, \\
            \dimspecunder(\bm{G}) &\leq \dimconfunder(G,\rho)\,.
         \end{align*}
      \end{theorem}

      In conjunction with \pref{thm:Rd-packings}, this shows that if $(G,\rho)$ is the distributional
      limit of finite $\R^d$-packable graphs (and $\deg_G(\rho)$ has negligible tails), then
      almost surely:
      \[p_{2T}^G(\rho,\rho) \geq T^{-d/2-o(1)} \quad \textrm{as} \quad T \to \infty\,.\]
      In particular, this establishes that the vast majority of random planar maps
      in the literature have almost sure spectral dimension at most two.
      This was unknown even for UIPT and UIPQ.  For UIPQ, it was recently
      established by Gwynne and Miller \cite{GM17} that this is tight:
      the spectral dimension is almost surely two.  For UIPT, a matching lower
      bound remains open.

      An overview of the proof of \pref{thm:spec-conf} in the
      special case of $2$-dimensional growth is given in \pref{sec:spectral-measure}.
      The inequalities in \pref{thm:spec-conf} imply that the conformal growth
      rate provides a lower bound on the
      return probabilities up to a $T^{o(1)}$ correction factor.
      We remark that if one makes the stronger assumption that
      $\deg_G(\rho)$ has exponential tails, then the implied correction
      factors are only polylogarithmic; see the discussion in \pref{sec:conformals}.

      Finally, we note that the inequalities in \pref{thm:spec-conf} cannot
      be reversed in general: There is a unimodular random graph $(G,\rho)$ with
      uniformly bounded degrees such that $\dimspec(G,\rho)$ is finite but
      $\dimconf(G,\rho)$ is infinite; see \pref{sec:homo} where
      we review an example due to \cite{AHNR18}.

      \medskip
      \noindent
      {\bf Uniformization and intrinsic dimension.}
         Certainly circle packings of planar graphs are a powerful, elegant, and
         ``conformally natural'' \cite{Rohde11} tool.
         Still, it is enlightening to think of situations where ambient representations
         do a poor job of emphasizing the intrinsic geometry of the underlying graph.
         In general, this is the case when the dimension of the graph differs
         from that of the ambient space.

         A basic example is the planar graph $G=(V,E)$ which is the product of a triangle
         and a bi-infinite path:  $V=\{0,1,2\} \times \Z$ and $\{x,y\} \in E$ 
         if and only if $\|x-y\|_{1} = 1$.
         This graph is quasi-isometric to $\Z$, and thus manifestly one-dimensional.
         The appropriate uniformizing discrete conformal metric (by transitivity) is $\omega \equiv \1$.
         The circle packing in $\R^2$ (which is unique up to M\"obius transformations)
         has an accumulation point in $\R^2$, and the radii of the
         circles grow with geometrically increasing radii
         from the accumulation point to infinity.

         Consider another example:
         the incipient infinite cluster (IIC) of critical percolation $(G_d^{\mathrm{IIC}}, 0)$
         on $\Z^d$, for $d$ sufficiently large.  In their solution to the Alexander-Orbach conjecture in high
         dimensions, Kozma and Nachmias \cite{KN09} show that for $d \geq 11$\footnote{One needs
         to use \cite{FH17} to obtain $d \geq 11$; the original reference proves it for $d \geq 19$.}, almost surely
         $\dimspec\left(G_d^{\mathrm{IIC}}\right) = 4/3$.
         \pref{thm:spec-conf} implies
         that $\dimconf(G_d^{\mathrm{IIC}}) \geq 4/3$.
         Moreover, one can show that this is tight.
         For instance,
         $G_d^{\mathrm{IIC}}$ is spectrally homogeneous in the sense of \eqref{eq:homo1} in
         \pref{sec:homo} with $d=4/3$
         \cite{AsafComm2017}.
         In this case, the inequality in \pref{thm:spec-conf} can be reversed,
         and $\dimconf(G_d^{\mathrm{IIC}})=4/3$.

\subsection{Dimension two:  Gauged quadratic growth and recurrence}

   The conformal growth exponent is not precise enough to study recurrence (which depends on
   lower-order factors in the heat kernel $p_T^G(\rho,\rho)$).
Say that a unimodular random graph $(G,\rho)$ is {\em $(C,R)$-quadratic} for $C > 0$ and $R \geq 1$ if
\begin{equation}
   \inf_{\omega} \|\#B_{\omega}(\rho,R)\|_{L^\infty} \leq C R^2\,,
\end{equation}
where the infimum is over all normalized conformal metrics on $(G,\rho)$.
We say that $(G,\rho)$ has {\em gauged quadratic conformal growth (gQCG)} if there is a constant $C > 0$
such that $(G,\rho)$ is $(C,R)$-quadratic for all $R \geq 1$.
Note that we allow a different conformal weight $\omega$ for every choice of $R$,
and this is necessary for distributional limits of finite planar graphs to have gQCG (see \pref{lem:cbt}).

\begin{theorem}\label{thm:qcg1}
      If $(G,\rho)$ is a unimodular random graph with uniformly bounded degrees and gauged quadratic
      conformal growth, then $G$ is almost surely recurrent.
   \end{theorem}
   In \pref{sec:uniform-gQCG}, we argue that distributional limits of planar graphs, $H$-minor free
   graphs, string graphs, and other families have gauged quadratic conformal growth.
   Thus \pref{thm:qcg1} generalizes the Benjamini-Schramm Recurrence Theorem \cite{BS01} to $H$-minor-free graphs,
   confirming a conjecture stated there.  After initial dissemination of a draft
   of this manuscript,
   we learned that Angel and Szegedy (personal communication) had previously discovered a
   proof of the $H$-minor-free case using a detailed analysis of the Robertson-Seymour
   classification \cite{RS04}.\footnote{We remark that establishing quadratic conformal
   growth for $H$-minor-free graphs does not require the Robertson-Seymour theory.}
   \begin{remark}[String graphs]
      By the Koebe-Andreev-Thurston circle packing theorem, planar graphs are precisely the tangency graphs 
      of interior-disjoint disks in the plane.  {\em String graphs} are a significant generalization: They
      are the interesction graphs of a collection of arbitrary path-connected regions in the plane
      (with no assumption on disjointness).
      Such graphs can be dense, but string graphs with uniformly bounded degrees have quadratic
      conformal growth (see \pref{sec:uniform-gQCG}).
   \end{remark}

   \noindent
   {\bf Unbounded degrees.}
   Recently Gurel-Gurevich and Nachmias \cite{GN13} resolved a central open
   problem by showing that the uniform infinite planar triangulation (UIPT)
   and quadrangulation (UIPQ) are almost surely recurrent.
   They achieved this by extending the Recurrence Theorem of Benjamini and Schramm
   in a different direction:  In every distributional limit of finite
   planar graphs where the degree of the root has exponential tails,
   the limit is almost surely recurrent.
   It was previously known that both UIPT and UIPQ satisfy this hypothesis.

   Let $\mu$ denote the law of $(G,\rho)$, and define
   $\bar{d}_{\mu} : [0,1] \to \R_+$ by
   \[
   \bar{d}_{\mu}(\e) \seteq \sup \left\{ \E\left[\deg_G(\rho) \mid \cE\right] : \Pr(\cE) \geq \e \right\}\,,
   \]
   where the supremum is over all measurable sets $\cE$ with $\Pr(\cE) \geq \e$.
   That $\deg_G(\rho)$ has exponential tails is equivalent to
   $\bar{d}_{\mu}(1/t) \leq O(\log t)$ as $t \to \infty$.
   Negligible tails as defined in \eqref{eq:neg-tails-def}
   is equivalent to the assumption that
   \begin{equation}\label{eq:negligible-tails}
      \bar{d}_{\mu}(1/t) \leq t^{o(1)} \quad \textrm{as} \quad t \to \infty\,.
   \end{equation}

   \begin{assumption}\label{assume:nice}
      Suppose $(G,\rho)$ is a unimodular random graph with law $\mu$ satisfying the following:
      \begin{enumerate}
         \item $(G,\rho)$ has gauged quadratic conformal growth.
         \item $(G,\rho)$ is uniformly decomposable (cf. \pref{sec:metric-spaces}).
         \item $\E[\deg_G(\rho)^2] < \infty$.
      \end{enumerate}
   \end{assumption}

   \begin{theorem}\label{thm:qcg2}
      Under \pref{assume:nice}, if additionally
            \[ 
               \sum_{t\geq 1} \frac{1}{t \bar{d}_{\mu}(1/t)} = \infty\,,
            \]
      then $G$ is almost surely recurrent.
   \end{theorem}

   It was previously known (see \pref{sec:metric-spaces})
   that many families of finite graphs---planar graphs, $H$-minor-free graphs, and string graphs---are 
   uniformly decomposable.  This property passes to distributional limits, hence \pref{thm:qcg2}
   generalizes the result of \cite{GN13}.
   Note that we allow slightly heavier tails:
   For instance, $\bar{d}_{\mu}(1/t) \leq O(\log t \log \log t)$ is still enough
   to guarantee recurrence.

   Moreover, \pref{thm:qcg2} is tight in the following sense:
   For any monotonically non-decreasing sequence $\{d_t : t=1,2,\ldots\}$
   such that $\sum_{t \geq 1} \frac{1}{t d_t} < \infty$,
   there is a unimodular random planar graph satisfying \pref{assume:nice}
   that is almost surely transient, and such that $\bar{d}_{\mu}(1/t) \leq d_t$ for all $t$
   sufficiently large;
   see \pref{sec:transient-example}.

   \subsection{Estimates on the spectral measure and the heat kernel}
   \label{sec:spectral-measure}

   Let us now describe some of the elements of the proof of \pref{thm:qcg2},
   along with more detailed information about the random walk.
   In \pref{sec:amenable}, we argue that $\dimconfunder(G,\rho) < \infty$ implies
   that $(G,\rho)$ is invariantly amenable, and thus it is a distributional
   limit of finite graphs: $\{G_n\} \todl (G,\rho)$.

   Thus for simplicity, let us consider a finite planar graph $G_n$ and
   a root $\rho_n \in V(G_n)$ chosen uniformly at random.  Without loss,
   we may assume that $n=|V(G_n)|$.
   Define
   \[
      \Delta_{G_n}(k) \seteq \max_{S \subseteq V(G_n) : |S|\leq k} \sum_{x \in S} \deg_{G_n}(x)
   \]
   to be the sum of the $k$ largest vertex degrees in $G_n$.
   In \pref{sec:return-times}, we establish the bound
   \begin{equation}\label{eq:ev-bounds}
      \lambda_k(G_n) \leq c \frac{\Delta_{G_n}(k)}{n}\,,
   \end{equation}
   where $c$ is a universal constant and $\{1-\lambda_k(G_n) : k=0,1,\ldots,n-1\}$ are the eigenvalues
   of the random walk operator on $G_n$.  In \cite{KLPT09} a weaker bound was proved, with $k \cdot \Delta_{G_n}(1)$ in
   place of $\Delta_{G_n}(k)$.

   Such a bound provides average estimates for the diagonal of the heat kernel:
   Let $P$ denote the random walk operator on $G_n$.
   Then for an integer $T \geq 0$,
   \begin{equation}\label{eq:annealed}
      \E[p_T^{G_n}(\rho_n,\rho_n)] = \frac{1}{n} \sum_{x \in V(G_n)} \langle \1_x, P^T \1_x\rangle = \frac{\tr(P^T)}{n} = \frac{1}{n} \sum_{k=0}^{n-1} (1-\lambda_k(G_n))^T
      \geq \frac{\# \left\{ k : \lambda_k(G_n) \leq 1/T \right\}}{4n}\,,
   \end{equation}
   where the last inequality holds for $T \geq 2$.

   Thus if the vertex degrees are uniformly bounded along the sequence $\{G_n\}$, then
   we have $\Delta_{G_n}(k) \leq O(k)$, and combining 
   \eqref{eq:ev-bounds} and \eqref{eq:annealed} yields
   \[
      \E\left[p^G_T(\rho,\rho)\right] \geq \liminf_{n \to \infty} \E\left[p^{G_n}_T(\rho_n,\rho_n)\right] \gtrsim \frac{1}{T}\,.
   \]
   If $\{G_n\} \todl (G,\rho)$ and we impose only the weaker assumption that $\deg_{G}(\rho)$ has exponential tails, then it must hold that for $n$ sufficiently large,
   $\Delta_{G_n}(\frac{n}{T}) \leq O(\frac{n}{T} \log T)$, and one obtains
   \begin{equation}\label{eq:TlogT}
      \E\left[p^G_T(\rho,\rho)\right] \geq \liminf_{n \to \infty} \E\left[p^{G_n}_T(\rho_n,\rho_n)\right] \gtrsim \frac{1}{T \log T}\,.
   \end{equation}

   In this way, the degree-modified two-dimensional Weyl law in \eqref{eq:ev-bounds}
   predicts recurrence when the degree of the root has exponential tails, since $\sum_{T \geq 1} \frac{1}{T\log T} = \infty$.
   But a significant obstacle is that the annealed estimate \eqref{eq:TlogT} does not necessarily
   imply anything for the distributional limit.
       The issue is that the lower bound in \eqref{eq:TlogT} could come entirely from
       a small set of vertices (and such small sets could be negligible in the distributional limit).
   Indeed, one could add to $G_n$ only $\e n$ isolated vertices to achieve 
   $\frac{1}{n} \sum_{x \in V(G_n)} p^{G_n}_T(x,x) \geq \e$.
   (See \pref{sec:homo} for a family of connected examples where the
   contribution of the expected return probabilities come from a small measure of roots.)
   Thus even to obtain almost sure recurrence, we need an estimate stronger than \eqref{eq:annealed}.
   We state now the following strengthening of \pref{thm:qcg2}.

   \begin{theorem}\label{thm:heat-kernel}
      Under \pref{assume:nice}, the following holds.
      There is a constant $C=C(\mu)$ such that
      for every $\delta > 0$ and all $T \geq C/\delta^{2}$,
      \[
         \Pr\left[p_{2T}^G(\rho,\rho) < \frac{\delta}{T \bar{d}_{\mu}(1/T^3)}\right] \leq C \delta^{1/17}\,.
      \]
   \end{theorem}

   Note that even for the special case of UIPT/UIPQ, this estimate yields $\E p_{2T}^G(\rho,\rho) \gtrsim \frac{1}{T \log T}$ for every $T \geq 2$, improving over the bound $\E p_{2T}^G(\rho,\rho) \gtrsim \frac{1}{T^{4/3} (\log T)^{O(1)}}$ from \cite{BC13}.
      As a consequence, one obtains a bound on the rate of divergence of the Green function:
    Define
   \begin{align*} 
      g_{\mu}(T) &\seteq \sum_{t=1}^T \frac{1}{t \bar{d}_{\mu}(1/t)}\,.
   \end{align*}
    For instance, for UIPT/UIPQ, one has $g_{\mu}(T) \asymp \log \log T$.

   \begin{theorem}\label{thm:green-diverge}
      Under \pref{assume:nice}, the following holds.
      If $g_{\mu}(T) \to \infty$, then $G$ is almost surely recurrent.
      Moreover, almost surely:
      \[
         \limsup_{T \to \infty} \dfrac{\sum_{t=1}^T p^G_t(\rho,\rho)}{g_{\mu}(T)} > 0\,.
      \]
   \end{theorem}

\medskip
\noindent
{\bf Exhaustion by capacitors.}
   For a finite graph $G$ and $\f : V(G) \to \R$, define the normalized Dirichlet energy
   \begin{equation*}
      \cE_{G}(\f) \seteq \frac{1}{|E(G)|} \sum_{\{x,y\} \in E(G)} |\f(x)-\f(y)|^2\,.
   \end{equation*}

   Let us call a pair $(A,\Omega)$ of subsets $A \subseteq \Omega \subseteq V(G)$ a {\em capacitor,}
   and define the {\em capacity of $(A,\Omega)$} by
   \[
      \capacity_{\Omega}(A) \seteq \inf_{\f : V(G) \to [0,1]} \left\{ \cE_{G}(\f) : \f|_{A} \equiv 1, \supp \f \subseteq \Omega \right\}\,,
   \]
   where $\supp \varphi \seteq \{ x \in V(G) : \varphi(x) \neq 0 \}$.
   We remark that, by the Dirichlet principle, the capacity of $(A,\Omega)$ is precisely
   the inverse of the effective resistance $\reff^G(A \leftrightarrow (V(G) \setminus \Omega))$;
   see \pref{sec:resistance}.

   \pref{thm:spec-conf}, \pref{thm:heat-kernel}, and \pref{thm:green-diverge}
   are proved as follows.
	One uses an appropriate conformal weight on a finite graph $G$
   to locate capacitors $(A_1,\Omega_1), \ldots, (A_k,\Omega_k)$
	such that the sets $\{\Omega_i\}$ are pairwise disjoint.
   Using reversibility of the random walk,
   such a collection of capacitors yields a lower bound on typical return probabilities
	(see \pref{thm:bump-return}):
	For every $T \geq 1$ and $\e > 0$,
\[
	\pi\left(\left\{x \in V(G) : p^G_{2T}(x,x) \geq \frac{\e^2}{4 M}\right\}\right) \geq - 3\e \bar{d}_G(\e) +
		\sum_{i=1}^k \pi(A_i) - 2T \sum_{i=1}^k \capacity_{\Omega_i}(A_i)\,,
\]
where $M = \max \{ |\Omega_i| : i=1,\ldots,k\}$, and $\bar{d}_G(\e) \seteq \frac{\Delta_G(\e n)}{\e n}$
(this is equal to $\bar{d}_{\mu}(\e)$ when $\mu$ is the law of $(G,\rho)$ with $G$ a fixed
finite graph and $\rho \in V(G)$ chosen uniformly at random), and
$\pi$ is the stationary measure on $G$.

\subsection{Intrinsic volume growth, Markov type, and subdiffusivity}

Consider a connected, infinite, locally finite graph $G$.
For $x \in V(G)$ and $r \geq 0$, let
\[
   B_G(x,r) = \left\{ y \in V(G) : \dist_G(x,y) \leq r \right\}\,,
\]
where $\dist_G$ denotes the (unweighted) graph distance in $G$.
Suppose that $G$ has nearly uniform $d$-dimensional
volume growth in the sense that
for $r$ sufficiently large,
\begin{equation}\label{eq:intrinsic-growth}
   r^{d-o(1)} \leq |B_G(x,r)| \leq r^{d+o(1)}
\end{equation}
holds uniformly for all $x \in V(G)$.

When $G$ is planar and $d > 2$, one suspects that the structure of
$G$ should be somewhat degenerate.  Indeed, Itai Benjamini has put forth
a number of conjectures to this effect.
For instance, in \cite{BP11} it is conjectured that
if $G$ is planar and \eqref{eq:intrinsic-growth} is
satisfied, then the random walk on $G$ should be {\em subdiffusive}
with the natural speed estimate:
\begin{equation}\label{eq:bp-conj}
   \E[\dist_G(X_0, X_T)] \leq T^{1/d + o(1)}\,.
\end{equation}
Recent examples show this to be false.
\begin{theorem}[{\cite[Thm. 1.3]{EL20}}]
   For every rational $d > 3$ and $\e > 0$, there is a constant $c(\e) > 0$ and
   a unimodular random planar graph $(G,\rho)$ such that $G$ almost surely has uniform polynomial growth of degree
   $d$, and
   \[
      \E\left[d_G(X_0,X_T) \mid X_0 = \rho\right] \geq c(\e) T^{1/\left(d-1+\e\right)},\qquad \forall T \geq 1.
   \]
\end{theorem}

Nevertheless, we we will see that \eqref{eq:bp-conj} holds if $d$ is replaced by $d-1$.
Subdiffusivity was confirmed specifically for UIPQ:  In \cite{BC13},
it is shown that
\begin{equation}\label{eq:UIPQ}
   \E[\dist_G(X_0,X_T) \mid X_0=\rho] \leq T^{1/3} (\log T)^{O(1)}\,.
\end{equation}
For UIPT \cite{Angel03} and UIPQ \cite{BC13},
an almost sure asymptotic variant of \eqref{eq:intrinsic-growth} is satisfied
with $d=4$.
Thus the estimate \eqref{eq:UIPQ}, while non-trivial,
does not meet the conjectured exponent of $1/4$.

In establishing \eqref{eq:UIPQ}, the authors undertook a detailed study of the
geometry of UIPQ.  We show that the phenomenon is somewhat more general:
To obtain subdiffusivity for a unimodular random planar graph,
one need only assume asymptotic $d$-dimensional volume growth for some $d > 3$.

\begin{theorem}\label{thm:subd-intro-1}
   Suppose that $(G,\rho)$ is a unimodular random planar graph and for some $d > 3$,
   there is a function $h : \R_+ \to \R_+$ with $h(r) \leq r^{o(1)}$ and such that almost surely,
   \begin{equation}\label{eq:uniform-growth-est}
      \frac{r^d}{h(r)} \leq |B_G(\rho,r)| \leq h(r) r^d \qquad \forall r \geq 1\,.
   \end{equation}
   Then the random walk on $(G,\rho)$ is strictly subdiffusive.  More specifically,
   \begin{equation}\label{eq:intro-subd}
      \E\left[\dist_G(X_0, X_T)^{d-1}\mid X_0=\rho\right] \leq T^{1+o(1)}\quad \textrm{as}\quad T \to \infty\,.
   \end{equation}
\end{theorem}

\begin{remark}[Sharpness of the assumption $d > 3$]
   \label{rem:asaf-omer}
   After initial circulation of this manuscript, Omer Angel and Asaf Nachmias
   constructed, for every $\e > 0$ sufficiently small, a unimodular random planar graph $(G,\rho)$
   with uniformly bounded degrees
   and such that almost surely the random walk is diffusive and yet
   \begin{equation}\label{eq:limit-growth}
      \lim_{r \to \infty} \frac{\log |B_G(\rho,r)|}{\log r} = 3-\e\,.
   \end{equation}
   This suggests that the assumption $d > 3$ in \pref{thm:subd-intro-1} may be tight,
   except possibly at the critical value $d=3$.
   Note that \pref{thm:subd-intro-1} does not apply directly because
   \eqref{eq:limit-growth} only entails 
   an almost sure asymptotic bound, while
   \eqref{eq:uniform-growth-est} demands a uniform bound over the choice of $(G,\rho)$.

   Their examples {\em do show} that the assumption $d > 3$ in \pref{thm:uipq-speed} below
   is sharp, as they satisfy Assumption (A) with $d=3-\e$ for every $q > 1$ and $\alpha > 0$.
\end{remark}

The key geometric fact about planar graphs used to prove \pref{thm:subd-intro-1} 
is due to Benjamini and Papasoglu \cite{BP11}.
They
show that in any planar graph satisfying \eqref{eq:intrinsic-growth},
for any $x \in V(G)$ and $r$ sufficiently large,
there is a set of vertices of size at most $r^{1+o(1)}$ that separates 
$B_G(x,r)$ from $V(G) \setminus B_G(x,2r)$
(see \pref{fig:annulus-separator}).
In fact, the existence of such separators together with \eqref{eq:uniform-growth-est}
is enough to obtain the estimate \eqref{eq:intro-subd},
without requiring that the graph be planar.

For a connected, locally-finite graph $G$, a node $x \in V(G)$, and
two radii $r' > r > 0$,
let $\kappa_G(x; r, r')$ denote the cardinality of the smallest set
\[
   U \subseteq B_G(x,r') \setminus B_G(x,r)
\]
such that $U$ separates $x$ and $V(G) \setminus B_G(x,r')$ in $G$.
The next lemma follows easily from the argument in \cite{BP11};
we simply note the quantitative dependence on the growth rate.

\begin{lemma}[\cite{BP11}]
   \label{lem:bp11}
   If $(G,\rho)$ is a unimodular random planar graph satisfying \eqref{eq:intrinsic-growth},
   then there is a constant $c > 1$ such that
   \[
      \kappa_G(\rho; r, 2 r) \leq c r \frac{h(4r)}{h(r/4)}, \qquad \forall r \geq 1\,.
   \]
\end{lemma}

Using \pref{lem:bp11}, the next theorem generalizes \pref{thm:subd-intro-1}.

\begin{theorem}\label{thm:subd-intro-3}
   Suppose that $(G,\rho)$ is a unimodular random graph and
   for some numbers $d \geq k+1 \geq 2$,
   there is a function
   $h : \R_+ \to \R_+$ satisfying $h(r) \leq r^{o(1)}$
   and such that almost surely \eqref{eq:uniform-growth-est} holds,
   and moreover,
   \begin{equation}\label{eq:sepsize}
      \kappa_G\left(\rho; r, h(r) r\right) \leq h(r) r^{k-1} \qquad \forall r \geq 1\,.
   \end{equation}
   Then one has the estimate:
   \[
      \E\left[\dist_G(X_0,X_T)^{d-k+1} \mid X_0=\rho\right] \leq T^{1+o(1)}\quad\textrm{as}\quad T \to \infty\,.
   \]
   In particular, when $d > k + 1$, the random walk is strictly subdiffusive.
\end{theorem}

One can view this result as saying that when a graph has $d$-dimensional volume
growth, but $k$-dimensional isoperimetry and $d > k + 1$, then
this discrepancy necessitates subdiffusivity.
We now elaborate on the proof of \pref{thm:subd-intro-3} in
the special case when $G$ is almost surely planar.
In particular, one will find in the argument below
the relatively simple construction of a useful discrete
conformal metric.

\medskip
\noindent
{\bf Markov type of normalized conformal metrics.}
A central tool is K. Ball's notion of Markov type \cite{Ball92}.

\begin{definition}[Markov type]
   \label{def:markov-type}
   A metric space $(X,d)$ is said to have {\em Markov type $p \in [1,\infty)$} if there is a constant $M > 0$ such that for every $n \in \mathbb N$, the following holds. For every reversible Markov chain $\{Z_t\}_{t=0}^{\infty}$ on $\{1,\ldots,n\}$, every mapping $f : \{1,\ldots,n\} \to X$, and every time $t \in \mathbb N$, \begin{equation}\label{eq:mtype} \E \left[d(f(Z_t), f(Z_0))^p\right] \leq M^p t \, \E \left[d(f(Z_0), f(Z_1))^p\right]\,, \end{equation} where $Z_0$ is distributed according to the stationary measure of the chain. One denotes by $M_p(X,d)$ the infimal constant $M$ such that the inequality holds. \end{definition}

   In \pref{sec:diffusive}, basic Markov type theory 
   is used to prove the following.

\begin{theorem}\label{thm:unimodular-markov-type-intro}
   Suppose that $(G,\rho)$ is a unimodular random graph that almost
   surely satisfies
   \[
      |B_G(\rho,r)| \leq C r^q \qquad \forall r \geq 1
   \]
   for some numbers $C,q\geq 1$.
   Then for any normalized conformal metric $\omega$ on $(G,\rho)$, the following holds:
   For any $q' \geq 1$ and $T \geq 2$,
   \[
      \E\left[T^{q'} \wedge \dist_{\omega}(X_0,X_T)^2 \mid X_0=\rho\right] \leq C'T (\log T)^2\,,
   \]
   where $C'=C'(C,q,q')$.
\end{theorem}
We remark that the truncation by $T^{q'}$ is a technical matter that will not play
a significant role in the application of \pref{thm:unimodular-markov-type-intro}
to speed in the graph metric.
We will compare $\dist_{\omega}(X_0,X_T)$ to $\dist_G(X_0,X_T)$ and,
since $\dist_G(X_0,X_T) \leq T$ always holds, the truncation will be irrelevant
after passing to the graph distance; see the argument below.

Thus in order to prove a subdiffusive estimate for the speed in $\dist_G$, 
it suffices to construct a normalized conformal metric on $(G,\rho)$
with a suitable relationship between $\dist_{\omega}$ and $\dist_G$.

\begin{figure}
   \begin{center}
   \subfigure[Separating two spheres in $G$]{ \includegraphics[width=5cm]{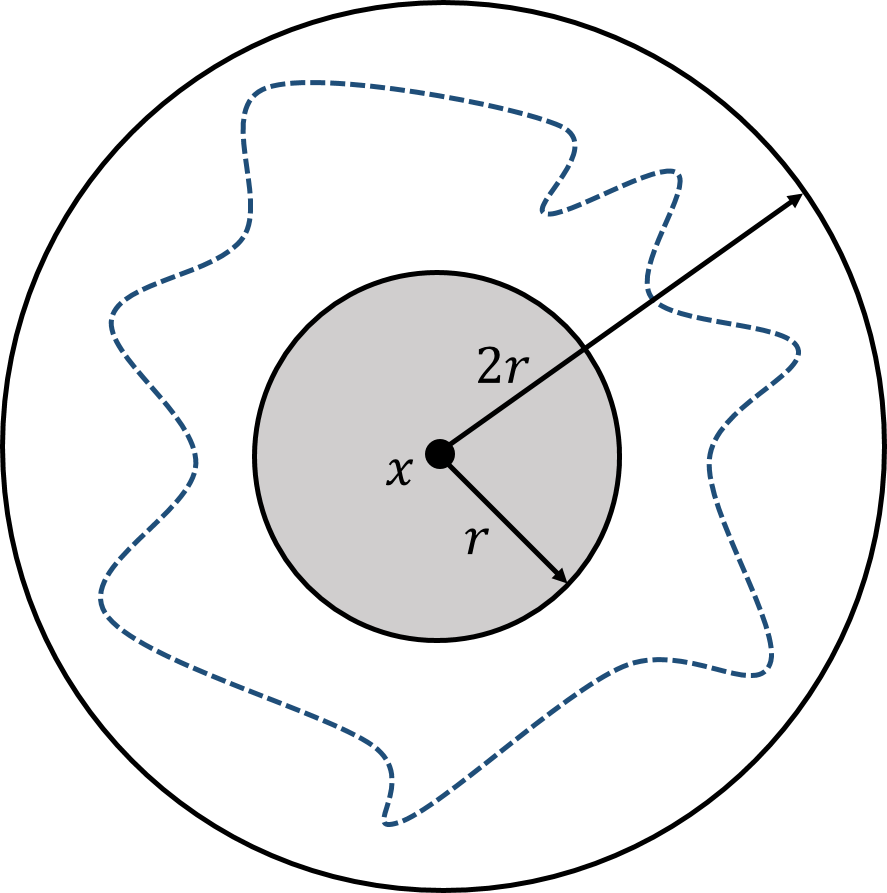}\label{fig:annulus-separator} } \hspace{0.6in}
\subfigure[Laying down a conformal barrier]{  \includegraphics[width=5.3cm]{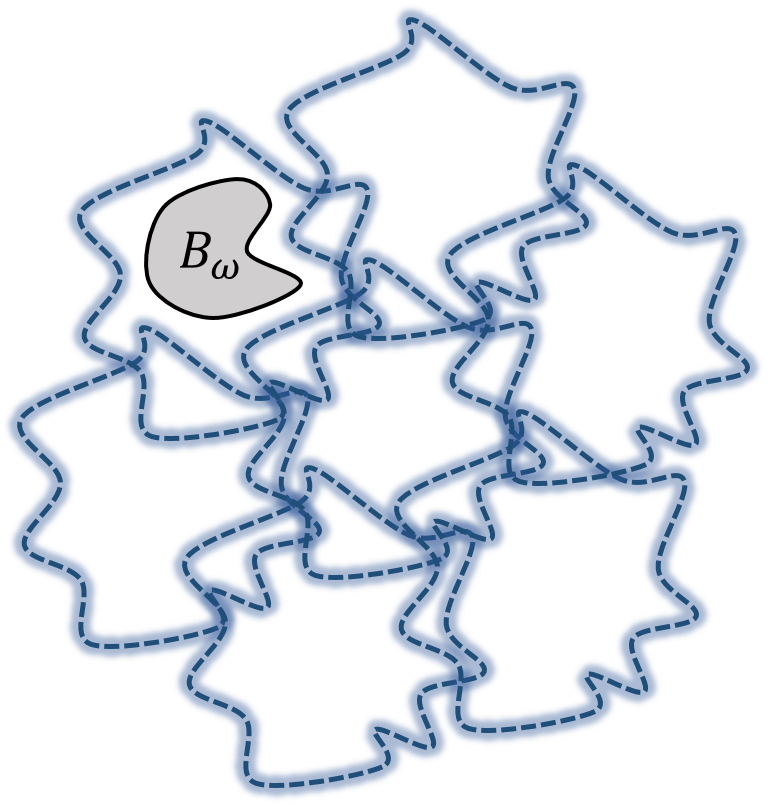}\label{fig:barrier} } \caption{Establishing subdiffusivity} \end{center}
\end{figure}

\medskip
\noindent
{\bf Weights from separators.}
First let us fix a radius $s \geq 1$.
By iteratively ``cutting out'' separators guaranteed by \eqref{eq:sepsize}
in a unimodular way
(see \pref{fig:barrier}), a generalization
of the following fact is established in \pref{sec:separators} (see
\pref{lem:separators-gen}).

\begin{lemma}\label{lem:barriers-intro}
   Suppose the assumptions of \pref{thm:subd-intro-3} hold.
   For every $s > 0$, there is a triple $(G,\rho,W_s)$ that
   is unimodular as a marked network, and such that the following holds:
   \begin{enumerate}
      \item $\Pr\left[\rho \in W_{s}\right] \leq s^{k-1-d+o(1)}$,
      \item Almost surely, every connected component of $G[V(G) \setminus W_{s}]$
         has diameter at most $s$ (in the metric $\dist_G$).
   \end{enumerate}
\end{lemma}
Heuristically, these parameters make sense:  The separator occupies
$s^{k-1+o(1)}$ nodes out of the $s^{d+o(1)}$ nodes in the $s$-ball.
Now we define, for every $j \in \N$, the normalized conformal weight:
\[
   \omega_j \seteq \frac{\1_{W_{2^j}}}{\sqrt{\Pr[\rho \in W_{2^j}]}}\,,
\]
and we note that from \pref{lem:barriers-intro}(1), we have
\begin{equation*}
   \omega_j \geq 2^{j(d-k+1)/2-o(1)} \1_{W_{2^j}}\,.
\end{equation*}
In particular, combined with \pref{lem:barriers-intro}(2), this implies that
\begin{equation}
   \label{eq:atleast}
   \dist_{G}(x,y) > 2^j \implies \dist_{\omega_j}(x,y) \geq 2^{j(d-k+1)/2-o(1)}\,.
\end{equation}

Finally, consider the normalized metric
\[
   \hat{\omega}_T \seteq \sqrt{\frac{1}{\lceil \log_2 T\rceil} \sum_{j=0}^{\lceil\log_2 T\rceil} \omega_j^2}
\]

From \eqref{eq:atleast}, we have almost surely for every $x,y \in V(G)$ with $\dist_G(x,y) \leq T$:
\[
   \dist_G(x,y) \leq \dist_{\hat{\omega}_T}(x,y)^{2/(d-k+1)} T^{o(1)} \quad\textrm{as}\quad T \to \infty\,.
\]
In particular, for every $T \geq 1$:
\begin{align*}
   \E\left[\dist_G(X_0, X_T)^{d-k+1}\right] &= 
   \E\left[T^{d-k+1} \wedge \dist_G(X_0, X_T)^{d-k+1}\right] \\
   &\leq T^{o(1)} \E\left[T^{d-k+1} \wedge \dist_{\hat{\omega}_T}(X_0,X_T)^{2}\right]\\
   &\leq T^{1+o(1)}\,,
\end{align*}
where the final inequality follows from \pref{thm:unimodular-markov-type-intro}.
This yields \pref{thm:subd-intro-3}.

\pref{thm:subd-intro-1} is not strong enough to reproduce \eqref{eq:UIPQ}
because even the vertex degrees in UIPQ are unbounded, and thus
no uniform estimate of the form \eqref{eq:uniform-growth-est} can hold.
The next result remedies this.
It shows that if we have a $(1+\delta)$-moment bound on the size of balls
and a stretched exponential lower tail, then our methods
can still be applied.

\begin{theorem}\label{thm:uipq-speed}
   Suppose that $(G,\rho)$ is a unimodular random graph that satisfies the following
   conditions for some numbers $C > 0$ and $d > 3$:
   \begin{enumerate}
      \item For every $r \geq 2$:  $\ \E |B_G(\rho,r)| \leq r^C$.
      \item The degree of the root has exponential tails:
         \[\Pr[\deg_G(\rho) > \lambda] \leq e^{-\lambda/C} \qquad \forall \lambda \geq 1\,.\]
   \end{enumerate}
    Assume, additionally, that one of the following two conditions holds for some $\alpha > 0$:
    \begin{enumerate}
       \item[(A)] {\bf Planar; volume statistics.}
          \begin{enumerate}
             \item[(i)] $G$ is almost surely planar, and
             \item[(ii)] For some $q > 1$
                and for every $r \geq 2$:
                \begin{align*}
                   \left(\E |B_G(\rho,6 r)|^q\right)^{1/q} &\leq C r^d\,, \\
                   \Pr\left[|B_G(\rho,r)| < \frac{\e}{C} r^d\right] &\leq \exp(-(1/\e)^{2/\alpha})
                \qquad \forall \e > 0\,.
               \end{align*}
         \end{enumerate}
      \item [(B)] {\bf Unimodular shattering.}
         For every $r \geq 2$, there is a random subset $W_r \subseteq V(G)$ such
         that $(G,\rho,W_r)$ is unimodular (as a marked network), and
         \begin{enumerate}
            \item [(i)] $\Pr[\rho \in W_r] \leq C r^{1-d} (\log r)^{\alpha}$
            \item [(ii)] Almost surely, every component of $G[V(G) \setminus W_r]$ has diameter
               at most $r$ in $\dist_G$.
         \end{enumerate}
   \end{enumerate}
   Then the random walk on $(G,\rho)$ is almost surely subdiffusive:
   There is a constant $C' \geq 1$ such that for all $T \geq 2$,
   \[
      \E\left[\dist_G(X_0,X_T)^{d-1} \mid X_0=\rho\right] \leq C' T (\log T)^{\alpha(d-1)+5}.
   \]
\end{theorem}

\medskip
\noindent
{\bf Application to UIPT and UIPQ.}
It is well-known that UIPT and UIPQ satisfy conditions (1) and (2).
See \cite{AS03} for UIPT and \cite{BC13} and for UIPQ.
The fact that (A) is satisfied with $\alpha=4$ for UIPT/UIPQ 
is somewhat more delicate, and is discussed in \pref{sec:appendix-curien}.
In \pref{sec:separators}, it is shown that 
(A) $\implies$ (B), but we include (A) to demonstrate
that for unimodular random planar graphs,
sufficient
volume statistics are enough to yield subdiffiusive behavior of the random walk.

\subsection{Discussion and related work}

Our use of ``conformal metric'' as a vertex-weighting $\omega : V(G) \to \R_+$
is inspired by the Riemannian setting and, in particular,
the classical argument of Hersch \cite{Hersch70}, which relies
on the fact that the Laplacian is conformally invariant in dimension two.
In this analogy, $\E[\omega(\rho)^2]$ plays the role of the area.
See, for instance, \pref{thm:bumps-easy}.
When the degrees are unbounded, the families of vertex- and edge-weighted metrics
on a graph $G$ can be substantially different (see \cite{Lee16}, where
this plays a fundamental role in the main theorem).

Suppose $G$ is a finite planar graph represented as the tangency graph of interior-disjoint
circles $\{C_v : v \in V(G) \}$ with radii $\{ r_v > 0 : v \in V(G) \}$.
Considering only the weight given by $\omega(v) \seteq r_v$, the topology of the Euclidean plane is removed;
we are left only with an ``area measure'' on $G$, and the topology of the graph $G$ itself.
While it may seem we have abstracted out too much, many ``conformally-flavored'' properties remain.

For instance, if $G$ is a triangulation with degree at most $d$, the classical ring lemma
asserts the existence of a constant $C=C(d)$ such that $r_u \leq C r_v$ for $\{u,v\} \in E(G)$.
In \pref{lem:regulate}, we show that for any normalized weight $\omega : V(G) \to \R_+$,
one can choose a related normalized weight $\hat{\omega}$ that satisfies an analogous property.
Similarly, in \pref{sec:mt-speed}, we see that if $(G,\rho,\omega)$ is a reversible conformal random
planar graph with a normalized conformal weight $\omega : V(G) \to \R_+$, then
\[
   \E\!\left[\dist_{\omega}(X_0, X_T)^2 \mid X_0 = \rho\right] \leq O(T) \E[\omega(\rho)^2], \qquad \forall T \geq 1\,.
\]
In other words, the random walk has at most diffusive speed in the metric $\dist_{\omega}$,
inheriting a property of Brownian motion in the Euclidean plane.

On the other hand, it does not hold in general that $(V(G), \dist_{\omega})$ is quasisymmetric to $(V(G), \dist_G)$ (see \cite{Heinonen01}
for background on quasisymmetric and quasiconformal maps), in the same way that if $\{C_v : v \in V(G)\}$ is a circle packing of $G$ in
the Euclidean plane and $x_v \in \R^2$ is the center of $C_v$, then 
the map $v \mapsto x_v$ is not necessarily a quasisymmetry.
Indeed, a circle packing of a planar graph can only be considered as a ``snapshot'' of a possible quasisymmetric mapping at
a particular scale.
This can be made precise by considering sequences of combinatorial approximations
of metric spaces homeomorphic to the Euclidean sphere,
as in the works of Cannon \cite{cannon94} and Bonk and Kleiner \cite{bk02}.

\medskip
\noindent
{\bf Anomalous diffusion on fractals.}
The topics studied here draw from a number of areas.
We are interested largely in examples
where various notions of dimension fail to coincide,
e.g., the toplogical dimension, exponent of volume growth (in
the intrinsic metric),
speed exponent of the random walk (cf. \eqref{eq:bp-conj}),
and the exponent of the isoperimetric profile (cf. \pref{thm:subd-intro-3}).
This phenomenon is characteristic of ``anomalous diffusion'' on fractals.
We refer to \cite{Barlow04} for a discussion
of such exponents and their possible relationships.

While there are conditions under which these exponents
are related in a natural way (see, e.g., \cite{BCK05}),
those settings generally involve a detailed relationship
between the graph distance and the effective resistance metric.
In contrast, we are only able to obtain such estimates
with respect to the conformal metric $\dist_{\omega}$
(and, even then, only for most choices of the root).
In this sense, the approach taken here is related to
that of \cite{ABGN16}, where the authors obtain
similar estimates in a planar graph
under the metric pulled back from a circle packing
(under the assumption of uniformly bounded degrees).

\medskip
\noindent
{\bf Geometric analysis on manifolds.}
The eigenvalue estimate \eqref{eq:ev-bounds} 
is closely related to one established much earlier
by Korevaar \cite{Korevaar93} in
addressing a question of S. T. Yau on the spectrum of the Laplace-Beltrami
operator on orientable surfaces.  There it is shown that if
$\Omega$ is a subdomain in a complete $d$-dimensional Riemannian
manifold $(M,g_0)$ with nonnegative Ricci curvature
and $(M,\varphi g_0)$ is a finite-volume conformal metric, then
\[
   \lambda_k \leq C_d \left(\frac{k}{\vol(\Omega, \f g_0)}\right)^{d/2} \qquad k=1,2,\ldots,
\]
where $C_d$ is a constant depending only on $d$, $\{\lambda_k\}$ are the Neumann eigenvalues
of the Laplacian, and $\vol(\cdot)$ is the Riemannian volume.
For background, we refer to the lecture notes \cite{SY94}.

An immediate consequence of the uniformization theorem is the estimate
\[
   \lambda_k \leq C \frac{k}{\vol(\vvmathbb{S}^2, g)}\qquad k=1,2,\ldots
\]
for any Riemannian metric $g$ on $\vvmathbb{S}^2$.

Korevaar's approach (see also \cite{GNY04} where it is
generalized and expounded upon at length) also proceeds
by the construction of a family of disjoint capacitors
(as we do \pref{sec:bumps}).  But at a technical level,
his argument is substantially different from the one taken here:
He constructs
capacitors in the metric measure space $(\vvmathbb{S}^2, d_{\vvmathbb{S}^2}, \mu)$,
where $d_{\vvmathbb{S}^2}$ is the standard (constant curvature)
geodesic metric on the sphere, and $\mu$
is the pushforward of the volume measure of $(\vvmathbb{S}^2,g)$
under a conformal map.  In particular, one knows little about $\mu$
except that it is non-atomic.
On the other hand, our conformal metric $\dist_{\omega}$
is chosen to interact nicely with the counting measure on the graph $G$.

\medskip
\noindent
{\bf Quasisymmetric uniformization.}
The conformal growth exponent is---at least philosophically---related
to Pansu's notion of {\em conformal dimension} \cite{Pansu89}.
For a metric space $X$, this is the infimal Hausdorff dimension
of all metric spaces that are quasisymmetrically equivalent to $X$.
A closely related notion occurs implicitly in a paper of Bourdon and Pajot \cite{BP03}
and is defined by explicitly by Bonk and Kleiner \cite{BK05}:  
The {\em Ahlfors regular conformal dimension of $X$} is the infimal
$d$ such that $X$ is quasisymmetric to an Ahlfors $d$-regular metric measure space.

The quasisymmetric uniformization problem asks when a metric space
can be quasisymmetrically parameterized by a class of model spaces.
We refer to the ICM lectures of Bonk \cite{BonkICM} and Kleiner \cite{KleinerICM}
for a survey of this area.

\subsection{Preliminaries}

We use the notation $\R_+ = [0,\infty)$ and $\Z_+ = \Z \cap \R_+$.
We also employ the asymptotic notations $A \lesssim B$ and $A \leq O(B)$
to denote that $A \leq c \cdot B$ where $c > 0$ is a positive
constant that is independent of other parameters.
Since the symbol $\lesssim$ appears only finitely many times
in this paper, one could in fact take $c > 0$ to be a fixed universal
constant.

We sometimes write $[n] = \{1,2,\ldots,n\}$.
When $X$ is a finite set and $f : X \to \R$, we use the notations:
\begin{align*}
   \|f\|_{\ell^2(X)} &\seteq \sqrt{\sum_{x \in X} f(x)^2}\,, \\
   \|f\|_{L^2(X)} &\seteq \sqrt{\frac{1}{|X|} \sum_{x \in X} f(x)^2}\,.
\end{align*}

All graphs appearing in this paper are undirected and locally finite
and without loops or multiple edges.
If $G$ is such a graph, we use $V(G)$ and $E(G)$ to denote the
vertex and edge set of $G$, respectively.
If $S \subseteq V(G)$, we use $G[S]$ for the induced subgraph on $S$.
For $A,B \subseteq V(G)$, we write $E_G(A,B)$ for the set of edges
with one endpoint in $A$ and the other in $B$.
We write $\dist_G$ for the unweighted path metric on $V(G)$, and
$B_G(x,r) = \{ y \in V(G) : \dist_G(x,y) \leq r \}$
to denote the closed $r$-ball around $x \in V(G)$.
Also let $\deg_G(x)$ denote the degree of a vertex $x \in V(G)$, and $\dmax(G) = \sup_{x \in V(G)} \deg_G(x)$.
Write $G_1 \cong G_2$ to denote that $G_1$ and $G_2$ are isomorphic as graphs.
If $(G_1,\rho_1)$ and $(G_2,\rho_2)$ are rooted graphs, we write $(G_1,\rho_1) \cong_{\rho}
   (G_2,\rho_2)$ to denote the existence of a rooted isomorphism.

Consider a pseudometric space $(X,d)$ (i.e., we allow for the possibility
that $d(x,y)=0$ when $x \neq y$).  Throughout the paper,
we will deal only with complete, separable, pseudometric spaces.
For $x \in X$ and two subsets $S,T \subseteq X$, we use the notations
$d(S,T) \seteq \inf_{x \in S, y \in T} d(x,y)$ and $d(x,S) \seteq d(\{x\},S)$.
Define $\diam(S,d) \seteq \sup_{x,y \in S} d(x,y)$ and for $R \geq 0$, define
the closed balls
\[
   B_{(X,d)}(x, R) = \{ y \in X : d(x,y) \leq R\}\,.
\]
We omit the subscript $(X,d)$ if the underlying metric space is clear from context.

   \paragraph{Graph minors and region intersection graphs}

   If $H$ and $G$ are finite graphs, one says that $H$ is {\em a minor of $G$}
   if $H$ can be obtained from $G$ by a sequence of edge deletions, vertex deletions,
   and edge contractions.  If $G$ is infinite, say that $H$ is a minor of $G$
   if there is a finite subgraph $G'$ of $G$ that contains an $H$ minor.
   Recall Kuratowski's theorem:  Planar graphs are precisely the graphs that do not contain $K_{3,3}$ or $K_5$ as a minor.

   A graph $G$ is a {\em region intersection graph over $G_0$} if the vertices of $G$ correspond to
   connected subsets of $G_0$ and there is an edge between two vertices of $G$ precisely
   when those subsets intersect. More formally, there is a family of connected subsets
   $\{R_u \subseteq V_0 : u \in V\}$ such that $\{u,v\} \in E \iff R_u \cap R_v \neq \emptyset$.
   We use $\rig(G_0)$ to denote the family of all {\em finite} region intersection
   graphs over $G_0$.
   
   A prototypical family of region intersection graphs is the set of {\em string graphs}; these are the
   intersection graphs of continuous arcs in the plane.  It is not difficult to see
   that $\rig(\Z^2)$ is precisely the family of all finite string graphs (see \cite[Lem. 1.4]{Lee16}).

\subsubsection{Unimodular random graphs and distributional limits}
\label{sec:unimodular}

We begin with a discussion of unimodular random graphs and distributional limits.
One may consult the extensive reference of Aldous and Lyons \cite{aldous-lyons}.
The paper \cite{BS01} 
offers a concise introduction to distributional limits of finite planar graphs.
We briefly review some relevant points.

Let $\graphs$ denote the set of isomorphism classes of connected, locally finite graphs;
let $\rgraphs$ denote the set of {\em rooted} isomorphism classes of {\em rooted}, connected,
locally finite graphs.
Define a metric on $\rgraphs$ as follows:  $\dloc\left((G_1,\rho_1), (G_2,\rho_2)\right) = 1/(1+\alpha)$,
where
\[
   \alpha = \sup \left\{ r > 0 : B_{G_1}(\rho_1, r) \cong_{\rho} B_{G_2}(\rho_2, r) \right\}\,.
\]
$(\rgraphs, \dloc)$ is a separable, complete metric space. For probability measures $\{\mu_n\}, \mu$ on
$\rgraphs$, 
write $\{\mu_n\} \Rightarrow \mu$ when $\mu_n$ converges weakly to $\mu$ with respect to $\dloc$.

\paragraph{The Mass-Transport Principle}
Let $\rrgraphs$ denote the set of doubly-rooted isomorphism classes of doubly-rooted, connected, locally finite graphs.
A probability measure $\mu$ on $\rgraphs$ is {\em unimodular} if it obeys the following
{\em Mass-Transport Principle:}  For all Borel-measurable $F : \rrgraphs \to [0,\infty]$,
\begin{equation}\label{eq:mtp}
   \int \sum_{x \in V(G)} F(G,\rho,x) \,d\mu((G,\rho)) = \int \sum_{x \in V(G)} F(G,x,\rho)\,d\mu((G,\rho))\,.
\end{equation}
If $(G,\rho)$ is a random rooted graph with law $\mu$, and $\mu$ is unimodular,
we say that $(G,\rho)$ is a {\em unimodular random graph}.

\paragraph{Distributional limits of finite graphs}
As observed by Benjamini and Schramm \cite{BS01},
distributional limits of finite graphs are unimodular random graphs.
Consider a (possibly random) sequence $\{G_n\} \subseteq \graphs$ of finite graphs,
and let $\rho_n$ denote a uniformly random element of $V(G_n)$.  Then $\{(G_n,\rho_n)\}$
is a sequence of $\rgraphs$-valued random variables,
and one has the following.

\begin{lemma}\label{lem:dl-unimodular}
   If $\{(G_n,\rho_n)\} \todl (G,\rho)$, then $(G,\rho)$ is unimodular.
\end{lemma}

If $\{(G_n,\rho_n)\} \todl (G,\rho)$,
we say that $(G,\rho)$ is the {\em distributional limit} of the sequence $\{(G_n,\rho_n)\}$.
When $\{G_n\}$ is a (possibly random) sequence of finite graphs, we write $\{G_n\} \todl (G,\rho)$ for $\{(G_n,\rho_n)\} \todl (G,\rho)$
   where $\rho_n \in V(G_n)$ is chosen uniformly at random.

\paragraph{Unimodular random conformal graphs}

A {\em conformal graph} is a pair $(G,\omega)$, where $G$ is a connected, locally finite graph and
$\omega : V(G) \to \R_+$.
Let $\cgraphs$ and $\crgraphs$ denote the collections of
isomorphism classes of conformal graphs and
conformal rooted graphs, respectively.
As in \pref{sec:unimodular}, one can define a metric on $\crgraphs$ as follows:
$\dloc^*\left((G_1,\omega_1, \rho_1), (G_2, \omega_2, \rho_2)\right) = 1/(\alpha+1)$,
where
\[
   \alpha = \sup \left\{ r > 0 : B_{G_1}(\rho_1,r) \cong_{\rho} B_{G_2}(\rho_2, r) \textrm{ and }
   \vvmathbb{d}\!\left(\omega_1|_{B_{G_1}(\rho_1,r)},\omega_2|_{B_{G_2}(\rho_2,r)}\right) \leq \frac{1}{r}\right\}\,,
\]
where for two weights $\omega_1 : V(H_1) \to \R_+$ and $\omega_2 : V(H_2) \to \R_+$
on rooted-isomorphic graphs $(H_1,\rho_1)$ and $(H_2,\rho_2)$, we write
\[
   \vvmathbb{d}\!\left(\omega_1, \omega_2\right) \seteq \inf_{\psi : V(H_1) \to V(H_2)} \left\|\omega_2 \circ \psi - \omega_1\right\|_{\ell^{\infty}}\,,
\]
where the infimum is over all graph isomorphisms from $H_1$ to $H_2$ satisfying $\psi(\rho_1)=\rho_2$.

If $\{\mu_n\}$ and $\mu$ are probability measures on $\crgraphs$, we abuse
notation and write $\{\mu_n\} \todl \mu$ to denote weak convergence with respect to $\dloc^*$.
One defines unimodularity of a random rooted conformal graph $(G,\omega,\rho)$ analogously to \eqref{eq:mtp}:
It should now hold that for all Borel-measurable $F : \crrgraphs \to [0,\infty]$,
\[
   \int \sum_{x \in V(G)} F(G,\omega,\rho,x) \,d\mu((G,\omega,\rho))= \int \sum_{x \in V(G)} F(G,\omega,x,\rho)\,d\mu((G,\omega,\rho))\,.
\]
Indeed, such decorated graphs are a special case of the marked networks
considered in \cite{aldous-lyons}, and again it holds that
every distributional limit of finite unimodular random conformal graphs is a unimodular random conformal graph.

Given a random conformal graph $(G,\omega,\rho)$, we define
\[
   \|\omega\|_{L^2} \seteq \sqrt{\E \omega(\rho)^2}\,.
\]
Say that $\omega$ is {\em normalized} if $\|\omega\|_{L^2} = 1$.

Suppose that $(G,\rho)$ is a unimodular random graph.  A {\em conformal weight on $(G,\rho)$}
is a unimodular random conformal graph $(G',\omega,\rho')$ such that $(G,\rho)$ and $(G',\rho')$
have the same law.  We will speak simply of a ``conformal metric $\omega$ on $(G,\rho)$.''
Only such unimodular metrics are considered in this work.

\medskip
\noindent
{\bf Convergence of infinite sums of metrics.}
The following construction will occasionally be useful.
Consider a family $\{\omega_j : j \geq 1\}$ of normalized conformal weights on
$(G,\rho)$.
Recall that formally this is a sequence $\{ (G_j,\omega_j,\rho_j) : j \geq 1\}$
such that $(G_j,\omega_j,\rho_j)$ is a unimodular random
conformal graph and $(G_j,\rho_j)$ has the same law as $(G,\rho)$ for every $j \geq 1$.
Thus we may consider a joint coupling $(G,\{\omega_j : j \geq 1\},\rho)$.

Now we would like to define a unimodular random conformal graph $(G,\omega,\rho)$ where
\[
   \omega \seteq \sqrt{\sum_{j \geq 1} \frac{\omega_j^2}{j^2}}\,.
\]
A priori, it is not clear that almost surely this sum converges for every $x \in V(G)$.
So let us momentarily allow $\omega : V(G) \to [0,+\infty]$ to take extended real values.
Fix $r \geq 1$ and define the transport $F(G,\omega,x,y) \seteq \1_{\{\dist_G(x,y) \leq r\}} \1_{\{\omega(y) =+\infty\}}$.
Then by the Mass-Transport Principle:
\begin{align*}
   \Pr\left[\max_{x \in B_G(\rho,r)} \omega(x)=+\infty\right] &= \E\left[\sum_{x \in V(G)} F(G,\omega,\rho,x)\right] \\
   &= \E\left[\sum_{x \in V(G)} F(G,\omega,x,\rho)\right] = \E[|B_G(\rho,r)| \1_{\{\omega(\rho)=+\infty\}}] = 0\,,
\end{align*}
where the latter equality follows from the fact that each $\omega_j$ is normalized, hence $\E[\omega(\rho)^2] < \infty$.
Since this holds for every $r \geq 1$,
we conclude that almost surely, $\sup_{x \in V(G)} \omega(x) < \infty$.

\section{Quadratic conformal growth and recurrence}
\label{sec:recurrence}

The assumption $\dimconf(G,\rho) \leq 2$ is not sufficient to ensure almost sure recurrence of $G$.
Instead, we need a more delicate way to measure quadratic growth using a family of metrics.
Recall that a unimodular random graph $(G,\rho)$ is {\em $(C,R)$-quadratic} for $C > 0$ and $R \geq 1$ if
\begin{equation}\label{eq:ball-growth}
   \inf_{\omega} \|\#B_{\omega}(\rho,R)\|_{L^\infty} \leq C R^2\,,
\end{equation}
where the infimum is over all normalized conformal metrics on $(G,\rho)$.
Say that $(G,\rho)$ has {\em gauged quadratic conformal growth (gQCG)} if there is a constant $C > 0$
such that $(G,\rho)$ is $(C,R)$-quadratic for all $R \geq 1$.
Note that we allow a different conformal weight $\omega$ for every choice of $R$.

A sequence $\{(G_n,\rho_n)\}$ of unimodular random graphs has {\em uniform gQCG}
if there is a constant $C > 0$ such that for
$(G_n,\rho_n)$ is $(C,R)$-quadratic for all $R \geq 1$ and $n \geq 1$.
A family $\cF \subseteq \cG$ of {\em finite} graphs has {\em uniform gQCG}
if the family of unimodular random graphs $\{(G,\rho): G \in \cF\}$ has uniform
gQCG, where $\rho \in V(G)$ is chosen uniformly at random.

Finally, we say that $(G,\rho)$ has {\em asymptotic gQCG} if
there is a constant $C > 0$ and
a sequence of radii $\{R_n\}$ with $R_n \to \infty$ such that $(G,\rho)$ is $(C,R_n)$-quadratic
for all $n \geq 1$.  We can now state the main theorem of this section.
The proof appears in \pref{sec:resistance}.

\begin{theorem}\label{thm:AgQCG-recurrent}
   If $(G,\rho)$ has asymptotic gQCG and $\|\deg_G(\rho)\|_{L^\infty} < \infty$,
   then $G$ is almost surely recurrent.
\end{theorem}

\begin{remark}
   Note that $(\Z^2,0)$ has gQCG, and moreover, one can consider a single conformal weight $\omega \equiv \1$ for every $R \geq 1$.
   On the other hand, the infinite ternary tree does not have gQCG.
\end{remark}

\begin{remark}\label{rem:weights}
   Note that if $(G,\rho)$ has gQCG, then $\dimconf(G,\rho) \leq 2$.
   To see this, consider the corresponding family of normalized conformal weights $\{\omega_{2^k} : k \geq 1\}$
   arising from applying \eqref{eq:ball-growth} with $R=2^k$,
   and let
   \[ 
      \omega \seteq \sqrt{\frac{6}{\pi^2} \sum_{k \geq 1} \frac{\omega^2_{2^k}}{k^2}}\,.
   \]
   Then $\|\omega\|_{L^2} = 1$, and moreover $\omega \geq \frac{\sqrt{6}}{k \pi} \omega_{2^k}$ for every $k \geq 1$, hence
   \[
      \left\|\#B_{\omega}\left(\rho,\tfrac{\sqrt{6}}{k \pi} 2^k\right)\right\|_{L^{\infty}} \leq 
     \left\|\#B_{\omega_{2^k}}(\rho, 2^k)\right\|_{L^{\infty}}  \leq C 4^k\,,
   \]
   implying that $\dimconf(G,\rho) \leq 2$.
\end{remark}

\begin{remark}
   In \pref{sec:uniform-gQCG}, we will show that the family of all finite planar graphs has uniform gQCG.
   But for this to hold, it must be that we allow $\omega=\omega_R$ to depend on the scale $R$ in \eqref{eq:ball-growth}.
   Indeed, let $T_n$ denote the complete binary tree of height $n$, then for some constant $c > 0$,
   and any normalized conformal metric $\omega : V(T_n) \to \R_+$, 
   \begin{equation}\label{eq:canopy}
      \max_{R \geq 0} \frac{|B_{\omega}(x,R)|}{R^2} \geq c \sqrt{n}\,.
   \end{equation}
   This is proved in \pref{lem:cbt}.
   Let $(T,\rho)$ denote the distributional limit of $\{T_n\}$ (this is known as the ``canopy tree'' from \cite{AW06}).
   Then \eqref{eq:canopy} implies that
   no single normalized conformal metric $\omega$ on $(T,\rho)$
   can have quadratic growth.
\end{remark}

\subsection{Comparing graph balls to conformal balls}
\label{sec:ring-lemma}

In order to use a conformal weight $\omega : V(G) \to \R_+$ to establish
recurrence, we will need a way of comparing
the conformal metric $\dist_{\omega}$ to the graph metric $\dist_G$.
Say that a conformal graph $(G,\omega)$ is {\em $C$-regulated} if
it satisfies the following properties: 
\begin{enumerate}
\item $\omega(x) \geq 1/2$ for all $x \in V(G)$.
\item If $\{u,v\} \in E(G)$, then $\omega(u) \leq C \,\omega(v)$.
\end{enumerate}
This definition allows us to compare balls in the metrics $\dist_G$ and $\dist_{\omega}$.

\begin{lemma}\label{lem:compare}
   If $(G,\omega)$ is $C$-regulated for some $C \geq 2$, then
it holds that for every $x \in V(G)$ and $r \geq 0$,
\[
B_G\left(x, \frac{\log \frac{r}{2{\omega}(x)}}{\log C}\right) \subseteq B_{{\omega}}(x,r)
\subseteq B_G(x, 2 r)\,.
\]
\end{lemma}

\begin{proof}
The latter inclusion is straightforward from property (1) of $C$-regulated.
The proof of the former inclusion is by induction.  Trivially, $B_G(x,0) \subseteq B_{{\omega}}(x,r)$.
Suppose that $B_G(x,k-1) \subseteq B_{{\omega}}(x,r)$ and $v \in B_G(x,k)$.
Property (2) of $C$-regulated yields
${\omega}(v) \leq {\omega}(x) C^{k}$, which implies inductively that
\[
\dist_{{\omega}}(x,v) \leq {\omega}(x) \sum_{j=0}^{k} C^j \leq 2 {\omega}(x) C^{k} \leq r\,,
\]
as long as $k \leq \frac{\log \frac{r}{2{\omega}(x)}}{\log C}$.
\end{proof}

Now let us see that when the degrees are uniformly bounded,
one can convert any conformal weight into a $C$-regulated weight
where $C$ is a constant depending only on the maximum degree.

\begin{lemma}
   \label{lem:regulate}
    Let $(G,\omega,\rho)$ be a normalized unimodular random conformal graph,
   and suppose that $d \seteq \|\deg_G(\rho)\|_{L^{\infty}}$.
Then there exists a normalized, $\sqrt{2d}$-regulated unimodular random conformal graph $(G,\hat{\omega},\rho)$ such that
   $\hat{\omega} \geq \frac{1}{2} \omega$.
\end{lemma}

\begin{proof}
   For a conformal pair $(G,\omega)$, we define
   \[
      \omega_0(x) =\sqrt{\sum_{y \in V(G)} \omega(y)^2 \,(2d)^{-\dist_G(x,y)}}\,.
   \]
   Notice that $\omega_0 \geq \omega$ pointwise, and moreover for $\{u,v\} \in E(G)$, we have
   \begin{equation}\label{eq:pre-ring}
      \omega_0(u)^2 \leq 2d \omega_0(v)^2
   \end{equation}
   by construction.

   In order to analyze $\|\omega_0\|_{L^2}$, we define a mass transportation:
   For $x,y \in V(G)$,
   \[
      F(G,\omega,x,y) = \omega(x)^2\, (2d)^{-\dist_G(x,y)}\,.
   \]
   Note that the total flow out of $x$ is bounded by
   \[
      \omega(x)^2 \sum_{y \in V(G)} (2d)^{-\dist_G(x,y)} \leq \omega(x)^2 \sum_{k \geq 0} 2^{-k} \leq 2 \omega(x)^2\,.
   \]
   Therefore by the Mass-Transport Principle, it holds that
   \begin{align*}
      2 \E\left[\omega(\rho)^2\right] &\geq
      \E\left[\sum_{x \in V(G)} F(G,\omega,\rho,x)\right] \\
      &=
      \E\left[\sum_{x \in V(G)} F(G,\omega,x,\rho)\right]  \\
      &= \E\left[\omega_0(\rho)^2\right]\,,
   \end{align*}
   where the last equality follows from the definition of $\omega_0$ and $F$.
   In particular, we conclude that $\E[\omega_0(\rho)^2] \leq 2 \E[\omega(\rho)^2] = 2$.

Now define the normalized weight
\[
   \hat{\omega} \seteq \frac{\sqrt{\frac{1}{4} \1+\frac{3}{8} \omega_0^2}}{\sqrt{\frac14 + \frac{3}{8} \E[\omega_0^2]}}\,.
\]
It satisfies property (1) of $C$-regulated by construction,
and also $\hat{\omega} \geq \frac{1}{2} \omega_0 \geq \frac{1}{2} \omega$ pointwise.
Furthermore, property (2) of $C$-regulated is a consequence of \eqref{eq:pre-ring} with $C=\sqrt{2d}$.
\end{proof}

\subsection{Bounding the effective resistance}
\label{sec:resistance}

In the present section, we will use the notion of the {\em effective resistance} $\reff^G(S \leftrightarrow T)$
between two subsets $S,T \subseteq V(G)$ in a graph.
For completeness, we present one definition that aligns with our use of the quantity;
for more background, we refer the reader to \cite[Ch. 2 \& 9]{LP:book}.
For $S,T \subseteq V(G)$ with $S \cap T = \emptyset$, the Dirichlet principle asserts that
\[
   \reff^G(S \leftrightarrow T) = \left(\inf_{f \in \cF_{S,T}} \cE(f)\right)^{-1},
\]
where $\cF_{S,T} = \{ f : V(G) \to \R \mid f|_S=0, f|_T=1 \}$, and
we recall the (unnormalized) Dirichlet energy functional
\[
	\cE(f) \seteq \sum_{\{x,y\} \in E(G)} |f(x)-f(y)|^2\,.
\]
Let us define $\reff^G(S \leftrightarrow T) = 0$ when $S \cap T \neq \emptyset$.
We will soon prove \pref{thm:AgQCG-recurrent} using the following well-known characterization;
see, e.g., \cite[Lem. 9.22]{LP:book}.

\begin{theorem}\label{thm:how-recurrence}
   A graph $G$ is recurrent if and only if there is some vertex $x \in V(G)$ and constant $c > 0$ such
   that for all $R \geq 0$, there is a finite set $S_R \subseteq V(G)$ such that
   \[
      \Reff^G(B(x,R) \leftrightarrow V(G) \setminus S_R) \geq c
   \]
\end{theorem}

First we will need a lemma about the expected area of balls.
Let us define
\[
   \ao(x,R) \seteq \sum_{y \in B_{\omega}(x,R)} \omega(y)^2\,.
\]

\begin{lemma}\label{lem:ball-area}
   Let $(G,\omega, \rho)$ be a unimodular random conformal graph with $\E \omega(\rho)^2=1$.
   Then for every $R \geq 1$,
   \[
      \E\left[\ao(\rho,R)\right] \leq \left\|\# B_{\omega}(\rho,R)\right\|_{L^\infty}\,.
   \]  
\end{lemma}

\begin{proof}
We employ the Mass-Transport Principle:
For a conformal graph $(G,\omega)$ and $x,y \in V(G)$, define the flow
\[
   F(G,\omega,x,y) = \omega(x)^2 \1_{\{\dist_{\omega}(x,y) \leq R\}}\,.
\]
Then,
\begin{align*}
   \E\left[\ao(\rho,R)\right] = \E\left[\sum_{x \in V(G)} F(G,\omega,x,\rho)\right] &=\E\left[\sum_{x \in V(G)} F(G,\omega,\rho,x)\right] \\
   &= \vphantom{\bigoplus} \E\left[\omega(\rho)^2 |B_{\omega}(\rho,R)|\right] \leq \left\|\# B_{\omega}(\rho,R)\right\|_{L^\infty}\,.\qedhere
\end{align*}
\end{proof}

In order to apply \pref{thm:how-recurrence}, we use a conformal weight to construct
a test function of small energy.

\begin{lemma}\label{lem:test-function}
   Consider a graph $G$, vertex $x \in V(G)$, and scale $R \geq 0$.
   Then for any $C$-regulated conformal weight $\omega : V(G) \to \R_+$, it holds that
   \[
      \Reff\left(B_G\left(x, \frac{\log \tfrac{R}{4 \omega(x)}}{\log C}\right) \leftrightarrow V(G) \setminus B_G(x,2R)\right) \geq 
      \frac{1}{4 (1+C)^2 \dmax(G)} \cdot      \frac{R^2}{\ao(x,R)}\,.
   \]
\end{lemma}

\begin{proof}
   Define $f : V(G) \to \R$ by
   \[
      f(x) = \frac{2}{R} \min\left(\frac{R}{2}, \max \left(0, \dist_{\omega}(x,v)-\frac{R}{2}\right)\right)\,.
   \]
   Note that $v \in B_{\omega}(x, \frac{R}{2}) \implies f(v)=0$ and $v \notin B_{\omega}(x, R) \implies f(v)=1$.

   Therefore by the Dirichlet principle,
\[
\reff\left(\vphantom{\bigoplus}
B_{\omega}(x,\tfrac{R}{2}) \leftrightarrow V(G) \setminus B_{\omega}(x, R)\right) \geq \frac{1}{\cE_G(f)}\,,
\]
and
\begin{align*}
   \cE(f) &\leq
\frac{4\dmax(G)}{R^2} (1+C)^2 \sum_{v \in B_{\omega}(x, R)} \omega(v)^2 \\
&=4\dmax(G)(1+C)^2  \, \frac{\area_{\omega}(x, R)}{R^2}\,,
\end{align*}
where we have used the fact that $f$ is $2/R$-Lipschitz, and
the fact that $\omega$ is $C$-regulated which, in particular, asserts that for $\{u,v\} \in E(G)$, one has
\[
\dist_{\omega}(u,v)^2 \leq \omega_k(u)^2 + \omega_k(v)^2 \leq (1+C)^2 \omega_k(v)^2\,.
\]
Finally, we use \pref{lem:compare} to arrive at the desired conclusion,
replacing $\dist_{\omega}$ balls by $\dist_G$ balls.
\end{proof}

\begin{proof}[Proof of \pref{thm:AgQCG-recurrent}]
   Consider a radius $R \geq 1$ and a normalized, $C$-regulated conformal weight $\omega : V(G) \to \R_+$ satisfying
   $\|\#B_{\omega}(\rho,R)\|_{L^\infty} \leq c R^2$ for some constant $c > 0$.
   From \pref{lem:ball-area}, we have $\E[\ao(\rho,R)] \leq c R^2$, hence employing Markov's inequality,
   \[
      \Pr\left[\omega(\rho)^2 < R \textrm{ and } \ao(\rho,R) < \frac{c}{\e} R^2  \right] \geq 1-\e-\tfrac{1}{R}\,.
   \]
   Combining this with \pref{lem:test-function}, we see that
   \begin{equation}\label{eq:eps-error}
      \Pr\left[\reff\left(B_G(\rho,\tfrac{\log (R/4)}{2 \log C}) \leftrightarrow V(G) \setminus B_G(\rho, 2R)\right) \geq c'\e \right] \geq 1-\e-\frac{1}{R}\,,
   \end{equation}
   where $c'$ is a constant depending only on $C$ and $\|\deg_G(\rho)\|_{L^\infty}$.

   By assumption, $(G,\rho)$ has asymptotic gQCG and $\|\deg_G(\rho)\|_{L^\infty} < \infty$.
   Combining the definition of asymptotic gQCG with \pref{lem:regulate} (to derive a $C$-regulated conformal metric)
   shows that \eqref{eq:eps-error} holds for $R=R_n$, where $\{R_n\}$ is a sequence
   of radii with $R_n \to \infty$.

   In particular, Fatou's Lemma tells us that
   \[
      \Pr\left[\limsup_{n \to \infty} \reff\left(B_G(\rho,\tfrac{\log (R_n/4)}{2 \log C}) \leftrightarrow V(G) \setminus B_G(\rho, 2R_n)\right) \geq c'\e \right] \geq 1-\e\,.
   \]

   \jnote{
      Explanation of Fatou's Lemma:

      Let $X_n$ be the indicator of whether $\reff\left(B_G(\rho,\tfrac{\log (R_n/4)}{2 \log C}) \leftrightarrow V(G) \setminus B_G(\rho, 2R_n)\right) \geq c'\e$.

      Then $\E[X_n] \geq 1-\e-\frac{1}{R}$.

      Therefore $\limsup_{n \to \infty} \E[X_n] \geq 1-\e$.

      Therefore $\E\left[\limsup_{n \to \infty} X_n\right] \geq 1-\e$.

      This implies that with probability at least $1-\e$, we get infinitely many good events.
   }

   Since $\{B_G(\rho,2R_n)\}$ is a sequence of finite sets,
   \pref{thm:how-recurrence} yields
   \[
      \Pr[\textrm{$G$ recurrent}] \geq 1-\e\,.
   \]
   Sending $\e \to 0$ completes the proof.
\end{proof}

\subsection{Region intersection graphs and energy-minimizing conformal weights}
\label{sec:uniform-gQCG}

The next two theorems essentially follow from prior work.
Note that for the special case of planar graphs, 
an alternate proof of the next result based on circle packings appears in
\cite{Lee17b}.  It has the advantage that it extends suitably
to graphs that are sphere-packed in any Euclidean space.

\begin{theorem}[Uniform gQCG for $H$-minor-free graphs \cite{KLPT09}]
\label{thm:quadgrowth}
For every fixed graph $H$, the family of finite graphs excluding $H$ as a minor has uniform gQCG.
In particular, if $H=K_h$ for some $h \geq 2$, then every such graph is $(\kappa,R)$-quadratic
for all $R \geq 1$, where $\kappa \leq O(h^2 \log h)$.
\end{theorem}

\begin{theorem}[Uniform gQCG for region intersection graphs \cite{Lee16}]
\label{thm:quadgrowth-rig}
For every $\lambda > 0$ and fixed graph $H$, the family of finite region intersection graphs $G$
over an $H$-minor-free graph with $\dmax(G) \leq \lambda$ has uniform gQCG.
In particular, if $H=K_h$ for some $h \geq 2$, then every such graph
is $(\kappa,R)$-quadratic for all $R \geq 1$, where $\kappa \leq O(\lambda h^2 \log h)$.
\end{theorem}

Let us remark on the proof of these results.
Fix a finite graph $G=(V,E)$.  Let $n=|V|$ and consider a positive number $k \leq n$.
Define the set $P_k(G) \subseteq \ell^2(V)$ by
\[
   P_k(G) \seteq \left\{ \omega \geq 0 : \frac{1}{k^2} \sum_{x,y \in S} \dist_{\omega}(x,y) \geq 1 \quad \forall S \subseteq V, |S| \geq k\right\}.
\]
While it may not be immediately apparently, the set $P_k(G)$ is a polytope because
one can replace every such inequality indexed by a subset $S \subseteq V$ with the family of inequalities:
\[
   \frac{1}{k^2} \sum_{x,y \in S} \len_{\omega}(\gamma_{xy}) \geq 1 \qquad \forall\{\gamma_{xy} : x,y \in S\}\,,
\]
where in the latter quantifer, $\gamma_{xy}$ ranges over all simple $x$-$y$ paths in $G$.
Since $\len_{\omega}(\gamma)$ is a linear function in the values $\{ \omega(z)  : z \in V\}$,
the claim follows.
\begin{lemma}\label{lem:size-basic}
   For any $z \in V$ and $\omega \in P_k(G)$, it holds that $|B_{\omega}(z,1/2)| < k$.
\end{lemma}
\begin{proof}
   Denote $S \seteq B_{\omega}(z,1/2)$.
   If $|S| \geq k$, let $S' \subseteq S$ be a subset with $|S'|=k$.
   Since $\omega \in P_k(G)$, we have
   \[
      \frac{1}{|S'|^2} \sum_{x,y \in S'} \dist_{\omega}(x,y) \geq 1,
   \]
   but this is a contradiction since $\diam_{\omega}(S') \leq 1$.
\end{proof}

Consider now the optimization problem
\begin{equation}\label{eq:thetak}
   \theta_k(G) \seteq \min \left\{ \|\omega\|_{L^2(V)} : \omega \in P_k(G) \right\}\,.
\end{equation}

We claim that the values $\{\theta_k(G) : k =1,2,\ldots,n\}$ dictate the quadratic conformal growth:
$G$ is $(\kappa,R)$-quadratic for some $\kappa \geq 1$ and every $R \geq 0$ if and only if 
$\theta_k(G) \leq C/k^2$ for some $C \geq 1$ and every $k \leq n$.

\begin{claim}
   For every $n$-vertex graph $G=(V,E)$, the following holds.
   \begin{enumerate}
         \item For every $k \leq n$, $G$ is $(\kappa,R)$-quadratic with $\kappa=4 k \theta_k(G)^2$ and $R=1/(2\theta_k(G))$.
         \item For every $\kappa, R > 0$, it holds that if $G$ is $(\kappa,R)$-quadratic, then
            \[
               \theta_k(G) \leq \frac{2}{R} \textrm{ for all } k \geq 2 \kappa R^2\,. 
            \]
   \end{enumerate}
\end{claim}

\begin{proof}
   Let us first prove (1).
   Consider $\omega \in P_k(G)$ with $\theta \seteq \|\omega\|_{L^2(V)}$ and define $\hat{\omega} \seteq \omega/\theta$.
   Then by \pref{lem:size-basic}, for any $z \in V$,
   \[
      |B_{\omega}(z,1/(2\theta))| = |B_{\hat{\omega}}(z, 1/2)| < k\,,
   \]
   implying that $G$ is $(4\theta^2 k,1/(2\theta))$-quadratic.

   To prove (2), consider $\omega : V \to \R_+$ such that $\|\omega\|_{L^2(V)}=1$ and $|B_{\omega}(x,R)| \leq \kappa R^2$ for
   all $x \in V$.
   Consider any $S \subseteq V$ with $|S| \geq 2 \kappa R^2$.
   Then for every $x \in S$, it holds that
   \[
      |B_{\omega}(x,R) \cap S| \leq \frac{1}{2} |S|\,.
   \]
   Hence:
   \begin{equation}\label{eq:sizelb}
      \frac{1}{|S|^2} \sum_{x,y \in S} \dist_{\omega}(x,y) \geq \frac{1}{|S|^2} \sum_{x \in S} \frac{|S|}{2} R \geq \frac{R}{2}\,.
   \end{equation}
   If we now set $\hat{\omega} \seteq (2/R) \omega$, then \eqref{eq:sizelb} gives $\hat{\omega} \in P_{k}(G)$ for
   any $k \geq 2 \kappa R^2$, hence $\theta_k(G) \leq 2/R$ for all such $k$.
\end{proof}

\pref{thm:quadgrowth} and \pref{thm:quadgrowth-rig} are proved by analyzing the optimization \eqref{eq:thetak}.
It entails minimizing a strongly convex function over a polytope, thus \eqref{eq:thetak} has a unique optimal solution.
The authors of \cite{BLR08} developed a ``flow crossing'' theory for understanding the dual optimization
problem, and that was expanded upon in the works \cite{KLPT09,Lee16}.
The following is a consequence of \cite[Thm. 1.13]{Lee17b}.

\begin{lemma}\label{lem:gQCG-limits}
If $\{(G_n,\rho_n)\} \todl (G,\rho)$ and $\{(G_n,\rho_n)\}$ has
uniform gQCG, then $(G,\rho)$ has gQCG.
In particular, if $(G,\rho)$ is a distributional limit of
finite $H$-minor-free graphs, then $(G,\rho)$ has gQCG.
\end{lemma}

In particular, combining \pref{lem:gQCG-limits} with the preceding two theorems and \pref{thm:AgQCG-recurrent} yields
the following corollary.
We recall that $\rig(\cF(H))$ is the set of all finite graphs
that are region intersection graphs over some graph $G_0$ 
that excludes the graph $H$ as a minor.

\begin{corollary}
   For every fixed graph $H$, if $\{G_n\} \subseteq \rig(\cF(H))$
   is a sequence of graphs with uniformly bounded degrees
   and $\{G_n\} \todl (G,\rho)$, then $G$ is almost surely recurrent.
\end{corollary}

We end this section by observing that one cannot equip every finite planar graph
with a single normalized metric that achieves uniform quadratic volume growth
at all scales simultaneously (thus justifying the necessity for multiple conformal weights
in the definition of gQCG).

\begin{lemma}\label{lem:cbt}
   There is a constant $C > 0$ such that the following holds.
   Let $T_n$ be the complete binary tree of height $n$, and consider any normalized conformal weight $\omega$ on $T_n$.
   Then
   \[
      \max_{x \in V(T_n)} \max_{R \geq 0} \frac{|B_{\omega}(x,R)|}{R^2} \geq C \sqrt{n}\,.
   \]
\end{lemma}

\begin{proof}
   Let $\omega : V(T_n) \to \R_+$
   be a conformal weight satisfying $|B_{\omega}(x,R)| \leq R^2$ for all $x \in V(T_n)$ and
   $R \geq 1$.
   Consider the family $\cP$ of ${2^n \choose 2}$ paths in $T_n$ between all the leaves of $T_n$.
   There must exist a
   constant $c > 0$ and a 
   subset $\cP_n \subseteq \cP$ of paths going through the root of $T_n$
   with $|\cP_n| \geq c 2^{2n}$ and such that every path in $\cP_n$
   has $\omega$-length at least $c 2^{n/2}$.  (Otherwise there would be some leaf that could reach
   $c 2^n$ other leaves using paths of length $\ll 2^{n/2}$, contradicting the quadratic volume assumption.)
   Similarly, there exist disjoint subsets $\cP_{n-1}^{(0)}, \cP_{n-1}^{(1)} \subseteq \cP$
   of paths in the left and right subtrees, each containing $c 2^{2(n-1)}$ paths of $\omega$-length
   at least $c 2^{(n-1)/2}$, and so on.

   Let $\cP_k = \bigcup_{j \in \{0,1\}^{n-k}} \cP_k^{(j)}$ be the set of such ``long'' paths
   in subtrees of height $k$.
   Observe that this union is disjoint by construction.
   For a vertex $v \in V(T_n)$, define
   \[
      \alpha(v) = \sum_{k=1}^n 2^{-3k/2} \# \{ \gamma \in \cP_k : v \in \gamma \}\,.
   \]
   Then we have
   \[
      \sum_{k=1}^n c^2 2^{-3k/2} 2^{n-k} 2^{2k} 2^{k/2}
      \leq  \sum_{k=1}^n 2^{-3k/2} |\cP_k| \min_{\gamma \in \cP_k} \len_{\omega}(\gamma)
      \leq \sum_{v \in V(T_n)} \alpha(v) \omega(v) \leq \|\alpha\|_{\ell^2(V(T_n))} \|\omega\|_{\ell^2(V(T_n))}\,,
   \]
   where the last inequality is Cauchy-Schwarz.  The left-hand side is $c^2 n 2^n$.

   Now a simple calculation yields:
   \[
      \sum_{v \in V(T_n)} \alpha(v)^2 \leq \sum_{k=1}^n 2^{n-k} 2^{4k} 2^{-3k} = n 2^n.
   \]
   We conclude that
   \[
      \|\omega\|_{\ell^2(V(T_n))} \geq c^2 \sqrt{n 2^n}\,,
   \]
   implying that $\|\omega\|_{L^2} = 2^{-n/2} \|\omega\|_{\ell^2(V(T_n))} \geq c^2 \sqrt{n}$,
   and completing the argument.
\end{proof}

\section{Return probabilities and spectral geometry on finite graphs}
\label{sec:return-times}

We now turn to heat kernel estimates on finite graphs.

\subsection{The normalized Laplacian spectrum}
\label{sec:sgt}

Let $G=(V,E)$ be a connected, finite graph with $n=|V|$.
Let $\pi(x) = \frac{\deg_G(x)}{2|E|}$ denote the stationary measure.
We will use $L^2(\pi)$ for the Hilbert space of functions $f : V \to \R$
equipped with the inner product
\[
\langle f,g\rangle_{\pi} = \sum_{x \in V} \pi(x) f(x) g(x)\,,
\]
and denote by
\[
\langle f,g\rangle = \sum_{x \in V} f(x) g(x)
\]
the inner product on $\ell^2(V)$.  We use $\|\cdot\| \seteq \|\cdot\|_{\ell^2(V)}$
and $\|f\|_{\pi} = \sqrt{\langle f,f\rangle_{\pi}}$.

Define the operators $A, D, P, L, \cL : \ell^2(V) \to \ell^2(V)$ as follows
\begin{align*}
A f(x) &\seteq \sum_{y : \{x,y\} \in E} f(y)\,, \\
D f(x) &\seteq \deg_G(x) f(x)\,, \\
P  &\seteq D^{-1} A\,, \\
L  &\seteq I-P\,, \\
\cL &\seteq I - D^{-1/2} A D^{-1/2}\,.
\end{align*}
The {\em normalized Laplacian} $\cL$ is symmetric and positive semi-definite.
We denote its eigenvalues by
\[
0 = \lambda_1(G) \leq \lambda_2(G) \leq \cdots \leq \lambda_{n-1}(G)\,.
\]
We use $\lambda_k \seteq \lambda_k(G)$ if the graph $G$ is clear from context.
Note that $P = I - D^{-1/2} \cL D^{1/2}$, so
if $\cL f = \lambda f$, then $P D^{-1/2} f = (1-\lambda) D^{-1/2} f$.
Thus the spectrum of $P$ is $\{1-\lambda_k(G) : k=0,1,\ldots,n-1\}$.

Define the {\em Rayleigh quotient $\cR_G(f)$} of non-zero $f \in L^2(\pi)$ by
\[
\cR_G(f) \seteq \frac{\langle D^{1/2}  f, \cL D^{1/2} f\rangle}{\langle D^{1/2} f,D^{1/2} f\rangle} =
\frac{\langle f, L f\rangle_{\pi}}{\|f\|_{\pi}^2}
=
\frac{\frac{1}{|E|} \sum_{\{x,y\} \in E} |f(x)-f(y)|^2}{\|f\|_{\pi}^2}\,.
\]

Recall also the variational formula for eigenvalues:
\begin{equation}\label{eq:variational}
\lambda_k(G) = \min_{U \subseteq L^2(\pi)}\ \max_{0 \neq f \in U} \cR_G(f)\,,
\end{equation}
where the minimum is over all subspaces $U \subseteq L^2(\pi)$ with $\dim(U)=k+1$.
The preceding fact has a useful corollary.

\begin{corollary}\label{cor:variational}
Suppose that $\psi_1, \ldots, \psi_r : V \to \R$ are disjointly supported
functions with $\cR_G(\psi_i) \leq \theta$ for $i=1,2,\ldots,r$.  Then,
\[
   \lambda_{r-1}(G) \leq 2 \theta\,.
\]
\end{corollary}

\begin{proof}
   Let $U = \mathrm{span}(\psi_1, \ldots, \psi_r)$, and note that $\dim(U)=r$
   since $\{\psi_i\}$ are mutually orthogonal.
   Consider $f \in U$ and write $f = \sum_{i=1}^r \alpha_i \psi_i$.
   
   Since the functionals have mutually disjoint supports, for any $x,y \in V$, we have
   \[
      |f(x)-f(y)|^2 \leq 2 \sum_{i=1}^r \alpha_i^2 |\psi_i(x)-\psi_i(y)|^2\,.
   \]
   Therefore,
   \[
      \cR_G(f) \leq \frac{2 \sum_{i=1}^r \alpha_i^2 \|\psi_i(x)-\psi_i(y)\|^2}{\sum_{i=1}^r \alpha_i^2 \|\psi_i\|_{\pi}^2} \leq 2 \theta\,.
   \]
   Now the claim follows from $\dim(U)=r$ and the variational characterization of eigenvalues.
\end{proof}

\paragraph{Relation to return probabilities}

Let $\{\phi_k\}$ be an $L^2(\pi)$-orthonormal family of eigenfunctions for $P$
such that $P \phi_k = (1-\lambda_k) \phi_k$ for each $0 \leq k \leq n-1$.
The connection between return probabilities and eigenvalues is straightforward: For any $x \in V$ and $T \geq 1$,
\begin{equation}\label{eq:returns-ev}
p^G_T(x,x) =  \frac{\langle\1_x, P^T \1_x\rangle_{\pi}}{\pi(x)} = \sum_{k=0}^{n-1} \pi(x) \phi_k(x)^2 (1-\lambda_k)^T\,.
\end{equation}

\subsection{Random partitions}
\label{sec:metric-spaces}

We now introduce a tool that will be used
for analyzing the heat kernel in the remainder of this section.
Let $(X,d)$ denote a pseudometric space.

\paragraph{Random partitions}
For a partition $\cP$ of $X$, we use $\cP(x)$ to denote
the unique set in $\cP$ containing $x$.
We will consider only partitions $\cP$ with an at most countable
number of elements.
Denote
\[
   \Delta(\cP) \seteq \sup\left\{\diam(S,d):{S \in \cP}\right\}\,.
\]
A random partition $\bm{P}$ is {\em $(\tau,\alpha)$-padded}
if it satisfies the following conditions:
\begin{enumerate}
   \item Almost surely: $\Delta(\bm{P}) \leq \tau$.
   \item For all $x \in X$ and $\delta > 0$,
      \[ \Pr\left[B(x, \delta \tau/\alpha) \subseteq \bm{P}(x)\right] \geq 1-\delta\,. \]
\end{enumerate}

The reader might gain some intuition from considering the case $X=\R^d$ equipped with the Euclidean metric.
If one takes $\bm{P}$ to be a randomly translated partition of $\R^d$ into axis-aligned cubes of side-length $L$,
then $\bm{P}$ is $(L \sqrt{k}, k \sqrt{k})$-padded.
\medskip
\noindent
{\bf Uniformly decomposable graph families.}
We say that a family $\cF$ of locally finite, connected graphs is 
{\em $\alpha$-decomposable} if there is an $\alpha > 0$ 
such that for every $G \in \cF$,
every conformal weight $\omega : V(G) \to \R_+$, and every $\tau > 0$,
the metric space $(V(G),\dist_{\omega})$ admits a $(\tau,\alpha)$-padded
random partition.
We say that $\cF$ is {\em uniformly decomposable} if it is $\alpha$-decomposable
for some $\alpha > 0$.

The next result is proved in \cite{Lee16}.  The special case for graphs $G$
that themselves exclude $K_h$ as a minor
was established much earlier in \cite{KPR93}.
For such graphs, the bound $\alpha \leq O(h^2)$ was established in \cite{FT03},
and this was improved to $\alpha \leq O(h)$ in \cite{robbers14}.
Let $\cF(K_h)$ denote the family of connected, locally finite graphs that exclude $K_h$ as a minor,
and denote $\rig(\cF(K_h)) \seteq \bigcup_{G_0 \in \cF(K_h)} \rig(G_0)$.

\begin{theorem}[\cite{Lee16}]
   \label{thm:goodpad}
   For every $h \geq 1$, the family $\rig(\cF(K_h))$ is $\alpha$-decomposable
   for some $\alpha \leq O(h^2)$.
\end{theorem}

We say that a unimodular random graph $(G,\rho)$ is {\em uniformly decomposable}
if there is an $\alpha > 0$ such that $G$ is almost surely $\alpha$-decomposable.

We remark that often in the literature (e.g., in \cite{Lee16}), one only exhibits
random partitions that satisfy property (2) of a padded partition with $\delta=1/2$.
The following (unpublished) lemma of the author and A. Naor shows that this
is sufficient to conclude that it holds for all $\delta \in [0,1]$, with
a small loss in parameters.

\begin{lemma}
   Suppose that a metric space $(X,d)$ admits a random partition $\bm{P}$
   with $\Delta(\bm{P}) \leq \tau$ almost surely, and for every $x \in X$,
   \begin{equation}\label{eq:asspad}
     \Pr[B(x,\tau/\alpha) \subseteq \bm{P}(x)] \geq \frac12\,.
   \end{equation}
   Then there is a random partition $\bm{P'}$
   with $\Delta(\bm{P'}) \leq \tau$ almost surely, and such that
   for every $\delta > 0$ and $x \in X$,
   \begin{equation}\label{eq:delta-version}
      \Pr[B(x,\delta \tau/\alpha) \subseteq \bm{P'}(x)] \geq 1-4\delta\,.
   \end{equation}
\end{lemma}

\begin{proof}
   For a subset $S \subseteq X$ and a number $\lambda > 0$, denote
   \[
      S_{-\lambda} \seteq \{x \in S : B(x,\lambda) \subseteq S\}\,.
   \]
   Let $\{\bm{P}_k\}$ be an infinite sequence of i.i.d. random partitions with the law of $\bm{P}$.
   Let $\{\e_k\}$ be an independent infinite sequence of i.i.d. random variables
   where $\e_k \in [0,1]$ is chosen uniformly at random.
   We will define a sequence $\{A_k\}$ where each $A_k$ is a collection of disjoint 
   subsets of $X$
   and define $\bm{P'} = \bigcup_{k \geq 1} A_k$.

   Denote $A_0 = \emptyset, X_0 = \emptyset$ and for $k \geq 1$,
   \begin{align*}
   A_k &= \left\{ S_{-\e_k \tau/\alpha} \setminus X_{k-1} : S \in \bm{P}_k \right\} \\
      X_k &= X_{k-1} \cup \bigcup_{S \in A_k} S\,.
   \end{align*}

   First, observe that for every $x \in X$, it holds that almost surely
   $x \in S \in A_k$ for some $k \in \N$.
   This is because for every $k \geq 1$,
   \begin{equation}\label{eq:probk}
      \Pr\left[x \in \left(\bm{P}_k(x)\right)_{-\e_k \tau/\alpha}\right]
      \geq \Pr\left[B(x,\tau/\alpha)\subseteq \bm{P}_k(x)\right] \geq \frac12\,,
   \end{equation}
   where the latter inequality is from \eqref{eq:asspad}.
   Since we deal only with separable metric spaces,
   to verify that $\bm{P'}$ is almost surely a partition,
   it suffices to consider a dense countable subset of $X$.
   Similarly, we can conclude that $\bm{P'}$ almost surely satisfies $\Delta(\bm{P'})\leq \tau$.

   So now we move on to verifying \eqref{eq:delta-version}.  Fix $x \in X$ and $\delta \in [0,1/4]$.
   Let $R=\delta \tau/\alpha$.
   Then,
   \begin{align*}
      \Pr\left[B(x, R) \nsubseteq \bm{P'}(x)\right] \leq
      \sum_{k \geq 1} \Pr\left[B(x,R) \cap X_{k-1} = \emptyset\right] \cdot
      \Pr\left[B(x,R) \cap X_k \neq \emptyset \wedge B(x,R) \nsubseteq (\bm{P}_k(x))_{-\e_k \tau/\alpha}\right]\,.
   \end{align*}
   Now observe that \eqref{eq:probk} implies $\Pr[B(x,R) \cap X_{k-1} = \emptyset] \leq 2^{1-k}$.

   For $y \in X$, let $\eta_k(y) = \sup \left\{ \eta \geq 0 : B(y,\eta) \subseteq \bm{P}_k(y) \right\}\,.$
   Note that, conditioned on $\bm{P}_k$, $\eta_k$ is a $1$-Lipschitz function.
   Therefore,
   \[
      \Pr\left[B(x,R) \cap X_k \neq \emptyset \wedge B(x,R) \nsubseteq (\bm{P}_k(x))_{-\e_k \tau/\alpha}\right]
      \leq 
      \Pr\left(\e_k \in \left[\frac{\eta_k(x)}{\tau/\alpha}-\delta, \frac{\eta_k(x)}{\tau/\alpha}+\delta\right]\right) \leq 2\delta\,,
   \]
   We conclude that
   \[
      \Pr\left[B(x,R) \nsubseteq \bm{P'}(x)\right] \leq 2\delta \sum_{k \geq 1} 2^{1-k} = 4\delta\,,
   \]
   completing the proof.
\end{proof}

\subsection{Eigenvalues and the degree distribution}

Let us define $\Delta_G : \Z_+ \to \N$ by
\[
   \Delta_G(k) = \max \left\{ \sum_{x \in S} \deg_G(x) : S \subseteq V, |S| \leq k \right\}\,.
\]
Momentarily, we will prove the following theorem.
Say that a graph is {\em $(\kappa,\alpha)$-controlled} if 
it is $\alpha$-decomposable and $(\kappa,R)$-quadratic for all $R \geq 1$.

\begin{theorem}\label{thm:deg-ev}
   Suppose that a family $\cF$ of finite graphs has uniform gQCG and is uniformly decomposable.
   Then there is a constant $c > 0$ such that for every $G \in \cF$ and $k=0,1,\ldots,|V(G)|-1$,
   \[
      \lambda_k(G) \leq c \frac{\Delta_G(k)}{|V(G)|}\,.
   \]
   Quantitatively, if a finite graph $G$ is $(\kappa,\alpha)$-controlled, then
   \[
      \lambda_k(G) \lesssim \alpha^2 \kappa \frac{\Delta_G(k)}{|V(G)|}\,.
   \]
\end{theorem}

We present an illustrative corollary of \pref{thm:quadgrowth} and \pref{thm:goodpad} in conjunction with \pref{thm:deg-ev}.

\begin{corollary}\label{cor:Kh-eigenvalues}
   Suppose $G$ is an $n$-vertex graph that excludes $K_h$ as a minor.
   Then there is a constant $c_h \leq O(h^6 \log h)$ such that
   for every $k=0,1,\ldots, n-1$,
   \[
      \lambda_k(G) \leq c_h \frac{\Delta_G(k)}{n}\,.
   \]
\end{corollary}

Note that a weaker statement
was established in \cite{KLPT09} with $\Delta_G(k)$ replaced by $k \cdot \dmax(G)$.
The proof of \pref{thm:deg-ev} is immediate
from \pref{cor:variational} and the following result.

\begin{theorem}
   \label{thm:firstbumps}
   Suppose $G$ is an $n$-vertex graph and $\omega : V(G) \to \R_+$
   is a normalized conformal weight such that:
   \begin{enumerate}
      \item For all $x \in V(G)$,
         \[
            \left|B_{\omega}\!\left(x, R_*\right)\right| \leq \kappa R_*^2\,,
         \]
         where $R_* = \sqrt{\frac{n}{16 k \cdot \kappa}}$.
      \item $(V(G),\dist_{\omega})$ admits an $(\alpha, R_*/2)$-padded random partition.
   \end{enumerate}
   Then for every $k \leq n$, there are disjointly supported functions
   $\psi_1, \ldots, \psi_k : V \to \R$ such that
   \[
      \cR_G(\psi_i) \lesssim \alpha^2 \kappa \frac{\Delta_G(k)}{n}\,.
   \]
\end{theorem}

\pref{sec:bumps} is devoted to the construction of the bump functions $\psi_1,\ldots,\psi_k$.
As discussed in the introduction,
the spectral bounds from \pref{thm:deg-ev} are
not strong enough to yield almost sure
bounds on the heat kernel
of a distributional limit.
Consulting \eqref{eq:returns-ev},
one sees that to control the return probabilities
for most vertices $x \in V(G)$
requires us to say something about the distribution
of the low-frequency eigenfunctions of $G$.

\subsection{Return probabilities and spectral delocalization}
\label{sec:retprob}

Let us now indicate how 
a sufficient strengthening of \pref{thm:deg-ev} 
will allow us to control
return probabilitie for most of the vertices.
For our finite graph $G=(V,E)$,
it will help to define for every $\e > 0$:
\[
   \pi_G^*(\e) \seteq \max \{ \pi(S) : |S| \leq \e |V| \}\,.
\]

\begin{theorem}\label{thm:bump-return}
   Let $G=(V,E)$ be an $n$-vertex graph.
	Suppose that for some $k \leq n$, there are capacitors $(A_1,\Omega_1), \ldots, (A_k,\Omega_k)$
   so that $\{\Omega_i\}$ are pairwise disjoint and $|\Omega_i| \leq M$ for
	all $i=1,\ldots,k$.  Then for all $\e > 0$ and $T \geq 1$:
   \begin{equation}\label{eq:first-con}
		\pi\left(\left\{x \in V : \frac{p_{2T}^G(x,x)}{\pi(x)} \geq \frac{\e |V|}{4 M}\right\}\right)
			\geq -2 \pi_G^*(\e)
+ \sum_{i=1}^k \pi(A_i) - 2T \sum_{i=1}^k \capacity_{\Omega_i}(A_i)\,.
\end{equation}
   In particular, for any $\beta > 0$,
   \begin{equation}\label{eq:second-con}
      \pi\left(\left\{ x \in V : p^G_{2T}(x,x) \geq \frac{\e \beta}{4 M}\right\}\right) \geq
      -2 \pi_G^*(\e) -\beta
			+ \sum_{i=1}^k \pi(A_i) - 2T \sum_{i=1}^k \capacity_{\Omega_i}(A_i)
   \end{equation}
\end{theorem}

For illustration, consider a bounded-degree graph planar graph.
In \pref{sec:bumps}, we will show that under this assumption,
for every $\e > 0$ and $M \ll n$, we can find such a family $\{(A_i,\Omega_i)\}$
satisfying
\[
   \sum_{i=1}^k \pi(A_i) \geq 1-\e\,,
\]
and for each $i=1,\ldots,k$,
\begin{align}
   \capacity_{\Omega_i}(A_i) &\leq \frac{c(\e)}{M} \pi(A_i)\,,
\end{align}
where $c(\e)$ is some function of $\e$.
Since our graph has bounded degrees, we have $\pi_G^*(\e) \leq O(\e)$,
so choosing $T \leq \frac{\e^2 M}{c(\e)}$ yields
\[
   \pi\left(\left\{ x \in V : p^G_{2T}(x,x) \geq \frac{c'(\e)}{T}\right\}\right) \geq 1-O(\e)\,.
\]
for some other function $c'(\e)$.

We will require the following two preliminary results.

\begin{lemma}\label{lem:capacs}
	For any capacitor $(A,\Omega)$ and $T \geq 1$,
\[
   \langle \1_{\Omega}, P^T \1_{\Omega}\rangle_{\pi} \geq \pi(A) - T \cdot \capacity_{\Omega}(A)\,.
\]
\end{lemma}

\begin{remark}
   We remark that the same argument gives an identical lower bound on
   $\langle \1_{\Omega}, (I_{\Omega} P I_{\Omega})^T \1_\Omega\rangle_{\pi}$ where $I_\Omega$ is the multiplication
   operator on $L^2(\pi)$ defined by $I_\Omega f(x) = \1_\Omega(x) f(x)$.
   Then $I_\Omega P I_\Omega$ is the operator of the walk killed off $\Omega$.
\end{remark}

Using reversibility, a lower bound on $\langle\1_S, P^T \1_S\rangle_{\pi}$ will give us control
on return probabilities.

\begin{lemma}\label{lem:p2t}
   Suppose that, for some $S \subseteq V$, we have
   \[
      \langle \1_S, P^T \1_S\rangle_{\pi} \geq (1-\delta) \pi(S)\,.
   \]
   Then for any $\gamma > 0$,
   \[
      \pi\left(\Big\{x \in S : p_{2T}(x,x) \geq \tfrac{\pi(x)}{4 \gamma |S|}\Big\} \right) \geq \left(1-2\delta \right)\pi(S) - 2 \pi\left(\{x \in S : \pi(x) > \gamma\}\right)\,.
   \]
\end{lemma}

\begin{proof}[Proof of \pref{thm:bump-return}]
   Apply \pref{lem:capacs} and \pref{lem:p2t} to each $(A_i,\Omega_i)$
   with $\delta_i = T \frac{\capacity_{\Omega_i}(A_i)}{\pi(A_i)}$ and sum over $i=1,\ldots,k$, yielding
   \[
      \sum_{i=1}^k \pi\left(\left\{x \in A_i : p^G_{2T}(x,x) \geq \frac{\pi(x)}{4 \gamma |\Omega_i|}\right\}\right)\geq
         \sum_{i=1}^k (1-2\delta_i) \pi(A_i) - 2 \pi\left(\{x \in V : \pi(x) > \gamma\}\right)\,,
   \]
   where we have used that the sets $\{A_i\}$ are pairwise disjoint.

   The former sum is precisely
   \[
      \sum_{i=1}^k \pi(A_i) - 2 T \capacity_{\Omega_i}(A_i)\,,
   \]
   and $|\Omega_i| \leq M$ for all $i=1,\ldots,k$ by assumption.
      Conclude the proof of \eqref{eq:first-con} by setting $\gamma = 1/(\e |V|)$
      so that the second term is at least $-2\pi_G^*(\e)$.
         To obtain \eqref{eq:second-con}, remove all $x \in V$ with $\pi(x) < \beta/|V|$.
\end{proof}

Let us now prove the lemmas.

\begin{proof}[Proof of \pref{lem:capacs}]
   We need the following basic fact.
\begin{lemma}[\cite{MS59,BR65}]
\label{lem:psd-power}
   Suppose $Q$ is a self-adjoint operator on $L^2(\pi)$ with $\langle \1_u,Q \1_v\rangle_{\pi} \geq 0$ for all $u,v \in V$,
and $\psi \in L^2(\pi)$ satisfies $\psi \geq 0$ and $\|\psi\|_{\pi}=1$.  Then for every integer $T \geq 1$:
   \[
      \langle \psi, Q^T \psi\rangle_{\pi} \geq \left(\langle \psi, Q \psi\rangle_{\pi}\right)^T.
   \]
\end{lemma}

Now let $\f : V \to [0,1]$ be any function satisfying $\supp \f \subseteq \Omega$.
Define $\psi \seteq \f/\|\f\|_{\pi}$.
Using \pref{lem:psd-power} and $\f \leq \1_{\Omega}$, we have
\[
   \frac{\langle \1_{\Omega}, P^T \1_{\Omega}\rangle_{\pi}}{\|\f\|_{\pi}^2} \geq
		\langle \psi, P^T \psi\rangle_{\pi}
 		\geq \langle \psi, P \psi\rangle_{\pi}^T \geq \left(1 - \langle \psi, (I-P) \psi\rangle_{\pi}\right)^T
			= \left(1-\cR_{G}(\psi)\right)^T.
\]
Since $\cR_G(\f)=\cR_G(\psi)$,
\[
	\langle \1_\Omega, P^T \1_\Omega\rangle_{\pi} \geq \|\f\|_{\pi}^2 \left(1-\cR_{G}(\f)\right)^T
		\geq \|\f\|_{\pi}^2 \left(1- T \cR_G(\f)\right)
= \|\f\|_{\pi}^2 - T \cE_{G}(\f).
\]

Using the definition of the capacity, take now a $\f$ that additionally satisfies $\f|_A \equiv 1$
and $\cE_G(\f) = \capacity_{\Omega}(A)$, yielding
\[
   \langle \1_{\Omega}, P^T \1_{\Omega}\rangle_{\pi} \geq \pi(A) - T\cdot\capacity_{\Omega}(A)\,.\qedhere
\]
\end{proof}

\begin{proof}[Proof of \pref{lem:p2t}]
   Let $L_{\gamma}(S) \seteq \{ y \in S : \pi(y) \leq \gamma\}$.
Using reversibility, write
\begin{align*}
   p_{2T}(x,x) \geq \sum_{y \in S} p_T(x,y) p_T(y,x) 
               &= \sum_{y \in S} p_T(x,y)^2 \frac{\pi(x)}{\pi(y)} \\
   &\geq \frac{\pi(x)}{\gamma} \sum_{y \in S : \pi(y) \leq \gamma} p_T(x,y)^2  \\
   &\geq
   \frac{\pi(x)}{\gamma |S|} \left(\sum_{y \in S : \pi(y) \leq \gamma} p_T(x,y)\right)^2 \\
   &= \frac{\pi(x)}{\gamma|S|} p_T(x, L_{\gamma}(S))^2\,.
\end{align*}
This gives:
\begin{equation}\label{eq:p2t0}
   \pi\left(\left\{x \in S : p_{2T}(x,x) \geq \frac{1}{4} \frac{\pi(x)}{\gamma |S|} \right\}\right) 
   \geq
   \pi\left(\left\{x \in S : p_T(x,L_{\gamma}(S)) \geq \frac12 \right\}\right).
\end{equation}

On the other hand, note that
\begin{align*}
   \sum_{x \in S} \pi(x) p_T(x,S) = \sum_{x,y \in S} \langle \1_x, P^T \1_y\rangle_{\pi} = \langle \1_S, P^T \1_S\rangle_{\pi} \geq (1-\delta) \pi(S)\,.
\end{align*}
Therefore, 
\[
   \sum_{x \in S} \pi(x) p_T(x,L_{\gamma}(S)) \geq \left(1-\delta\right) \pi(S) - \pi\left(S \setminus L_{\gamma}(S)\right),
\]
and Markov's inequality yields
\[
   \pi\left(\left\{ x \in S : p_T(x,L_{\gamma}(S)) \geq \frac12\right\}\right) \geq (1-2\delta) \pi(S) - 2 \pi(S \setminus L_{\gamma}(S)).
\]
Combining this with \eqref{eq:p2t0} yields the claimed inequality.
\end{proof}

\subsection{Constructing bump functions}
\label{sec:bumps}

We will now show that, given a conformal metric $\omega : V \to \R_+$
with sufficiently nice properties, we can construct
many disjoint bump functions with small Rayleigh quotient.
Our main geometric tool will be random partitions of metric spaces
(cf. \pref{sec:metric-spaces}).

It will be easier to first prove \pref{thm:firstbumps},
and then to perform the more complicated construction
needed for \pref{thm:bump-return}.
Let us define the function $\bar{d}_G : [0,1] \to \N$ by
\[
   \bar{d}_G(\e) \seteq \frac{\Delta_G(\e n)}{\e n}\,,
\]
which is the average degree among the $\e n$ vertices of largest degree in $G$.
It is useful to observe that following simple fact:  For every $C > 1$,
\begin{equation}\label{eq:num-vs-avg}
   \#\left\{ x \in V : \deg_G(x) \geq C \bar{d}_G(\e) \right\} \leq \frac{\e n}{C}\,.
\end{equation}
Indeed, if $N \seteq \# \{ x \in V : \deg_G(x) \geq C \bar{d}_G(\e) \}$, then
$\min(N, \e n) C \bar{d}_G(\e) \leq \Delta_G(\e n)$,
implying $\min(N,\e n)\leq \frac{\e n}{C}$.  For $C > 1$, this gives \eqref{eq:num-vs-avg}.

\subsubsection{Many disjoint bumps}

Suppose we have
a conformal metric $\omega : V \to \R_+$ that
satisfies the following assumptions:
For some numbers $R > 0$ and $\alpha, K \geq 1$,
\begin{enumerate}
   \item[(A0)]  $K \leq n/2$.
   \item[(A1)]  For all $x \in V$, it holds that $|B_{\omega}(x, R)| \leq K$.
   \item[(A2)]  The space $(V, \dist_{\omega})$ admits an $(R/2, \alpha)$-padded random partition.
\end{enumerate}
Define the quantity
\begin{equation}\label{eq:eta-def-easy}
   \eta \seteq R/(12\alpha)\,.
\end{equation}

When dealing with unbounded degrees,
we have to be careful about handling vertices of large conformal weight.
To this end, for $\eta > 0$, define the set
\[
   V_L \seteq \{ x \in V : \omega(x) \geq \eta \}\,.
\]
For a subset $S \subseteq V$, define
\[
   \area^{\eta}_{\omega}(S) \seteq 16\, \avgd_G(1/K) \ao(S) + \eta^2 \cdot |E_G(S, V_L) |\,.
\]
Observe that $\area^{\eta}_{\omega}$ is a measure on $V$, and
\begin{equation*}
   \ao^{\eta}(V) 
\leq 16\,\avgd_G(1/K) \|\omega\|_{\ell^2(V)}^2 + \eta^2 \cdot |E_G(V, V_L)|\,.
\end{equation*}
Since $|V_L| \leq \frac{\|\omega\|_{\ell^2(V)}^2}{\eta^2}$, it holds that
\begin{equation}
\label{eq:totalvol-easy}
\area^{\eta}_{\omega}(V) \leq \|\omega\|^2_{\ell^2(V)}
\left(16\,\avgd_G(1/K)+\avgd_G\left(\tfrac{1}{n}\|\omega\|_{\ell^2(V)}^2/\eta^2\right)\right)\,.
\end{equation}

\begin{lemma}\label{lem:sepsets-easy}
  Under assumptionss (A0)--(A2),
there exist disjoint subsets $T_1, T_2, \ldots, T_r \subseteq V$
such that $r \geq n/8K$, and moreover:
\begin{enumerate}
   \item \label{item:eachsmall-easy} For all $i=1,\ldots, r$, it holds that $\frac{K}{2} \leq |T_i| \leq K$, and
\begin{align*}
   \ao^{\eta}\left(B_{\omega}(T_i,R/6\alpha)\right) &\leq \frac{3}{r} \ao^{\eta}(V)\,.
\end{align*}
\item \label{item:farapart-easy} For all $i \neq j$,
\[
\dist_{\omega}(T_i, T_j) \geq \frac{R}{2\alpha}\,.
\]
\end{enumerate}
\end{lemma}

\begin{proof}
Let $\cP = \{S_1, S_2, \ldots, S_m\}$ be a partition of $V$
such that $\diam_{\omega}(S_i) \leq R/2$ for each $i$.
By property (A1), it holds that
\begin{equation}\label{eq:sisize-easy}
|S_i| \leq K\,.
\end{equation}
For each $i=1,\ldots,m$, define
\[
\hat{S}_i = \left\{ x \in S_i : B_{\omega}(x, R/4\alpha) \subseteq S_i \vphantom{\bigoplus}\right\}.
\]
Observe that for $i \neq j$, we have $\dist_{\omega}(\hat S_i, \hat S_j) \geq R/2\alpha$ by
construction.

Let $N_{\cP} = |\hat{S}_1| + \cdots + |\hat {S}_m|$.
Suppose now that $\bm{P}$ is an $(R/2,\alpha)$-padded random partition.
From the definition and linearity of expectation, we have
\[
   \E[N_{\bm{P}}] \geq \frac12 |V|\,.
\]
So let us fix a partition $\cP$ satisfying $N_{\cP} \geq \frac12 |V|$ for the remainder of the proof.

Using \eqref{eq:sisize-easy}, it is possible to take
unions of the sets $\{\hat{S}_i : i \in I\}$ to form
disjoint sets $T_1, T_2, \ldots, T_r$ with
$\frac{K}{2} \leq |T_i| \leq K$ and
such that for $i \neq j$, $\dist_{\omega}(T_i, T_j) \geq R/(2\alpha)$.
In this process, we discard at most $K/2$ points, thus
\[
   |T_1|+\cdots+|T_r| \geq \frac12 |V| - \frac{K}{2} \geq \frac{1}{4} |V|\,,
\]
where the final inequality uses assumption (A0).
In particular, we have $r \geq n/4K$.

Let us now sort the sets so that 
\[
   \ao^{\eta}\left(B_{\omega}(T_1,R/6\alpha)\right) \leq 
   \ao^{\eta}\left(B_{\omega}(T_2,R/6\alpha)\right) \leq  \cdots \leq
   \ao^{\eta}\left(B_{\omega}(T_r,R/6\alpha)\right)\,.
\]
Then for $i \in \{1,2,\ldots,\lceil r/2\rceil\}$,
since the sets $\left\{ B_{\omega}(T_j, R/6\alpha) : j \in [r]\right\}$ are pairwise
disjoint by construction,
it must be that $\ao^{\eta}\left(B_{\omega}(T_i,R/6\alpha)\right) \leq \frac{3}{r} \ao^{\eta}(V)$.
Thus the statement of the lemma is satisfied by the sets $\{ T_i : i < r/2 + 1 \}$.
\end{proof}

Next, we observe that we can remove sets that have a vertex
of large degree.

\begin{lemma}\label{lem:sepsets2-easy}
   Under assumptions (A0)--(A2), 
there exist disjoint
subsets $T_1, T_2, \ldots, T_r \subseteq V$ with $r \geq n/16K$,
satisfying properties \eqref{item:eachsmall-easy} and \eqref{item:farapart-easy} of \pref{lem:sepsets-easy},
and furthermore
\begin{equation}
\label{eq:degree-bnd-easy}
\max \left\{ \deg_G(x) : x \in B_{\omega}(T_i,R/6\alpha)\right\} \leq 16\, \avgd_G(1/K)\,, \qquad i=1,2,\ldots,r.
\end{equation}
\end{lemma}

\begin{proof}
   Recalling \eqref{eq:num-vs-avg}, there are at most $n/16K$ vertices
   with degree larger than $16 \,\avgd_G(1/K)$.
   Thus one can apply \pref{lem:sepsets-easy} and then remove at most $n/16K$
   of the sets that contain a vertex of large degree.
\end{proof}

We are now ready to construct the bump functions.

\begin{theorem}\label{thm:bumps-easy}
If $\omega : V \to \R_+$ is a normalized conformal metric
on $G$ satisfying assumptions (A0)--(A2), then
then there exist disjointly supported functions
$\psi_1, \psi_2, \ldots, \psi_{r} : V \to \R_+$
with $r \geq n/16K$,
and such that for all $i=1,\ldots,r$,
\begin{equation}
\label{eq:rayleigh-easy}
\cR_G(\psi_i) \lesssim \frac{\alpha^2 \left(\avgd_G(1/K) +
\avgd_G\left(\alpha^2/R^2\right)\right)}{R^2}\,.
\end{equation}
\end{theorem}

\begin{proof}
   Let $T_1, T_2, \ldots, T_r \subseteq V$ be the subsets guaranteed by \pref{lem:sepsets2-easy}.
   For each $i \in [r]$, define
   \[
      \psi_i(x) = \max\left\{0, \eta - \dist_{\omega}(x, T_i)\right\}\,.
   \]
   By construction, $T_i \subseteq \supp(\psi_i) \subseteq B_{\omega}(T_i, \eta)$, hence by \pref{lem:sepsets-easy}(2),
   the functions $\{\psi_i : i \in [r]\}$ are disjointly supported.  
   (Recall that $\eta = R/(12\alpha)$.)

   Consider $\{x,y\} \in E$.  If $|\psi_i(x)-\psi_i(y)| > 0$, then at least one endpoint must lie in $\supp(\psi_i)
   \subseteq B_{\omega}(T_i,\eta)$.  Suppose $x \in B_{\omega}(T_i,\eta)$.  If $y \notin V_L$, then
   $\omega(y) < \eta$, implying that $y \in B_{\omega}(T_i,2 \eta)$.
   Therefore we can bound
   \[
      \sum_{\{x,y\} \in E} |\psi_i(x)-\psi_i(y)|^2 \leq 
      \eta^2 |E_G(B_{\omega}(T_i, \eta), V_L)|   +   
      \sum_{\substack{\{x,y\} \in E : \\ \{x,y\} \subseteq B_{\omega}(T_i, 2\eta)}} |\psi_i(x)-\psi_i(y)|^2,
   \]
   where we have used $|\psi_i(x)-\psi_i(y)| \leq \eta$ for all $x,y \in V$.
   Now use the fact that each $\psi_i$ is $1$-Lipschitz to write
   \begin{align*}
      \sum_{\{x,y\} \in E} |\psi_i(x)-\psi_i(y)|^2 &\leq 
      \eta^2 |E_G(B_{\omega}(T_i, \eta), V_L)|   +   
      \sum_{\substack{\{x,y\} \in E : \\ \{x,y\} \subseteq B_{\omega}(T_i, 2\eta)}} \dist_{\omega}(x,y)^2 \\
   &\stackrel{\mathclap{\eqref{eq:degree-bnd-easy}}}{\leq}
   \eta^2 |E_G(B_{\omega}(T_i, \eta), V_L)|   +  16 \,\avgd_G(1/K)\, \ao(B_{\omega}(T_i, R/6\alpha)) \\
      &\leq \ao^{\eta}(B_{\omega}(T_i, R/6\alpha))\,,
   \end{align*}
   where the second inequality uses $\dist_{\omega}(x,y)^2 \leq (\omega(x)^2 + \omega(y)^2)/2$ for
   $\{x,y\} \in E$, and we recall that $2 \eta = R/6\alpha$.

Combining this with \pref{lem:sepsets-easy}(1) yields
   \begin{align*}
      \cR_G(\psi_i) &= \frac{2 \sum_{\{x,y\} \in E} |\psi(x)-\psi(y)|^2}{\sum_{x \in V} \deg_G(x) \psi(x)^2} 
      \leq
   \frac{6 \ao^{\eta}(V)}{r \eta^2 |T_i|} 
   \leq
   \frac{864 \alpha^2 \ao^{\eta}(V)}{R^2 |V|}\,.
   \end{align*} 
   To arrive at the statement of the theorem, use \eqref{eq:totalvol-easy}
   and the assumption that $|V|^{-1} \|\omega\|_{\ell^2(V)}^2=1$.
\end{proof}

Let us now use this to prove \pref{thm:firstbumps}.

\begin{proof}[Proof of \pref{thm:firstbumps}]
   Consider $R=R_*=\sqrt{n/(16\kappa \cdot k)}$.
   By assumption, there is a normalized conformal metric $\omega : V \to \R_+$
   satisfying $\max_{x \in V} |B_{\omega}(x,R)| \leq \kappa R^2$.
   
   We may assume that $\kappa \geq 1$.
   Let $K=\kappa R^2$.
   Again by assumption,
   $(V,\dist_{\omega})$ admits an $(R/2, \alpha)$-padded random partition.
   Now apply \pref{thm:bumps-easy} to find $r \geq |V|/(16 \kappa R_*^2)$ disjointly supported test functions
   $\{\psi_i\}$,
   each with
   \[
      \cR_G(\psi_i) \lesssim \frac{\alpha^2 \left(\avgd_G\left(1/(\kappa R_*^2)\right) + \avgd_G\left(\alpha^2/R_*^2\right)\right)}{R_*^2}
      \leq 2
      \frac{\alpha^2 \avgd_G\left(k/n\right)}{R_*^2}
      \lesssim \alpha^2 \kappa \,\avgd_G(k/n) \frac{k}{n} = \alpha^2 \kappa \,\frac{\Delta_G(k)}{n}\,.
   \]

   We may assume that $k \leq n/(16 \kappa)$.  (Otherwise, we can just take $n$ functions---one supported on each vertex of the graph---since
   the bound we are required to prove on the Rayleigh quotient is trivial.)
   Note that in this case,  $r \geq |V|/(16 \kappa R_*^2) \geq k$, completing the proof.
\end{proof}

\subsubsection{Exhausting the stationary measure by capacitors}
\label{sec:super-bumps}

Our arguments will follow along similar lines to those of
the preceding section
although things will be somewhat more delicate.
$G=(V,E)$ is an $n$-vertex, connected graph.
Suppose we have
a conformal metric $\omega : V \to \R_+$ that
satisfies (A1) and (A2) 
for some numbers $R > 0$, $\alpha, K \geq 1$.

Consider a number $\delta > 0$ and
define
\begin{align*}
   \eta &\seteq \frac{\delta R}{18\alpha}\,, \\
   V_L &\seteq \{ x \in V : \omega(v) \geq \eta \}\,.
\end{align*}

\begin{lemma}\label{lem:sepsets}
   For any $\delta > 0$, it holds that
   under assumptions (A1) and (A2),
  there are pairwise disjoint sets $S_1, \ldots, S_r$ satisfying $\diam_{\omega}(S_i) \leq R/2$
  for each $i=1,\ldots,r$, and such that
  \begin{equation}\label{eq:en-mass}
   \sum_{i=1}^r \pi\left(\hat{S}_i\right) \geq 1-\delta-\pi_G^*(\delta)\,,
\end{equation}
where for a subset $S \subseteq V$, we denote
\[
   \hat{S} \seteq \left\{ x \in S : B_{\omega}(x, \delta R/6\alpha) \subseteq S \vphantom{\bigoplus}\right\}.
\]
Moreover, it holds that
\begin{equation}\label{eq:degree-bnd}
   \max \left\{ \deg_G(x) : x \in S_1 \cup \cdots \cup S_r \right\} \leq \avgd_G(\delta/K)\,,
\end{equation}
and
\begin{equation}\label{eq:halfies}
   \pi(\hat{S}_i) \geq \frac12 \pi(S_i) \qquad \forall i=1,\ldots,r\,.
\end{equation}
\end{lemma}

\begin{proof}
   Let $\bm{P}$ denote an $(R/2,\alpha)$-padded random partition of $(V,\dist_{\omega})$.
   Using linearity of expectation and the definition of a padded random partition yields
   \[
      \E\left[\sum_{S \in \bm{P}} \pi(\hat{S})\right] \geq 1-\delta/3\,.
   \]
   Let us fix a partition $\cP$ in the support of $\bm{P}$ satisfying $\sum_{S \in \cP} \pi(\hat{S}) \geq 1-\delta/3$.
   
   Let $\cP' = \{ S \in \cP : \pi(\hat{S}) \geq \tfrac12 \pi(S) \}$ and note that
   \[
      \sum_{S \in \cP'} \pi(\hat{S}) \geq 1-(\delta/3)-2(\delta/3)= 1 - \delta\,.
   \]
   Finally,
   denote
   \[
      \{S_1, \ldots, S_r\} = \left\{ S \in \cP' : \max \{ \deg_G(x) : x \in S \}\leq \avgd_G(\delta/K)\right\} \,.
   \]
   Recalling \eqref{eq:num-vs-avg}, there are at most $\frac{\delta}{K} |V|$ vertices in $G$
   with degree larger than $\avgd_G(\delta/K)$.
   Therefore,
\[
   \sum_{i=1}^r \pi(\hat{S}_i) \geq \sum_{S \in \cP'} \pi(\hat{S})
   - \sum_{S \in \cP \setminus \{S_1,\ldots,S_r\}} \pi(S)
   \geq 1-\delta
   - \pi_G^*\left(\frac{\delta}{K} 
   \max_{S \in \cP} |S_i|\right)
   \geq 1-\delta-\pi_G^*(\delta)\,,
\]
where the final inequality uses the fact that $|S| \leq K$ for $S \in \cP$ which follows
from $\diam_{\omega}(S)\leq R/2$ and (A1).
\end{proof}

We are now ready to construct the capacitors.

\begin{theorem}\label{thm:bumps}
If $\omega : V \to \R_+$ is a conformal metric
on $G$
satisfying assumptions (A1) and (A2),
then there exist pairwise disjoint
capacitors $(A_1, \Omega_1), \ldots, (A_k,\Omega_k)$ with
\begin{align}
   \sum_{i=1}^k & \pi\left(A_i\right) \geq 1 - \delta-\pi_G^*(\delta)\,,\label{eq:widesupp}
\end{align}
and such that for all $i=1,\ldots,k$,
\begin{align}
   \diam_{\omega}(\Omega_i) &\leq R/2\label{eq:supp-diam-bound}\,,\\
      |\Omega_i| &\leq K\label{eq:sizeybnd}\,.
\end{align}
And futhermore,
\begin{equation}
   \sum_{i=1}^k \capacity_{\Omega_i}(A_i) \leq
   \frac{18^2 \alpha^2}{\delta^2 R^2} \|\omega\|_{L^2(V)}^2
   \frac{\avgd_G\left(\frac{18^2 \alpha^2}{\delta^2 R^2} \|\omega\|_{L^2(V)}^2\right)+\avgd_G(\delta/K)}
      {\avgd_G(1)}\,.
\end{equation}
\end{theorem}

\begin{proof}[Proof of \pref{thm:bumps}]
   Let $S_1, S_2, \ldots, S_k \subseteq V$ be the subsets guaranteed by \pref{lem:sepsets}.
   For each $i=1,\ldots,k$, define
   \[
      \psi_i(x) \seteq \frac{1}{\eta} \max\left\{0, \eta -  \dist_{\omega}(x, \hat{S}_i)\right\}\,.
   \]
   We denote $A_i \seteq \hat{S}_i$ and $\Omega_i\seteq S_i$.
   By construction, we have $0 \leq \psi_i \leq 1$, as well as
   $\supp \psi_i \subseteq \Omega_i$ and $\psi_i|_{A_i} \equiv 1$.
   One observes that \eqref{eq:widesupp}
   follows from \eqref{eq:en-mass}.
   Similarly, \eqref{eq:supp-diam-bound} follows from \pref{lem:sepsets}, and \eqref{eq:sizeybnd}
   then follows from assumption (A1).

   Use the fact that $\eta \psi_i$ is $1$-Lipschitz to calculate
   \begin{align*}
   \eta^2 \sum_{\{x,y\} \in E} |\psi_i(x)-\psi_i(y)|^2 &\leq 
      \eta^2 |E_G(B_{\omega}(\hat{S}_i, \eta), V_L)| +
      \sum_{\substack{\{x,y\} \in E : \\ \{x,y\} \subseteq B_{\omega}(\hat{S}_i, \delta R/6\alpha)}} \dist_{\omega}(x,y)^2 
   \\
   &
   \leq
   \eta^2 |E_G(S_i, V_L)|   +  \avgd_G(\delta/K) \ao(B_{\omega}(\hat{S}_i, \delta R/6\alpha)) \\
   &\leq \eta^2 |E_G(S_i, V_L)| + \avgd_G(\delta/K) \ao(S_i)\,.
   \end{align*}

   Therefore,
   \begin{equation}\label{eq:capac1}
      \sum_{i=1}^k \capacity_{\Omega_i}(A_i) \leq \sum_{i=1}^k \cE_G(\psi_i) \leq \frac{1}{|E|} \left(|E_G(V,V_L)| + \bar{d}_G(\delta/K)
      \eta^{-2} \|\omega\|_{\ell^2(V)}^2\right)\,.
   \end{equation}
   Note that $|V_L| \leq \eta^{-2} \|\omega\|_{\ell^2(V)}^2$, yielding
   \[
      |E_G(V,V_L)|  \leq \Delta_G\left(\eta^{-2} \|\omega\|_{\ell^2(V)}^2\right) = \eta^{-2} \|\omega\|_{\ell^2(V)}
            \avgd_G\left(\eta^{-2} \|\omega\|^2_{L^2(V)}\right)\,.
   \]
   Finally, note that
   \[
      \frac{\|\omega\|_{\ell^2(V)}^2}{|E|} = \frac{\|\omega\|_{L^2(V)}^2}{\avgd_G(1)}\,.
   \]
   Plugging these into \eqref{eq:capac1} and using the definition of $\eta$ yields
   \[
      \sum_{i=1}^k \capacity_{\Omega_i}(A_i) \leq \frac{18^2 \alpha^2}{\delta^2 R^2} \|\omega\|_{L^2(V)}^2
      \frac{\avgd\left(\frac{18 \alpha^2}{\delta^2 R^2} \|\omega\|_{L^2(V)}^2\right)+\avgd(\delta/K)}{\avgd_G(1)}\,.
      \qedhere
   \]
\end{proof}

Combining \pref{thm:bumps} with \pref{thm:bump-return} yields the following corollary.

\begin{corollary}\label{cor:on-diag}
   There is a universal constant $C \geq 1$ such that
   if $\omega : V \to \R_+$ is a conformal metric satisfying assumptions (A1) and (A2), then
   for every $\delta,\beta > 0$ and $T \geq 1$,
   \begin{align*}
      \pi\left(\left\{x \in V : p_{2T}^G(x,x) < \frac{\delta \beta}{4K}\right\}\right) 
      \leq \beta + \delta &+ 3 \pi_G^*(\delta) \\
                          &+
      \frac{C \alpha T}{\delta^2 R^2} \|\omega\|_{L^2(V)}^2
      \frac{\avgd\left(\frac{\alpha^2}{\delta^2 R^2} \|\omega\|^2_{L^2(V)}\right)+\avgd(\delta/K)}{\avgd_G(1)}\,.
   \end{align*}
      If additionally, $\|\omega\|_{L^2(V)} \geq 1/2$, then the bound simplifies to
      \[\pi\left(\left\{x \in V : p_{2T}^G(x,x) < \frac{\delta \beta}{4K}\right\}\right) 
      \leq \beta + \delta + 3 \pi_G^*(\delta) +
      \frac{C \alpha T}{\delta^2 R^2} \|\omega\|_{L^2(V)}^2\ 
      \avgd(\delta/(K+R^2))\,.
   \]
\end{corollary}

\section{Conformal growth rates and random walks}

We will now apply the tools of the previous section
to establish our main claims on spectral dimension
and the diagonal heat kernel.
Toward this end, it will be convenient to start with
a unimodular random graph $(G,\rho)$ and derive from it a sequence $\{G_n\}$
of finite unimodular random graphs such that $\{G_n\} \todl (G,\rho)$.

\subsection{Invariant amenability and soficity}
\label{sec:amenable}

A unimodular random graph is called {\em sofic} if it is the distributional limit
of finite graphs.  It is an open question whether {\em every}
unimodular random graph is sofic (see, e.g., \cite[\S 10]{aldous-lyons}).
But as one might expect, for the proper definition of ``amenable,''
it turns out that all amenable graphs are sofic.

\paragraph{The invariant Cheeger constant}
A {\em percolation on $(G,\rho)$} is a $\{0,1\}$-marking
$\xi : E(G) \cup V(G) \to \{0,1\}$ of the edges and vertices
such that $(G,\rho,\xi)$ is unimodular as a marked graph.
One thinks of $\xi$ as specifying a (random) subgraph of $G$
corresponding to all the edges and vertices with $\xi = 1$.
One calls $\xi$ a {\em bond percolation} if $\xi(v)=1$ almost surely
for all $v \in V(G)$.
The {\em cluster of vertex $v$} is the connected component $K_{\xi}(v)$ of $v$
in the $\xi$-percolated graph.
Finally, one says that $\xi$ is {\em finitary} if
almost surely all its clusters are finite.

For a graph $G$ and a finite subset $W \subseteq V(G)$, we write $\partial^E_G W$ for
the {\em edge boundary of $W$:} The subset of edges $\partial^E_G W \subseteq E(G)$ that have exactly one endpoint in $W$.
The {\em invariant Cheeger constant} of a unimodular random graph $(G,\rho)$
is the quantity
\[
   \invcheeger(G,\rho) \seteq \inf \left\{ \E\left[\frac{|\partial^E_G K_{\xi}(\rho)|}{|K_{\xi}(\rho)|}\right]
   : \textrm{$\xi$ is a finitary percolation on $G$} \right\}\,.
\]
One says that $(G,\rho)$ is {\em invariantly amenable} if $\invcheeger(G,\rho)=0$.
Conversely, $(G,\rho)$ is {\em invariantly nonamenable} if it is not invariantly amenable.

\paragraph{Hyperfiniteness}
A unimodular random graph $(G,\rho)$ is {\em hyperfinite} if there is a sequence 
$\langle \xi_i\rangle_{i \geq 1}$ of percolations such that each $\xi_i$ is finitary,
$\xi_i \subseteq \xi_{i+1}$ almost surely, and almost surely $\bigcup_{i \geq 1} \xi_i = G$.
In this case, $\langle \xi_i\rangle_{i \geq 1}$ is called a {\em finitary exhaustion of $(G,\rho)$.}
One can consult \cite{AHNR18} for a proof of the following (stated without proof in \cite{aldous-lyons}).

\begin{theorem}[\cite{aldous-lyons}, Thm. 8.5]
   If $(G,\rho)$ is a unimodular random graph with $\E[\deg_G(\rho)] < \infty$, then
   $(G,\rho)$ is invariantly amenable if and only if it is hyperfinite.
\end{theorem}

The main point for us is that if $\langle \xi_i\rangle_{i \geq 1}$ is a finitary exhaustion
of $(G,\rho)$, then one has an approximation by finite unimodular random graphs:
\begin{equation}\label{eq:sofic-approx}
   \left\{G[K_{\xi_i}(\rho)] : i \geq 1\right\} \todl (G,\rho).
\end{equation}

To state the next corollary, let us recall that if $(G,\rho)$
is a unimodular random graph with law $\mu$, then
   \[
      \bar{d}_{\mu}(\e) \seteq \sup \left\{ \E\left[\deg_G(\rho) \mid \cE\right] : \Pr(\cE) \geq \e \right\}\,,
   \]
where the supremum is over all measurable sets $\cE$ with $\Pr(\cE) \geq \e$.

\begin{corollary}\label{cor:hyperfinite}
   If $(G,\rho)$ is a hyperfinite unimodular random graph, then there is a sequence $\{(G_n,\rho_n)\}$ of finite unimodular random graphs
   such that $\{(G_n,\rho_n)\} \todl (G,\rho)$ and moreover:
   \begin{enumerate}
      \item If $(G,\rho)$ is $\alpha$-decomposable, then for each $n \geq 1$, the unimodular random graph
         $(G_n,\rho_n)$ is $\alpha$-decomposable.
      \item For any $R \geq 1$ and any normalized metric $\omega$ on $(G,\rho)$, there is a sequence $\{\omega_n\}$
         of normalized metrics on $\{G_n\}$ such that for each $n \geq 1$,
         \begin{enumerate}
            \item Almost surely, $\|\omega_n\|_{L^2(V(G_n))}^2 \geq 1/2$.
            \item It holds that
               \[\|\# B_{\omega_n}(\rho_n, R/\sqrt{2})\|_{L^{\infty}} \leq \|\# B_{\omega}(\rho,R)\|_{L^{\infty}}\,.\]
   \end{enumerate}
   \end{enumerate}
\end{corollary}

\begin{proof}
   Property (1) follows from the definition of $\alpha$-decomposability.  Suppose $G$ is $\alpha$-decomposable;
   then so is $G[S]$ for every finite, connected subset $S \subseteq V(G)$, by simply extending any weight $\omega : S \to \R_+$
   to a weight $\hat{\omega} : V(G) \to \R_+$ defined by $\hat{\omega}(x) = \omega(x)$ if $x \in S$ and $\hat{\omega}(x) = \diam_{\omega}(S)$
   otherwise.  In this case, $\dist_{\hat{\omega}}|_{S \times S} = \dist_{\omega}$.

   Denote $\hat{\omega} = \sqrt{(\omega^2+1)/2}$.
   By assumption, $\hat{\omega}$ is normalized.
   The Mass-Transport Principle implies that (see, e.g., \cite[Lem. 3.1]{AHNR18}) if $\xi$ is
   finitary, then $\rho$ is uniformly distributed on its component $K_{\xi}(\rho)$, and therefore
   the unimodular random conformal graph $(G[K_{\xi}(\rho)], \hat{\omega}|_{K_{\xi}(\rho)}, \rho)$ is normalized as well.
   Moreover,
   \[
      \|\# B_{\omega_n}(\rho_n, R/\sqrt{2})\|_{L^{\infty}} \leq \|\# B_{\hat{\omega}}(\rho,R/\sqrt{2})\|_{L^{\infty}}
                                                         \leq \|\# B_{\omega}(\rho,R)\|_{L^{\infty}}\,,
   \]
   verifying property (2).
\end{proof}

\paragraph{Subexponential conformal growth and invariant amenability}

In conjunction with \pref{cor:hyperfinite}, the next result
will allow us to approximate a unimodular random graph with bounded
conformal growth exponent by a sequence of finite unimodular random graphs.

\begin{lemma}\label{lem:is-amenable}
   If $(G,\rho)$ is a unimodular random graph with $\E[\deg_G(\rho)^2] < \infty$
   and $\dimconfunder(G,\rho) < \infty$, then $(G,\rho)$ is invariantly amenable.
\end{lemma}

\begin{proof}
   Suppose that $(G,\rho)$ is invariantly nonamenable, $\omega$ is a normalized conformal metric on $(G,\rho)$.
   We will show that $\|B_{\omega}(\rho,R)\|_{L^{\infty}}$ grows at least exponentially in $R$,
   implying that $\dimconfunder(G,\rho)=\infty$.

   Let $h > 0$ be such that $\invcheeger(G,\rho) \geq h$.
   For some $K > 0$ to be specified soon,
   define a bond percolation $\xi : V(G) \to \{0,1\}$ by
   \[
      \xi(x) = \1_{\{\omega(x) \leq K\}}\,.
   \]
   For a vertex $x \in V(G)$, define $\deg_{\xi}(x)\seteq \sum_{y : \{x,y\} \in E(G)} \xi(y)$.

   From \cite[Thm. 8.13(i)]{aldous-lyons}, one concludes that if
   \begin{equation}\label{eq:the-rub}
      \E[\deg_{\xi}(\rho) \mid \xi(\rho)=1] > \E[\deg_G(\rho)] - h\,,
   \end{equation}
   then with positive probability, the subgraph $\{ x \in V(G) : \xi(x)=1 \}$ is nonamenable.
   Since a non-amenable subgraph has exponential growth and $\xi(x)=1 \implies \omega(x) \leq K$,
   we conclude that $\|B_{\omega}(\rho, R)\|_{L^{\infty}}$ grows (at least) exponentially as $R \to \infty$.
   We are thus left to verify \eqref{eq:the-rub} for some $K > 0$.

   Using Chebyshev's inequality and the fact that $\omega$ is normalized:
   \begin{equation}\label{eq:cheby2}
      \Pr[\xi(\rho) = 0] = \Pr[\omega(\rho) > K] \leq \frac{1}{K^2}\,.
   \end{equation}
   Now applying the Mass-Transport principle yields
   \begin{align}
      \E[\deg_G(\rho)-\deg_{\xi}(\rho)] &=\nonumber
      \E\left[\sum_{x : \{x,\rho\} \in E(G)} (1-\xi(x))\right] \\
      &= \E\left[\deg_G(\rho) (1-\xi(\rho))\right]\nonumber \\
      &\leq \sqrt{\E[\deg_G(\rho)^2]} 
      \sqrt{\Pr[\xi(\rho)=0]} \nonumber\\
      &\leq \frac{C}{K}\,,\label{eq:cheby3}
   \end{align}
   where $C \seteq (\E[\deg_G(\rho)^2])^{1/2}$.
   This gives
   \begin{align*}
   \E[\deg_{\xi}(\rho) \mid \xi(\rho)=1]
   &\geq \E[\deg_{\xi}(\rho)\xi(\rho)] \\
   &\geq \E[\deg_{\xi}(\rho)]-\E[\deg_G(\rho) (1-\xi(\rho))] \\
   &\stackrel{\mathclap{\eqref{eq:cheby3}}}{\geq} \E[\deg_{G}(\rho)] - \frac{2C}{K}\,.
   \end{align*}
   Choosing $K$ large enough verifies \eqref{eq:the-rub}, completing the proof.
\end{proof}

The preceding argument was suggested to us by Tom Hutchcroft,
replacing a considerably more complicated proof (a variant
of which appears in \pref{lem:rand-part} below).

\subsection{Conformal growth exponent bounds the spectral dimension}
\label{sec:conformals}

Let us now prove \pref{thm:spec-conf}. In fact, we will establish
the following quantitative strengthening.
If $(G,\rho)$ is a unimodular random graph with law $\mu$, let us define
\[
   \psi_{\mu}(R) \seteq \inf_{\omega} \|\# B_{\omega}(\rho,R)\|_{L^{\infty}},
\]
where the infimum is over all normalized conformal metrics $\omega$ on $(G,\rho)$.

  \begin{theorem}\label{thm:polylog}
     Suppose $(G,\rho)$ is a unimodular random graph and $\deg_G(\rho)$ has negligible tails.
     Assume furthermore that for some number
     $d > 0$ and an increasing sequence $\{R_n\}$ of radii, 
     it holds that
     \begin{equation}\label{eq:growth10}
        \psi_{\mu}(R_n) \leq R_n^{d+o(1)} \quad \textrm{as} \quad n \to \infty\,.
     \end{equation}
     Then for any sequence $\{\e_n\}$ with $\e_n \geq R_n^{-d-o(1)}$, there is a sequence of times $\{T_n\}$
     such that $T_n \geq \e_n^9 R_n^{2-o(1)}$, and
     \begin{equation}\label{eq:sp10}
         \Pr\left[p^G_{2 T_n}(\rho,\rho) \leq \frac{\e_n^6}{\psi_{\mu}(R_n)}\right] \leq \e_n^{1-o(1)} \quad \textrm{as} \quad n \to \infty\,.
      \end{equation}
      If $\deg_G(\rho)$ has exponential tails, then the corrections are polylogarithmic:
      There is a sequence $\{T_n\}$ satisfying $T_n \geq (\log R_n)^{-O(1)} \e_n^9 R_n^2$, and
      \[
         \Pr\left[p^G_{2 T_n}(\rho,\rho) \leq \frac{\e_n^6}{\psi_{\mu}(R_n)}\right] \leq \e_n\,(\log R_n)^{O(1)} \quad \textrm{as} \quad n \to \infty\,.
      \]
  \end{theorem}

  \begin{proof}[Proof of \pref{thm:spec-conf}]
     Choose $\e_n \seteq (\log R_n)^{-2}$.
     Then \eqref{eq:growth10} asserts the existence of a function $h(n) \leq T_n^{o(1)}$ such that
     \begin{equation}\label{eq:sp11}
         \Pr\left[p^G_{2 T_n}(\rho,\rho) \geq \frac{h(n)}{T_n^{d/2}}\right] \geq 1 - (\log R_n)^{-2+o(1)}\,.
      \end{equation}
     If $\dimconfunder({G}) \leq d$, then there is an unbounded sequence $\{R_n\}$ of radii satisfying \eqref{eq:growth10}.
     Thus there exists an unbounded sequence $\{T_n\}$ satisfying \eqref{eq:sp11}.
     It follows that almost surely (over the choice of $(G,\rho)$), there is an infinite subsequence
     $\{T_{n_j}\}$ such that $p^G_{2 T_{n_j}}(\rho,\rho) \geq h(n) T_{n_j}^{-d/2}$, implying that $\dimspecunder(G) \leq d$ as well.

  If $\dimconfover({G}) \leq d$, then one can take $\{R_n\}$ satisfying \eqref{eq:growth10} so that $\N \setminus \{R_n\}$ 
  is finite.
   Note that the even return times are monotone (see, e.g., \eqref{eq:ret-mono1}): For all integers $t \geq s \geq 1$,
   \begin{equation}\label{eq:ret-mono2}
      p^G_{2s}(\rho,\rho) \geq p^G_{2t}(\rho,\rho),
   \end{equation}
   hence we can assume that $T_n = R_n^2 g(n)$ for some $g(n) \leq R_n^{o(1)}$ as $n \to \infty$, and \eqref{eq:sp11} holds.
   We may similarly assume that $g(n+1) \leq 2 g(n)$, otherwise replace $g(n)$ by $\hat{g}(n) \seteq \min(g(n), 2 \hat{g}(n-1))$.

   In particular, since $\N \setminus \{R_n\}$ is unbounded, it holds that for some $h_1(n) \leq n^{o(1)}$,
   \[
      \Pr\left[p^G_{2\cdot 4^k}(\rho,\rho) \geq \frac{h_1(4^k)}{4^{kd/2}}\right] \geq 1 - k^{-2+o(1)} \quad \textrm{as} \quad k \to \infty.
   \]
   Now the Borel-Cantelli lemma asserts that almost surely $p^G_{2 \cdot 4^k}(\rho,\rho) \geq h_1(4^k)/4^{kd/2}$ for all but
   finitely many $k$.  Recalling again \eqref{eq:ret-mono2}, this gives $\dimspecover({G}) \leq d$, as desired.
  \end{proof}

   We record first some preliminary results.
   The following lemma is well-known; see, e.g., \cite[Lem. 3.11]{LN05}.
   Let $(X,\dist)$ be a metric space and consider $R > r > 0$.
   Let $\cC(X;R,r)$ denote the largest cardinality of a set $S \subseteq X$
   such that $x \neq y \in S \implies r \leq \dist(x,y) \leq R$.

   \begin{lemma}\label{lem:counting-padded}
      For any metric space $(X,\dist)$
      and $\tau > 0$, it holds that $(X,\dist)$ admits a $(\tau, \alpha)$-padded
      random partition for some $\alpha \leq O(\log \cC(X;2\tau,\tau/4))$.
   \end{lemma}

   \begin{lemma}\label{lem:efact}
      Suppose that $(G,\rho)$ is a {\em finite} unimodular random graph
      with stationary measure $\pi_G$, and
      such that $E[\deg_G(\rho)] < \infty$. 
      Then for any $\delta > 0$,
      \begin{equation}\label{eq:efact1}
         \pi_G^*(\delta) = \delta \frac{\avgd_{G}(\delta)}{\avgd_G(1)} \leq \delta \avgd_G(\delta)\,.
      \end{equation}
      (Recall the definition of $\pi_G^*$ from \pref{sec:retprob}.)
      Moreover, for any set of vertices $U_G \subseteq V(G)$, it holds that
      \[
         \Pr[\rho \in U_G \mid G]
         \leq \pi_G(U_G) \avgd_G(1)\,.
      \]
   \end{lemma}

   \begin{proof}
      The first fact follows directly from the definitions and $\avgd_G(1) \geq 1$
      since $G$ is almost surely connected.
      The second fact is a reformulation of $|U_G| \leq \sum_{x \in U_G} \deg_G(x)$.
   \end{proof}

   \begin{lemma}\label{lem:cg-prelims}
      Suppose that $(G,\rho)$ is a unimodular random graph with law $\mu$.
      Then there exists a sequence of finite, normalized unimodular random conformal graphs
      $\{(G_k,\omega_k,\rho_k)\}$ so that
      \begin{enumerate} 
         \item $\{(G_k,\rho_k)\} \todl (G,\rho)$.
         \item $\|\omega_k\|^2_{L^2(V(G_k))} \geq 1/2$ almost surely.
         \item For every $R > 0$, 
            \[
               \limsup_{k \to \infty} \|\# B_{\omega_k}(\rho_k,R)\|_{L^{\infty}} \leq \psi_{\mu}(\sqrt{2} R)\,.
            \]
         \item Almost surely over the choice of $(G_k,\rho_k)$:  For every $R > 0$,
            $(V(G_k), \dist_{\omega_k})$ admits 
            an $(R/2, \alpha_{R,k})$-padded random partition with
            \[
               \alpha_{R,k} \lesssim \log \|\# B_{\omega_k}(\rho_k, R)\|_{L^{\infty}}
            \]
   \end{enumerate}
   \end{lemma}

   \begin{proof}
      Combining \pref{lem:is-amenable} with \pref{cor:hyperfinite},
      we may take a sequence of finite normalized unimodular random conformal graphs $\{(G_k,\omega_k,\rho_k)\}$
      so that $\{(G_k,\rho_k)\} \todl (G,\rho)$ and (2) and (3) are satisfied.

      Note that \pref{lem:counting-padded} implies that
      $(V(G_k), \dist_{\omega_k})$ almost surely admits,
      for every $n \geq 1$, an $(R_n/2,\alpha_n)$-padded partition with
      \[
         \alpha_n \lesssim \log \|\# B_{\omega_k}(\rho_k,R_n)\|_{L^{\infty}}.\qedhere
      \]
   \end{proof}

   We will also require the following basic fact.

   \begin{lemma}\label{lem:av-deg}
      Suppose $G$ is a random finite graph and $\rho \in V(G)$ is chosen uniformly at random.
      Let $\mu$ be the law of the unimodular random graph $(G,\rho)$.
      Then for every $\e > 0$, it holds that
      \[
         \E\left[\avgd_G(\e)\right] \leq \avgd_{\mu}(\e/2)\,.
      \]
   \end{lemma}

   \begin{proof}
      Let $S_G^{\e}$ be the $\lfloor \e |V(G)|\rfloor$ vertices
      of largest degree in $G$, and define the event $\cE \seteq \{ \rho \in S_G^{\e} \}$.
      Then $\Pr(\cE \mid G) = \frac{|S_G^{\e}|}{|V(G)|} \geq \e/2$, hence $\Pr(\cE) \geq \e/2$.
      By definition, it follows that
      \[
         \avgd_{\mu}(\e/2) \geq \E\left[\frac{1}{|S_G^{\e}|} \sum_{x \in S_G^{\e}} \deg_G(x)\right] =
         \E\left[\avgd_G(\e)\right].\qedhere
      \]
   \end{proof}

   \begin{proof}[Proof of \pref{thm:polylog}]
      Let $\mu$ be the law of $(G,\rho)$.
      Apply \pref{lem:cg-prelims} 
      to obtain a sequence $\{(G_k,\omega_k,\rho_k)\}$ satisfying
      conclusions (1)--(4).   
      Define $N_{R,k} \seteq \|\#B_{\omega_k}(\rho_k, R)\|_{L^\infty}$ for $R,k \geq 1$,
      and fix some $\delta > 0$.
      Denote the event
      \begin{align*}
         \cQ_{R,k} \seteq 
         \left\{ \avgd_{G_k}(1) > \delta^{-1/3} \avgd_{\mu}(1/2) \right\} &\cup 
         \left\{ \avgd_{G_k}(\delta) > \delta^{-1/3} \avgd_{\mu}(\delta/2) \right\} \\
&\cup \left\{ \avgd_{G_k}(\delta/(N_{R,k}+R^2)) > \delta^{-1/3} \avgd_{\mu}(\delta/2(N_{R,k}+R^2)\right\}.
      \end{align*}

     Condition on $(G_k,\omega_k,\rho_k)$, and let
      $\pi_k$ denote the stationary measure on $G_k$.
      Apply \pref{cor:on-diag}
      with $\alpha=\alpha_{R,k}$, $K=N_{R,k}$ to obtain for any $T \geq 1$,
   \[
      \pi_k\left(\left\{x \in V(G_k) : p_{2T}^{G_k}(x,x) < \frac{\delta^2}{4N_{R,k}}\right\}\right) 
      \lesssim \delta + \pi_{G_k}^*(\delta)+
      \frac{\alpha_{R,k}^2}{\delta^2} \frac{T}{R^2} \|\omega_k\|_{L^2(V(G_k))}^2
      \avgd_{G_k}\!\left(\delta/(N_{R,k}+R^2)\right)\,.
   \]
   Observe that $\pi^*_{G_k}(\delta) \leq \delta \avgd_{G_k}(\delta)$ and
   use \pref{lem:efact} to change from the stationary to uniform measure, yielding
   \begin{equation*}
      \frac{\# \left\{ x \in V(G_k) : p_{2T}^{G_k}(x,x) < \frac{\delta^2}{4N_{R,k}}\right\}}{|V(G_k)|}
      \lesssim \avgd_{G_k}(1)\left(\delta (1+\avgd_{G_k}(\delta)) +
      \frac{\alpha_{R,k}^2}{\delta^2} \frac{T}{R^2} \|\omega_k\|_{L^2(V(G_k))}^2
   \avgd_G(\delta/(N_{R,k}+R^2))\right).
   \end{equation*}
   Taking expectation over $(G_k,\omega_k,\rho_k)$ and using that $\omega_k$ is normalized yields
   \begin{align}
      \Pr\left(p_{2T}^{G_k}(\rho_k,\rho_k) < \frac{\delta^2}{4N_{R,k}}\right) 
      &\lesssim \avgd_{\mu}(1/2) 
      \left(\delta^{2/3} (1+\delta^{-1/3} \avgd_{\mu}(\delta/2)) + 
      \frac{\alpha_{R,k}^2}{\delta^{8/3}} \frac{T}{R^2} \avgd_{\mu}(\delta/2(N_{R,k}+R^2))\right) + \Pr[\cQ_{R,k}] \nonumber \\
      &\lesssim
      \delta^{2/3} (1+\delta^{-1/3} \avgd_{\mu}(\delta/2)) + 
      \frac{\alpha_{R,k}^2}{\delta^{8/3}} \frac{T}{R^2}\avgd_{\mu}(\delta/2(N_{R,k}+R^2)) + \Pr[\cQ_{R,k}], \nonumber
\label{eq:abstep2} 
   \end{align}
   where the last inequality uses $\avgd_{\mu}(1/2) \leq O(1)$, since $\deg_G(\rho)$ has negligible tails.

   Let us now take $k \to \infty$ and use the fact that $\{(G_k,\rho_k)\} \todl (G,\rho)$.
   Since $\avgd_{\mu_k}(\e) \to \avgd_{\mu}(\e)$ for every $\e > 0$, \pref{lem:av-deg}
   and Markov's inequality show that $\limsup_{k \to \infty} \Pr[\cQ_{R,k}] \leq 3 \delta^{1/3}$.
   Using additionally \pref{lem:cg-prelims}(3)--(4) gives, for all $R,T \geq 1$,
   \begin{equation}\label{eq:return10}
      \Pr\left[p_{2T}^{G}(\rho,\rho) < \frac{\delta^2}{4 \psi_{\mu}(\sqrt{2} R)}\right] 
      \lesssim 
      \delta^{2/3} (1+\delta^{-1/3} \avgd_{\mu}(\delta))
      +
      \frac{\alpha_R^2}{\delta^{8/3}} \frac{T}{R^2} 
      \avgd_{\mu}(\delta/(\psi_{\mu}(\sqrt{2} R)+R^2)) + \delta^{1/3}\,,
   \end{equation}
   where $\alpha_R \seteq \log \psi_{\mu}(\sqrt{2} R)$.

   Let $\{R_n\}$ be an increasing sequence of radii satisfying \eqref{eq:growth10},
   and $\{\e_n\}$ a sequence satisfying $\e_n \geq R_n^{-d+o(1)}$.
   The assumption that $\deg_G(\rho)$ has negligible tails (recall \eqref{eq:negligible-tails})
   yields
   \[
      \avgd_{\mu}(\beta) \leq \beta^{-o(1)} \quad \textrm{as} \quad \beta \to 0\,,
   \]
   hence applying \eqref{eq:return10} with $R=R_n/\sqrt{2}$ and $\delta = 2 \e_n^3$ gives
   for all $T \geq 1$, as $n \to \infty$:
   \[
      \Pr\left[p_{2T}^{G}(\rho,\rho) < \frac{\e_{n}^6}{\psi_{\mu}(R_n)}\right] 
      \leq
      \e_n g(n)
      + \e_n^{-8} \frac{T}{R_n^{2}} h(n)
   \]
   where $g(n) \leq \e_n^{-o(1)}$ and $h(n) \leq R_n^{o(1)}$ as $n\to \infty$.
   If $\deg_G(\rho)$ additionally has exponential tails, one has $g(n) \leq O(\log (1/\e_n)) \leq O(\log R_n)$ and
   $h(n) \leq (\log R_n)^{O(1)}$.

   If we now define, for $n \geq 1$,
   \begin{equation}\label{eq:tndef}
      T_n \seteq \left\lfloor\e_n^9 \frac{R_n^2}{h(n)}\right\rfloor,
   \end{equation}
   we arrive at
   \[
      \Pr\left[p_{2T_n}^{G}(\rho,\rho) < \frac{\e_n^6}{\psi_{\mu}(R_n)}\right]  \leq \e_n (1+g(n)) \quad \textrm{as} \quad n \to \infty\,,
   \]
   yielding \eqref{eq:sp10}.
\end{proof}

\subsection{On-diagonal heat kernel bounds}
\label{sec:limits}

Our goal now is to prove \pref{thm:heat-kernel} and \pref{thm:green-diverge}.
We start with the former and restate it here for ease of reference.

\begin{theorem}[Restatement of \pref{thm:heat-kernel}]
      Suppose that $(G,\rho)$ satisfies the
      the conditions:
      \begin{enumerate}
         \item $(G,\rho)$ has gauged quadratic conformal growth and is uniformly decomposable,
         \item $\E[\deg_G(\rho)^2] < \infty$.
      \end{enumerate}
      Then there is a constant $C=C(\mu)$ such that
      for every $\delta > 0$ and all $T \geq C/\delta^{2}$,
      \begin{equation}\label{eq:heat-goal-1}
         \Pr\left[p_{2T}^G(\rho,\rho) < \frac{\delta}{T \bar{d}_{\mu}(1/T^3)}\right] \leq C \delta^{1/17}\,.
      \end{equation}
   \end{theorem}

\begin{proof}
   Let $C_{\mu} = \E[\deg_G(\rho)^2]$.
   Then for $\e > 0$,
   \begin{equation}\label{eq:chebyshevy}
      C_{\mu} \geq \e \bar{d}_{\mu}(\e)^2\,.
   \end{equation}

   From \pref{cor:hyperfinite},
   we can take a sequence $\{(G_n,\rho_n)\} \todl (G,\rho)$ 
   such that $(G_n,\rho_n)$ is a finite unimodular random graph that is almost surely:
   \begin{enumerate}
      \item $\alpha$-decomposable,
      \item $(\kappa,R)$-quadratic for every $R \geq 1$,
   \end{enumerate}
   where $\alpha,\kappa > 0$ are some constants depending on $\mu$.

   Consider $\delta > 0$, $T \geq 1$, and $n \geq 1$.
   Let $R \seteq \smashed{\gamma T \bar{d}_{\mu}(1/T^3)}$ for some number $\gamma > 0$ to be chosen soon.
   Recall from \pref{cor:hyperfinite} that we may assume
   that $\|\omega_n\|^2_{L^2(V(G_n))} \geq 1/2$
   almost surely.

   Set $K=\kappa R^2$ and apply \pref{cor:on-diag} with $\beta = \sqrt{\delta}$
   to obtain, for some constant $C_1=C_1(\alpha,\kappa)$,
   almost surely over the choice of $(G_n,\omega_n,\rho_n)$:
   \[
      \pi_{G_n}\left(\left\{x \in V : p_{2T}^{G_n}(x,x) < \frac{\delta^{3/2}}{4K}\right\}\right) 
      \lesssim \sqrt{\delta} + \pi_{G_n}^*(\delta)+
      \frac{C_1}{\delta^2 \gamma} \|\omega_n\|_{L^2(V_n)}^2
   \frac{\avgd_{G_n}\left(\frac{\delta}{C_1\gamma T \avgd_{\mu}(1/T^3)}\right)}{\avgd_{\mu}(1/T^3)}
   \]
   Observe that
   \begin{equation}\label{eq:obsthat}
      \avgd_{G_n}\left(\frac{\delta}{C_1 \gamma T \bar{d}_{\mu}(1/T^3)}\right)
      \stackrel{\eqref{eq:chebyshevy}}{\leq} C_1 \avgd_{G_n}\left(\frac{\delta}{C_1 C_{\mu} \gamma T^{2.5}}\right).
      \end{equation}
   From \pref{lem:efact}, we have
   \begin{equation}\label{eq:obspi}
      \pi_{G_n}^*(\delta) \leq \delta \bar{d}_{\mu}(\delta)
      \stackrel{\eqref{eq:chebyshevy}}{\leq} 2\delta \sqrt{\frac{C_{\mu}}{\delta}} = 2 \sqrt{C_{\mu} \delta}\,.
   \end{equation}
   Using \eqref{eq:obsthat} and \eqref{eq:obspi}, along with \pref{lem:efact} to change from the stationary measure to the uniform measure,
   gives
   \[
      \frac{\# \left\{x \in V : p_{2T}^{G_n}(x,x) < \frac{\delta^{3/2}}{4K}\right\}}{|V(G_n)|}
      \lesssim \avgd_{G_n}(1) \left(
         \left(1+ 2\smashed{C_{\mu}}\right) \sqrt{\delta} +
      \frac{C_1}{\delta^2 \gamma} \|\omega_n\|_{L^2(V_n)}^2
   \frac{\avgd_{G_n}\left(\frac{\delta}{C_1 C_{\mu} \gamma T^{2.5}}\right)}{\avgd_{\mu}(1/T^3)}
\right).
   \]

   Define the event
   \[
      \cQ_n \seteq 
      \left\{ \avgd_{G_n}(\eta) > \delta^{-1/4} \avgd_{\mu}(\eta/2) \right\} \cup
      \left\{ \avgd_{G_n}(1) > \delta^{-1/4} \avgd_{\mu}(1/2) \right \},
   \]
   where $\eta \seteq \frac{\delta}{C_1 C_{\mu} \gamma T^{2.5}}$.
   Taking expectation over $(G_n,\omega_n,\rho_n)$ gives
   \[
      \Pr\left(p_{2T}^{G_n}(\rho_n,\rho_n) < \frac{\delta^{3/2}}{4K}\right) 
      \leq O(\avgd_{\mu}(1/2) \delta^{-1/4})
      \left( 
         \left(1+ 2\smashed{C_{\mu}}\right) \sqrt{\delta} +
         \frac{C_1}{\delta^{9/4} \gamma} 
   \frac{\avgd_{\mu}\left(\eta/2\right)}{\avgd_{\mu}(1/T^3)}\right) + \Pr[\cQ_n].
   \]

   Now set $\gamma \seteq \delta^{-11/4}$ and take $C_2=C_2(\mu)$
   sufficiently large so that for $T \geq C_2/\delta^8$, we have $\eta \geq 2 T^{-3}$, and
   \[
      \Pr\left(p_{2T}^{G_n}(\rho_n,\rho_n) < \frac{\delta^{3/2}}{4K}\right) 
      \leq C_2 \delta^{1/4} + \Pr[\cQ_n].
   \]
   Recalling that $K = \kappa R^2 = \kappa \gamma T \avgd_{\mu}(1/T^3)$, this gives
   \[
      \Pr\left(p_{2T}^{G_n}(\rho_n,\rho_n) < \frac{\delta^{17/4}}{T \avgd_{\mu}(1/T^3)}\right) 
      \leq C_2 \delta^{1/4} + \Pr[\cQ_n].
   \]

   Let $\mu_n$ denote the law of $(G_n,\rho_n)$.
   Since $\{(G_n,\rho_n)\} \todl (G,\rho)$, it holds that $\avgd_{\mu_n}(\e) \to \avgd_{\mu}(\e)$ for every $\e > 0$,
   hence Markov's inequality in conjunction with \pref{lem:av-deg} 
   gives $\lim_{n \to \infty} \Pr[\cQ_n] \leq 2 \delta^{1/4}$, yielding
   \[
      \Pr\left(p_{2T}^{G}(\rho,\rho) < \frac{\delta^{17/4}}{T \avgd_{\mu}(1/T^3)}\right) 
      \leq (2+C_2) \delta^{1/4}.\qedhere
   \]
\end{proof}

Now we move on to the proof of \pref{thm:green-diverge}.

\begin{proof}[Proof of \pref{thm:green-diverge}]
   Let $d_t = \bar{d}_{\mu}(1/t)$ and observe that since $\{d_t\}$ is monotone increasing,
   \[
      \sum_{t \geq 1} \frac{1}{t d_t} \geq
      \sum_{t \geq 1} \frac{1}{t d_{t^3}} \geq
      \frac18 \sum_{k \geq 0} \frac{1}{d_{2^{3k}}} \geq \frac1{48} \sum_{k \geq 0} \frac{1}{d_{2^k}}
      \geq \frac{1}{96} \sum_{t \geq 1} \frac{1}{t d_t}\,.
   \]
   Define $c_t \seteq \frac{1}{t \bar{d}_{\mu}(1/t^3)}$.
   From the preceding inequalities, it suffices
   to consider $\hat{g}(T) = \sum_{t=1}^T c_{t}$ in place of $g(T)$.

   Let $C=C(\mu)$ be the constant from \eqref{eq:heat-goal-1}.  Fix $\delta > 0$.
   For $N \geq 1$, let $T_N = \min \{ T : \hat{g}(T) \geq N \}$.
   Choose $N(\delta)$ large enough so that for $N \geq N(\delta)$, we have $T_N \geq C/\delta^{2}$ and
   \[
      N \leq \hat{g}(T_N) \leq (1 + \delta)N\,.
   \]

   Define the random variable
   \[
      Z_N = \sum_{t=1}^{T_N} \min\left\{p^G_{2t}(\rho,\rho), \delta c_t \right\}\,.
   \]
   Then by definition, $Z_N \leq \delta \hat{g}(T_N)$, and \eqref{eq:heat-goal-1}
   implies that $\E[Z_N] \geq \delta (1-C \delta^{1/8}) \hat{g}(T_N)$, hence
   \[\Pr\left[Z_N \geq  \frac{\delta}{2} \hat{g}(T_N)\right] \geq 1-2C\delta^{1/8}\,.\]
   Define the sequence $\{Y_N\}$ by
   \[
      Y_N \seteq \frac{1}{N} \sum_{n \leq N} \1_{\{Z_n \geq \frac{\delta}{2} \hat{g}(T_n)\}}\,.
   \]
   Since $0 \leq Y_N \leq 1$ almost surely,
   Fatou's Lemma yields
   \begin{equation}\label{eq:fatou}
      \Pr\left[\limsup_{N \to \infty} Y_N > 0\right] \geq   \E\left[\limsup_{N \to \infty} Y_N\right] \geq
      \limsup_{N \to \infty} \E\left[Y_N\right] \geq 1 - 2C \delta^{1/8}\,.
   \end{equation}
   By construction, this implies
   \[
      \Pr\left[\limsup_{T \to \infty} \frac{\sum_{t=1}^T p_{2t}^G(\rho,\rho)}{g_{\mu}(T)} > 0\right] \geq 1-2C \delta^{1/8}\,.
   \]
   Now send $\delta \to 0$, concluding the proof.
\end{proof}

\subsubsection{Fatter degree tails and transience}
\label{sec:transient-example}

We generalize the example from \cite[\S 1.3]{GN13}.

\begin{lemma}
   For every monotonically increasing sequence $\{d_t : t=1,2,\ldots\}$ of positive integers such that
   $\sum_{t \geq 1} \frac{1}{t d_t} < \infty$,
   there is a unimodular random planar graph $(G,\rho)$ with law $\mu$ such that
   for all $t$ sufficiently large,
   \begin{equation}\label{eq:fat-dom}
      \bar{d}_{\mu}(1/t) \leq d_t\,,
   \end{equation}
   $\E[\deg_G(\rho)^2] < \infty$, and $G$ is almost surely transient.
\end{lemma}

\begin{proof}
   We may we replace $d_t$ by $\min(d_t, t^{1/4})$ so that $d_t \leq t^{1/4}$.
   We may also assume that $d_{2t} \leq 2 d_t$ for all $t \geq 1$
   without
   affecting convergence of the sequence $\sum_{t \geq 1} \frac{1}{t d_t}$.

Consider an increasing function $f : \N \to \N$.
Let $\cT_n$ be a complete binary tree of height $n$ and replace
each edge at height $k=1,2,\ldots,n$ from the leaves by $f(k)$ parallel edges
(at the end of the proof, we indicate how to convert the
construction into a simple graph).

Let $(\cT,\rho)$ be the distributional limit of $\{\cT_n\}$, and let
$\mu$ be the law of $(\cT,\rho)$.
If $f(k) \leq 2^{o(k)}$ as $k \to \infty$, then the distributional limit exists
and, moreover, the unique path to infinity is the one moving away from the leaves.
Thus almost surely,
\begin{equation}\label{eq:reff-small}
   \reff^{\cT}(\rho \leftrightarrow \infty) \leq \sum_{k=1}^{\infty} \frac{1}{f(k)}\,.
\end{equation}

Let us now define $f(k) \seteq 2 d_{2^{k}}-d_{2^{k+1}}$ so that
\begin{equation}\label{eq:fd-comp}
   d_{2^k} = 2^{k} \sum_{j=k+1}^{\infty} \left(2^{-j+1} d_{2^{j}}- 2^{-j} d_{2^{j+1}}\right) = \sum_{j=1}^{\infty} f(k+j) 2^{-j}\,,
\end{equation}
where convergence of the telescopic sum follows from our assumption that $d_{2^j} \leq 2^{j/4}$.
Now observe if $t > 2^{k-1}$, then
\[
   \bar{d}_{\mu}(1/t) \leq \sum_{j=1}^{\infty} f(k+j) 2^{-j} = d_{2^k}\,,
\]
hence \eqref{eq:fat-dom} is satisfied.

We also have:
\[
   \sum_{k=1}^{\infty} \frac{1}{f(k)} \stackrel{\eqref{eq:fd-comp}}{\lesssim}  \sum_{k=1}^{\infty} \frac{1}{d_{2^{k}}} \leq 2 \sum_{t=1}^{\infty} \frac{1}{t d_t} < \infty\,.
\]
From \eqref{eq:reff-small}, this implies that almost surely $\cT$ is transient.
Finally, note that $\E[\deg_G(\rho)^2] \leq 2 \sum_{k \geq 1} 2^{-k} (d_{2^k})^2 < \infty$ since we assumed that $d_{2^k} \leq 2^{k/4}$.

We may replace every parallel edge by a path of length two while affecting the degree distribution
only by a factor of $2$ (and one can rescale $f$ accordingly to maintain property \eqref{eq:fat-dom}).
\end{proof}

\subsection{Spectrally heterogeneous graphs}
\label{sec:homo}

There are unimodular random graphs $(G,\rho)$
with $\deg_G(\rho) \leq O(1)$ and $\dimspec(G) \leq O(1)$ almost surely,
but $\dimconf(G,\rho) = \infty$.
Indeed, there exist invariantly nonamenable graphs $(G,\rho)$ for which $\dimspecover(G) \leq O(1)$
almost surely.  This is asserted in
\cite{AHNR18}.

\medskip
\noindent
{\bf An invariantly nonamenable graph with bounded spectral dimension.}
We recall the construction alluded to there.
Fix a value $\alpha > 0$.
Let $\cT$ denote the infinite $3$-regular tree
and fix a vertex $v_0 \in V(\cT)$.
To each $v \in V(\cT)$, we attach a random path of length $L_v$,
where the random variables $\{L_v : v \in V(\cT)\}$ are independent and satisfy,
for $\ell \geq 1$:
\[
      \Pr(L_v = \ell) =
         \begin{cases}
            c_1 (\ell+1)\ell^{-2-\alpha} & v = v_0 \\
            c_2 \ell^{-2-\alpha} & v \in V(\cT) \setminus \{v_0\}\,,
         \end{cases}
\]
where $c_1, c_2 > 0$ are the unique values that give rise
to probability measures.  Let $\cG$ denote the resulting graph, and
let $P_0$ denote the path attached to $v_0$ (so that $P_0$
contains $L_{v_0}+1$ vertices).

It is not difficult to verify that $(\cG,\rho)$
is unimodular when $\rho \in V(P_0)$ is
chosen uniformly at random.  Indeed, $(\cG,\rho)$
is the distributional limit of finite graphs with uniformly random roots.
To see this, consider a sequence $\{G_n\}$ of $3$-regular graphs
with girth tending to infinity (see, e.g., \cite{Imrich84}).
Let $\cG_n$ denote the random graph in which we attach to every
vertex of $G_n$ a path of length $L_v$, where $\{L_v : v \in V(G_n)\}$
is a family of independent random variables
with law $\Pr(L_v=\ell) = c_2 \ell^{-2-\alpha}$ for $\ell \geq 1$.
Then $\cG_n$ is almost surely finite.
Let $u_n \in V(G_n)$ be
chosen uniformly at random.

\begin{claim}
  $\{(\cG_n,u_n)\} \todl (\cG,\rho)$.
\end{claim}

\begin{proof}
Let $\{ P_v : v \in V(G_n) \}$ be the collection of attached paths.
Calculate:
\begin{align*}
   \Pr(L_v = \ell \mid u_n \in P_v) &= \frac{\Pr(L_v = \ell)}{\Pr(u_n \in P_v)} \Pr(u_n \in L_v \mid L_v = \ell) \\
                                    &= c_2 \ell^{-2-\alpha} |V(G_n)| \E\left[\frac{\ell+1}{|V(\cG_n)|} \bigmid L_v = \ell\right].
\end{align*}
Note that the law of $|V(\cG_n)|$ conditioned on $\{L_v=\ell\}$ is $(\ell+1)+\sum_{u \in V(G_n) \setminus \{v\}} L_u$.
Thus
by the law of large numbers, it holds that
\[
   \lim_{n \to \infty} \Pr(L_v = \ell \mid u_n \in P_v) = 
   c_2 (\ell+1) \ell^{-2-\alpha}\lim_{n \to \infty} \frac{|V(G_n)|}{\E[L_v] |V(G_n)|}  = c_1 (\ell+1) \ell^{-2-\alpha}.\qedhere
\]
\end{proof}

An interesting feature of $(\cG,\rho)$ is that the mean return probability
is dominated by a small set of vertices (of measure $\approx T^{-\alpha/2}$):
\[
   \E[p_{2T}^\cG(\rho,\rho)] \approx \Pr[L_{v_0} \geq \sqrt{T}] \cdot \frac{1}{\sqrt{T}} \approx T^{-(1+\alpha)/2}\,.
\]

It turns out that one can obtain polynomial conformal volume growth
if they are willing to ignore the
``spectrally insignificant'' vertices.
Moreover, if $(G,\rho)$ is spectrally homogeneous in a strong sense (\eqref{eq:homo2} below),
one can reverse the bound of \pref{thm:spec-conf} and obtain
$\dimconfover(G,\rho) \leq \textrm{a.s.-}\dimspecover(G)$.

Consider a monotone non-decreasing function $h : \R_+ \to \R_+$ such that $h(n) \leq n^{o(1)}$ as $n \to \infty$
and a number $d > 0$.
For $T \geq 1$, define the set of vertices
with $d$-dimensional lower bounds on the diagonal heat kernel:
\[
   H_G(T) \seteq \left\{ x \in V(G) : p_{2T}^G(x,x) \geq \deg_G(x) \frac{T^{-d/2}}{h(T)}\right\}\,.
\]
Define also, for $R \geq 0$,
\[
   \widehat{H}_G(R) \seteq H_G\left(R^2 h(R)^4 (\log R)^4\right)\,.
\]

\begin{theorem}\label{thm:homo}
   Let $(G,\rho)$ be a unimodular random graph and
   suppose that for all $T \geq 1$,
   \begin{equation}\label{eq:eass}
      \E\left[p_{2T}^G(\rho,\rho)\right] \leq h(T) T^{-d/2}\,.
   \end{equation}
   Then there is a normalized conformal metric $\omega : V(G) \to \R_+$ such that
   \begin{equation}\label{eq:homo1}
      \left\|\1_{\widehat{H}_G(R)}(\rho)\cdot \#\left(B_{\omega}(\rho, R) \cap \widehat{H}_G(R)\right)\right\|_{L^{\infty}} \leq R^{d+o(1)}\qquad \textrm{as } R \to \infty\,.
   \end{equation}
   Moreover, if it holds that
   \begin{equation}\label{eq:homo2}
      \Pr[\rho \notin H_G(T)] \leq \frac{h(T)}{T}\quad \textrm{for all $T \geq 1$},
   \end{equation}
   then $\dimconfover(G,\rho) \leq d$.
\end{theorem}

In order to prove \pref{thm:homo}, we will need
to recall some background on the spectral measures
of infinite graphs.

\subsubsection{Spectral measures on infinite graphs}

Fix a connected, locally-finite graph $G$.  We use $\ell^2(G)$ for the Hilbert space of real-valued
functions $f : V(G) \to \R$ equipped with the inner product
\[
   \langle f,g\rangle_{\ell^2(G)} = \sum_{x \in V(G)} \deg_G(x) f(x) g(x)\,.
\]

For a graph $G$, define the averaging operator $P_G : \ell^2(G) \to \ell^2(G)$ by
\begin{align*}
   P_G \psi(u) &\seteq \frac{1}{\deg_G(u)} \sum_{v : \{u,v\} \in E(G)} \psi(v)\,.
\end{align*}
Observe that $P_G$ is self-adjoint:
\begin{align*}
   \langle \f, P_G \psi\rangle_{\ell^2(G)} &= \sum_{x \in V(G)} \deg_G(x) \f(x) \frac{1}{\deg_G(x)} \sum_{y : \{x,y\} \in E(G)} \psi(y) 
   = 2 \sum_{\{x,y\} \in E(G)} \f(x) \psi(y)\,.
\end{align*}
Since $P_G$ is an averaging operator, it is also bounded, and therefore the spectral theorem yields
a resolution of the identity $I_{P_G}$ so that $P_G = \int_{\R} \lambda d I_{P_G}(\lambda)$.

Given a vertex $v \in V(G)$, one can define
the associated {\em spectral measure $\mu_G^v$ at $v$} by
\[
   \mu_G^v((-\infty,\lambda)) = \deg_G(v)^{-1} \langle \1_v, I_{P_G}((-\infty,\lambda)) \1_v\rangle_{\ell^2(G)}\,.
\]
This is
the unique probability measure $\mu_G^v$
on $\R$ such that
\begin{equation}\label{eq:spec-meas}
   \deg_G(v) \int_{[-1,1]} \lambda^T\,d\mu_G^{v}(\lambda) = \langle \1_v, P_G^T \1_v\rangle_{\ell^2(G)}
\end{equation}
for all integers $T \geq 1$.

Let us record a few additional equalities.
Fix $\rho \in V(G)$.
Then by self-adjointness, for any $T \geq 1$, we have:
\begin{align}\nonumber
   \frac{\|P_G^T \1_{\rho}\|^2_{\ell^2(G)}}{\deg_G(\rho)} - \sum_{x \sim \rho} \frac{\langle P_G^{T} \1_x, P_G^{T} \1_{\rho}\rangle_{\ell^2(G)}}{\deg_G(\rho) \deg_G(x)}
   &= \frac{\langle \1_{\rho}, (I-P_G)P_G^{2T} \1_{\rho}\rangle_{\ell^2(G)}}{\deg_G(\rho)} \\ &= \int (1-\lambda) \lambda^{2T} d\mu_G^{\rho}(\lambda)\,,
\label{eq:spec1}
\end{align}
where we have used $P_G \1_{\rho} = \sum_{x \sim \rho} \frac{\1_x}{\deg_G(x)}$.  Moreover,
\begin{align}
  \sum_{x \sim \rho} \left\|\frac{P_G^T \1_{\rho}}{\deg_G(\rho)} - \frac{P_G^T \1_x}{\deg_G(x)}\right\|_{\ell^2(G)}^2 
  = \frac{\|P_G^T \1_{\rho}\|^2_{\ell^2(G)}}{\deg_G(\rho)} +
   \sum_{x \sim \rho} \frac{\|P_G^T \1_{x}\|_{\ell^2(G)}^2}{\deg_G(x)^2}
   - 2\sum_{x \sim \rho} \frac{\langle P_G^{T} \1_x, P_G^{T} \1_{\rho}\rangle_{\ell^2(G)}}{\deg_G(\rho) \deg_G(x)}.
\label{eq:spec2}
\end{align}

For any $x \in V(G)$ and integer $T \geq 0$, we have
\begin{equation}\label{eq:spec-return}
   \|P_G^T \1_x\|_{\ell^2(G)}^2 = \langle \1_x, P_G^{2T} \1_x\rangle_{\ell^2(G)} = \deg_G(x) \cdot p^G_{2T}(x,x)\,.
\end{equation}
Moreover, observe that $\{\1_x/\sqrt{\smash[b]{\deg_G(x)}} : x \in V(G)\}$ forms an orthornormal basis for $\ell^2(G)$, hence
\begin{equation}\label{eq:isotropic}
   \sum_{x \in V(G)} \frac{\langle P_G^T \1_x, P_G^T \1_{\rho}\rangle_{\ell^2(G)}^2}{\deg_G(x)}
   = \sum_{x \in V(G)} \frac{\langle \1_x, P_G^{2T} \1_{\rho}\rangle_{\ell^2(G)}^2}{\deg_G(x)} = \|P_G^{2T} \1_{\rho}\|_{\ell^2(G)}^2
\end{equation}
Note also that since $P_G$ is a Markov operator, it is a contraction on $\ell^2(G)$, hence for all integers $T \geq 1$,
\begin{equation}\label{eq:ret-mono1}
   p^G_{2T}(x,x) = \deg_G(x) \|P_G^T \1_x\|_{\ell^2(G)}^2 \geq \deg_G(x) \|P_G^{2T} \1_x\|_{\ell^2(G)}^2 = p^G_{4T}(x,x)\,.
\end{equation}

\paragraph{The heat kernel embedding and growth rates}
\label{sec:heatkernel}

Suppose that $G$ is a connected, locally finite graph,
and let $d > 0$ and $h : \R_+ \to \R_+$ be as in \pref{sec:homo}.
Let us define $\F^G_T : V(G) \to \ell^2(G)$ by
\[
   \F^G_{T}(x) \seteq \frac{P_G^T \1_x}{\deg_G(x)}\,.
\]

The heat-kernel embedding is closely related to the spectral embedding
which can be described as follows. 
Then for $\delta > 0$, the spectral embedding $\Psi^G_{\delta} : V(G) \to \ell^2(G)$
is given by
\[
   \Psi^G_{\delta}(x) \seteq \frac{I_{P_G}([1-\delta,1]) \1_x}{\sqrt{\deg_G(x)}}\,.
\]
Clearly these two embeddings are closely related for $T \asymp \frac{1}{\delta}$.
The geometry of the spectral embedding has been used in work
on higher-order Cheeger inequalities \cite{LOT14} and in connection
with return probabilities \cite{LO18}.

For $x \in V(G)$, also define the set of points
that are closer to $\F^G_T(x)$ than the origin in the
heat kernel embedding:
\[
   \cC_T^G(x) \seteq \left\{ y \in V(G) : \|\Phi_T^G(x)-\Phi_T^G(y)\|_{\ell^2(G)} \leq \|\Phi_T^G(y)\|_{\ell^2(G)} \right\}\,.
\]

The next lemma gives a relationship between return probabilities
and the size of $\cC_T^G(x)$.
This is inspired by the ``mass spreading'' property of the
spectral embedding employed in \cite{LOT14}.
For a subset $S \subseteq V(G)$, we will use the notation $\deg_G(S) = \sum_{x \in S} \deg_G(x)$.

\begin{lemma}\label{lem:spreading}
   For any $x \in V(G)$, it holds that
   \[
      \deg_G\left(\cC_T^G(x)\right) \leq \frac{4\,\deg_G(x)}{p_{2T}^G(x,x)}\,.
   \]
\end{lemma}

\jnote{Alternate:

\begin{lemma}
   For any $x \in V(G)$, it holds that
   \[
      |\cC_T^G(x)| \leq \frac{4 p_{4T}^G(x,x)^2}{p_{2T}^G(x,x)^4}\,.
   \]
\end{lemma}

}

\begin{proof}
   Note that $\langle u,v\rangle = \frac12 \left(\|u\|^2 + \|v\|^2 - \|u-v\|^2\right) \geq \tfrac12 \|u\|^2$ whenever $\|u-v\|\leq \|v\|$.
   Employ this in conjunction with \eqref{eq:isotropic} to write and 
   \[
      \|\F^G_{2T}(\rho)\|_{\ell^2(G)}^2 = \sum_{x \in V(G)} \deg_G(x)\,\langle \F^G_{T}(x), \F^G_{T}(\rho)\rangle_{\ell^2(G)}^2
      \geq \deg\left(\cC_T^G(\rho)\right) \frac{\|\F^G_{T}(\rho)\|_{\ell^2(G)}^4}{4}\,.
   \]
   To finish, use the fact that $P_G$ is a contraction on $\ell^2(G)$:  $\|\F^G_{2T}(\rho)\|_{\ell^2(G)} \leq \|\F^G_{T}(\rho)\|_{\ell^2(G)}$, hence
   \[
      \deg_G\left(\cC_T^G(\rho)\right) \leq \frac{4}{\|\F^G_T(\rho)\|_{\ell^2(G)}^2} \stackrel{\eqref{eq:spec-return}}{=}
      \frac{4\,\deg_G(\rho)}{p_{2T}^G(\rho,\rho)}\,.\qedhere
   \]
\end{proof}

\subsubsection{Spectrally significant vertices in unimodular random graphs}

If $(G,\rho)$ is a random rooted graph such that $P_G$ is almost surely self-adjoint,
one defines the {\em spectral measure of $(G,\rho)$} by
\[
   \mu \seteq \E \left[\mu_G^{\rho}\right]\,.
\]

Let $(G,\rho)$ be a unimodular random graph with spectral measure $\mu$.
Taking expectations in \eqref{eq:spec2} gives
\begin{align*}\label{eq:lap}
   \E\left[\sum_{x \sim \rho} \left\|\Phi_T^G(x)-\Phi_T^G(\rho)\right\|_{\ell^2(G)}^2\right] 
   &= \E\left[\sum_{x \sim \rho} \left\|\frac{P_G^T \1_{\rho}}{\deg_G(\rho)}
   - \frac{P_G^T \1_x}{\deg_G(x)}\right\|_{\ell^2(G)}^2\right] \\
   &=\E\left[\frac{\|P_G^T \1_{\rho}\|^2_{\ell^2(G)}}{\deg_G(\rho)} +
   \sum_{x \sim \rho} \frac{\|P_G^T \1_{x}\|_{\ell^2(G)}^2}{\deg_G(x)^2}
- 2\sum_{x \sim \rho} \frac{\langle P_G^{T} \1_x, P_G^{T} \1_{\rho}\rangle_{\ell^2(G)}}{\deg_G(\rho) \deg_G(x)}\right]
\end{align*}
The Mass-Transport Principle applied to the functional
$F(G,\rho,x) \seteq \frac{\1_{E(G)}(\{x,\rho\})}{\deg_G(x)^2} \|P^T_G \1_x\|_{\ell^2(G)}^2$
gives
\[
   \E\left[\sum_{x \sim \rho} \frac{\|P_G^T \1_{x}\|_{\ell^2(G)}^2}{\deg_G(x)^2}\right]
   = 
   \E\left[\frac{\|P_G^T \1_{\rho}\|^2_{\ell^2(G)}}{\deg_G(\rho)}\right],
\]
so that applying \eqref{eq:spec1}
shows that for all $T \geq 1$,
\begin{equation}\label{eq:lap}
   \E\left[\sum_{x \sim \rho} \left\|\Phi_T^G(x)-\Phi_T^G(\rho)\right\|_{\ell^2(G)}^2\right]
= 2 \int (1-\lambda) \lambda^{2T} d\mu(\lambda)\,.
\end{equation}

Consider some $d > 0$ and
split the latter integral into two pieces, depending on whether $\lambda \leq 1-\frac{(d+1) \log T}{T}$:
\begin{align}
   \int (1-\lambda) \lambda^{2T} d\mu(\lambda) \nonumber
   &\leq T^{-d-1} + \frac{(d+1) \log T}{T} \int \lambda^{2T} d\mu(\lambda) \\
   &\stackrel{\mathclap{\eqref{eq:spec-meas}}}{=}\ T^{-d-1} + \frac{(d+1) \log T}{T} \E[p^G_{2T}(\rho,\rho)]\,.
   \label{eq:int-bnd}
\end{align}

The remainder of this section is devoted to the proof of \pref{thm:homo}.

\begin{proof}[Proof of \pref{thm:homo}]
   For $k \geq 1$, define the weights $\omega_k : V(G) \to \R_+$ by
   \begin{align*}
      \omega_k(x) &\seteq \sqrt{\sum_{y : \{x,y\} \in E(G)} \|\Phi_{2^k}^G(x)-\Phi_{2^k}^G(y)\|_{\ell^2(G)}^2}.
   \end{align*}
   Under assumption \eqref{eq:eass}, we can employ \eqref{eq:lap} and \eqref{eq:int-bnd} to write
   \[
      \E \omega_k(\rho)^2 \leq 2^{-k(d+1)} + \frac{(d+1) k}{2^k} 2^{-kd/2} h(2^k)\,.
   \]
   Define now
   \[
      \omega \seteq \sqrt{\sum_{k \geq 1} \frac{2^{k(1+d/2)}}{k^3 h(2^k)} \omega_k^2 }
   \]
   so that
   \[
      \E \omega(\rho)^2 \lesssim \sum_{k \geq 1} \frac{1}{k^2} \lesssim 1\,.
   \]

   By construction, we have, for every $k \geq 1$ and $x,y \in V(G)$,
   \begin{equation}\label{eq:distcompare}
      \dist_{\omega}(x,y)^2 \geq \frac{2^{k(1+d/2)}}{k^{3} h(2^k)} \|\Phi_{2^k}^G(x)-\Phi_{2^k}^G(y)\|_{\ell^2(G)}^2\,.
   \end{equation}

\begin{lemma}
   For all $k \geq 1$, if $x \in H_G(2^k)$, then
   \begin{equation*}\label{eq:intermed}
      \left|B_{\omega}\left(x, \frac{2^{k/2}}{h(2^k) k^{3/2}}\right) \cap H_G(2^k)\right| \leq |\cC^G_{2^{k}}(x)| \leq h(2^k) 2^{kd/2+2}\,,
   \end{equation*}
\end{lemma}

\begin{proof}
   From \eqref{eq:distcompare},
   \begin{align*}
      y \in B_{\omega}\left(x, \frac{2^{k/2}}{h(2^k) k^{3/2}}\right) &\implies \|\Phi_{2^k}^G(x)-\Phi_{2^k}^G(y)\|_{\ell^2(G)}^2
      \leq \frac{2^{-kd/2}}{h(2^k)}\,.
   \end{align*}
   By definition: $y \in H_G(2^k) \implies p_{2^{k+1}}^G(y,y) \geq \deg_G(y) \frac{2^{-kd/2}}{h(2^k)}$.
   Therefore:
   \begin{align*}
      y \in B_{\omega}\left(x, \frac{2^{k/2}}{h(2^k) k^{3/2}}\right) \cap H_G(2^k) &\implies \|\Phi_{2^k}^G(x)-\Phi_{2^k}^G(y)\|_{\ell^2(G)}^2
      \leq \frac{p_{2^{k+1}}^G(y,y)}{\deg_G(y)} \stackrel{\eqref{eq:spec-return}}{=} \|\Phi^G_{2^k}(y)\|_{\ell^2(G)}^2,
   \end{align*}
   which yields $y \in \cC_{2^k}^G(x)$.
   Now \pref{lem:spreading} gives
   \[
      |\cC_{2^k}^G(x)| \leq \frac{4\,\deg_G(x)}{p_{2^{k+1}}^G(x,x)} \leq 4 h(2^k) 2^{kd/2},
   \]
   where the last inequality uses $x \in H_G(2^k)$.
\end{proof}

\begin{corollary}\label{cor:H-growth}
   For all $R$ sufficiently large, if $x \in \widehat{H}_G(R)$, then
   \[
      \left|B_{\omega}\left(x, R\right) \cap \widehat{H}_G(R)\right| \leq R^{d+o(1)}\,.
   \]
\end{corollary}

This confirms \eqref{eq:homo1}.
To verify that $\dimconfover(G,\rho) \leq d$ under \eqref{eq:homo2},
we define
\[
   \hat{\omega}_{k} = \frac{2^k}{\sqrt{h\left(4^k h(2^k)^4 k^4\right)}} \1_{V(G) \setminus \widehat{H}_G(2^k)}\,.
\]
Observe that from \eqref{eq:homo2},
\[
   \E \hat{\omega}_k(\rho)^2 = \frac{4^k}{h(4^k h(2^k)^4 k^4)} \Pr\left[\rho \notin \widehat{H}_G(2^k)\right] \lesssim 1\,.
\]
Define
\[
   \hat{\omega} \seteq \sqrt{\sum_{k \geq 1} \frac{\hat{\omega}_k^2}{k^2}}
\]
so that $\E \hat{\omega}(\rho)^2 \lesssim 1$.
Finally, note that for any $k \geq 1$:
\[
   x\notin \widehat{H}_G(2^k) \implies
   B_{\hat{\omega}}\left(x, 2^k/\sqrt{h(4^k h(2^k) k^4)}\right) = \{x\}\,.
\]
Thus taking the final weight $\omega_0 = \sqrt{\omega^2 + \hat{\omega}^2}$ verifies
that $\dimconfover(G,\rho) \leq d$.
\end{proof}

\jnote{
\subsubsection{Hyperbolic vs. parabolic}

Suppose that almost surely
\[
   \sum_{T \geq 1} T \cdot p_{2T}^G(\rho,\rho) = \infty\,.
\]
Does it follow that we can bound the conformal growth rate of $(G,\rho)$?

   For $k \geq 1$, define the conformal metric $\omega_k : V(G) \to \R_+$ by
   \[
      \omega_k(x) \seteq \sqrt{\sum_{y : \{x,y\} \in E(G)} \|\Phi_{2^k}^G(x)-\Phi_{2^k}^G(y)\|_{\ell^2(G)}^2}
   \]

   Then
   \[
      \E \omega_k(\rho)^2 \leq \frac{\beta_{2T} \log (T/\beta_{2T})}{T}
   \]
   So let us define
   \[
      \omega = \sqrt{\sum_{k \geq 1} \frac{2^k}{k^3 \beta_{2^{k+1}}} \omega_k^2}
   \]

   Then we get
   \[
      \omega_k \geq \frac{2^{k/2}}{k^{3/2} \sqrt{\E[p_{2^{k+1}}^G(\rho,\rho)]}} \dist_{\Phi_{2^k}^G}\,.
   \]
   We know that
   \[
      \left|C_{2^k}^G(x)\right| \leq \frac{4}{p_{2^{k+1}}^G(x,x)}
   \]
   And this cone contains all the points of
   \[
      \left|B_{\omega}\left(x, \frac{\delta_k 2^{k/2}}{k^{3/2}} \right) \cap
      \left\{ y \in V(G) : p_{2^{k+1}}^G(y,y) \geq \delta_k^2 \E[p_{2^{k+1}}^G(\rho,\rho)]\right\}\right| \leq \frac{4}{p_{2^{k+1}}^G(x,x)}\,.
   \]
   So we are interested in showing that
   \[
      \Pr\left(p_{2^{k+1}}^G(\rho,\rho) < \delta_k^2 \E\left[p_{2^{k+1}}^G(\rho,\rho)\right]\right) \delta_k^2 2^k \leq O(1)\,,
   \]

   In general, consider
   \[
      \max \left\{ \frac{\int (1-\lambda) \lambda^{2T} d\mu(\lambda)}{\int \lambda^{2T} d\mu(\lambda)} : \int \lambda^{2T} d\mu(\lambda) \geq \frac{1}{T^2}\right\}\,.
   \]
   The whole contribution could come from $\lambda \approx 1-\frac{\log T}{T}$.
}

\section{Markov type and speed of the random walk}
\label{sec:mt-speed}

We will now address the speed of the random walk on unimodular random graphs.
Our approach is to first establish that the random walk is at most diffusive
under any normalized conformal metric, and our main tool will be
the theory of Markov type.
We then use separators to construct conformal metrics
with H\"older-type comparisons to the graph metric,
allowing us to establish subdiffusive speed in certain settings.

\subsection{Diffusive bounds in the conformal metric}
\label{sec:diffusive}

In this section, it will be helpful to think about {\em reversible} random graphs.
Suppose $(G,\rho)$ is a random rooted graph and let $\{X_t\}$ denote
the random walk on $G$ (conditioned on $(G,\rho)$) with $X_0=\rho$.  Then {\em $(G,\rho)$ is reversible}
if we have the identity of laws:
\[
   (G,X_0,X_1) \stackrel{\textrm{law}}{=} (G,X_1,X_0)\,.
\]

The following lemma is from \cite[Prop. 2.5]{BC12}.

\begin{lemma}\label{lem:bc12}
   There is a correspondence betwen unimodular random graphs with $\E[\deg_G(\rho)] < \infty$
   and reversible random graphs:
   $(G,\rho)$ is unimodular if and only if $(\tilde{G},\tilde{\rho})$ is reversible,
   where $(\tilde{G},\tilde{\rho})$ has the law of $(G,\rho)$ biased by $\deg_G(\rho)$.
\end{lemma}

We first prove a general result for the case when
a reversible conformal random graph $(G,\omega,\rho)$
is such that the Markov type 2 constant of the metric
space $(V(G),\dist_{\omega})$ is essentially bounded.
For instance, by \cite{DLP13}, this is true when $G$ is almost surely planar.
Then we move on to graphs of annealed polynomial growth.
One should recall the definition of the Markov type $2$ constant
$M_2$ from \pref{def:markov-type}.

\begin{theorem}
   Suppose that $(G,\rho)$ is an invariantly amenable (cf. \pref{sec:amenable})
   reversible random graph.
   Then for any conformal metric $\omega$ on $(G,\rho)$, the following holds:
   For all $T \geq 1$,
   \[
      \E\left[\dist_{\omega}(X_0,X_T)^2 \mid X_0=\rho\right] \leq T \|M_2(V(G),\dist_{\omega})\|_{L^{\infty}}^2
      \cdot \E[\omega(\rho)^2]\,.
   \]
\end{theorem}

   \begin{proof}
      Let $\{\xi_j : j \geq 1\}$ denote a finitary exhaustion $(G,\rho)$.
      Then \pref{lem:bc12} gives the variant of \eqref{eq:sofic-approx} for
      reversible random graphs:
      \begin{equation*}\label{eq:fconverge} 
         \left\{\left(G[K_{\xi_j}(\rho)],\rho\right) : j\geq 1\right\} \todl (G,\rho)\,.
      \end{equation*}
      In particular, we have
      \begin{equation}\label{eq:dlocstar}
         \{(G[K_{\xi_j}(\rho)], \rho_j)\} \todl (G,\rho)\,,
      \end{equation}
      where $\rho_j$ is distributed according
      to the stationary measure on $G[K_{\xi_j}(\rho)]$.

      Let $\{X^j_t\}$ denote the random walk conditioned on $G[K_{\xi_j}(\rho)]$, where $X^j_0$
      has the law of the stationary measure on $G[K_{\xi_j}(\rho)]$.
If we let $M \seteq \|M_2(V(G),\dist_{\omega})\|_{L^{\infty}}$, then
the definition of Markov type yields,
for any $j \geq 1$:
      \[
         \E\left[\dist_{\omega}(X^j_0, X^j_T)^2 \mid G[K_{\xi_j}(\rho)]\right] \leq T M^2 \E\left[\dist_{\omega}(X^j_0,X^j_1)^2 \mid G[K_{\xi_j}(\rho)]\right]\,.
      \]
      Recall that if $\{x,y\} \in E(G)$, then $\dist_{\omega}(x,y) \leq \frac12 (\omega(x)^2+\omega(y)^2)$,
      hence by stationarity, the latter quantity is bounded by
      \[
         \E\left[\dist_{\omega}(X^j_0,X^j_1)^2 \mid G[K_{\xi_j}(\rho)]\right] 
         \leq \E\left[\omega(X^j_0)^2 \mid G[K_{\xi_j}(\rho)\right].
      \]
      Taking expectation over $G[K_{\xi_j}(\rho)]$ yields
      \[
         \E\left[\dist_{\omega}(X^j_0, X^j_T)^2\right] \leq T M^2 \E[\omega(X_0^j)^2]\,.
      \]

      If we replace the weight $\omega$ by $\tilde{\omega} \seteq \min(\omega, \tau)$ for some $\tau > 0$, this argument yields
      \[
         \E\left[\dist_{\tilde{\omega}}(X^j_0, X^j_T)^2\right] \leq T M^2 \E[\tilde{\omega}(X_0^j)^2]\,.
      \]

      Since $\tilde{\omega}$ is essentially bounded, taking a limit as $j \to \infty$ yields
      \[
         \E\left[\dist_{\tilde{\omega}}(X_0, X_T)^2 \mid X_0=\rho\right] \leq T M^2 \E[\tilde{\omega}(\rho)^2] \leq T M^2 \E[\omega(\rho)^2].
      \]
      Now taking the truncation parameter $\tau \to \infty$ yields the desired result.
\end{proof} 

In order to prove a similar result for reversible random graphs
of polynomial growth, we need a result about the Markov type 2 constants
of finite metric spaces.
It is an immediate consequence
of the following facts:
(1) Hilbert spaces have Markov type 2 with constant 1 \cite{Ball92},
(2) Markov type is a bi-Lispchitz invariant,
(3) every $n$-point metric space embeds into a Hilbert space with $O(\log n)$
bi-Lipschitz distortion \cite{Bourgain85}.

\begin{theorem}\label{thm:mt-finite}
   If $(X,d)$ is an $n$-point metric space with $n \geq 2$, then $M_2(X,d) \leq O(\log n)$.
\end{theorem}

The next lemma will allow us to choose a (unimodular) finitary exhaustion
whose sets have controlled diameters.

\begin{lemma}\label{lem:bdd-folner}
   Suppose that $(G,\rho)$ is a unimodular random graph.
   Then there is a sequence of bond percolations $\langle \xi_j : j\geq 1\rangle$
   with the following properties for every $j \geq 1$:
   \begin{enumerate}
      \item $(G,\rho, \langle \xi_j : j \geq 1\rangle)$ is unimodular as a marked network.
      \item Almost surely, $\diam_G(K_{\xi_j}(\rho)) \leq 2 j$.
      \item For every $r > 0$, it holds that
\[
	\Pr[B_G(\rho, r) \nsubseteq K_{\xi_j}(\rho)] \leq \frac{1}{|B_G(\rho,j)|^2}
																		+ \frac{12r}{j} \log |B_G(\rho,3j)|\,.
\]
   \end{enumerate}
\end{lemma}

\begin{proof}
Fix a parameter $R \geq 1$ and
let $\{R_x : x \in V(G)\}$ denote a sequence of independent
exponential random variables where $R_x$ has mean
\begin{equation}\label{eq:means}
	\frac{R}{3 \log |B_G(x,2 R)|}\,.
\end{equation}
Let $\{\beta_x \in [0,1] : x \in V(G)\}$ denote an independent family
of i.i.d. uniform random variables.

Let $\hat{R}_x \seteq \min(R_x, R)$ and define for every $x \in V(G)$:
\[
   \Theta_x \seteq \left\{ y \in B_G(x,R) : \hat{R}_y \geq \dist_G(x,y) \right\}\,.
\]
Let $\theta(x) \in \Theta_x$ be the vertex $y \in \Theta_x$ with $\beta_y$ minimal.
By construction, $\{\theta^{-1}(y): y \in V(G)\}$ is almost surely
a partition of $V(G)$ and $\theta^{-1}(y) \subseteq B_G(y,R)$ for every $y \in V(G)$.

Define a bond percolation $\xi_R : E(G) \to \{0,1\}$ on $G$ by $\xi_R(\{u,v\})=\1_{\{\theta(u)=\theta(v)\}}$.
By construction, every cluster $K_{\xi_R}(x)$ is
of the form $\theta^{-1}(y)$ for some $y \in V(G)$, and thus
$\diam_G(K_{\xi_R}(x)) \leq 2R$ for every $x \in V(G)$.

The marked network $(G,\rho,\xi_R)$ is unimodular because the law of $\xi_R$
depends only on the isomorphism class of $G$.
Let $\nu_G$ be the law of $\xi_R$ given $G$.
Then for any functional $F(G,x,y,\xi)$, one can define
\[
   \hat{F}(G,x,y) \seteq \E_{\nu_G} \left[F(G,x,y,\xi_R)\right],
\]
and then apply the Mass-Transport Principle to $\hat{F}$.
The next lemma verifies (3) and completes the proof.

\begin{lemma}\label{lem:rand-part}
	For every $x \in V(G)$, the following holds:
\[
	\Pr\left[B_G(x, r) \nsubseteq K_{\xi_R}(x)\right] \leq \frac{1}{|B_G(x,R)|^2} + 
\frac{12r}{R} \log |B_G(x,3R)|\,.
\]
\end{lemma}

\begin{proof}
	Denote the event $\cE = \{ \exists y \in B_G(x,R) : R_y > R\}$.
	Note that $y \in B_G(x,R) \implies |B_G(y, 2R)| \geq |B_G(x,R)|$, and thus
reviewing the mean of each $R_y$ in \eqref{eq:means}, a union bound yields
\[
	\Pr[\cE] \leq e^{-3 \log |B_G(x,R)|} |B_G(x,R)| \leq \frac{1}{|B_G(x,R)|^2}\,.
\]

Define the set
\[
   U_x \seteq \{ y \in B_G(x,R) : R_y \geq \dist_G(x,y) - r \}\,.
\] 
For $y \in B_G(x,R)$,
let $\cE'_y$ denote the event $\{R_y > \dist_G(x,y) + r \}$, and note that by
the memoryless property of the exponential distribution,
\begin{align} \Pr\left[\lnot \cE'_y \mid U_x, y \in U_x\right] &= 
   \Pr\left[\lnot \cE'_y \mid R_y \geq \dist_G(x,y) -r \right] \nonumber \\ &=
\frac{6 \log |B_G(y, 2 R)|}{R} \int_0^{2r} \exp\left(\frac{-6t \log |B_G(y, 2R)|}{R}\right)\,dt \nonumber
\\
&\leq
\frac{12r}{R} \log |B_G(y_*,2R)|
\leq \frac{12r}{R} \log |B_G(x,3R)|\,.\label{eq:memoryless}
\end{align}

Now let $y_* \in V(G)$ denote the vertex
with $\beta_{y_*}$ minimal in the set $U_x$.
Note that $U_x$ is always non-empty since $x \in U_x$.
Moreover,
\[
   \cE'_{y_*} \wedge \lnot \cE \implies B_G(x,r) \subseteq B_G(y_*, R_{y_*}) \implies B_G(x,r) \subseteq K_{\xi_R}(x)\,,
\]
where the second implication follows from $\beta_{y^*} \leq \min \{ \beta_z : z \in \bigcup_{v \in B_G(x,r)} \Theta_v \}$.
Since $y_*$ is independent of $\{ R_{y} : y \in U_x\}$ conditioned on $U_x$, \eqref{eq:memoryless} gives
\[
   \Pr\left[\lnot \cE'_{y_*}\right] \leq \frac{12r}{R} \log |B_G(x,3R)|,
\]
completing the proof.
\end{proof}
\end{proof}

We will soon define a finite-state Markov chain on the cluster $K_{\xi}(\rho)$;
the following definition will be helpful.

\begin{definition}[Restricted random walk] \label{def:restricted}
Consider a graph $G=(V,E)$ and a finite subset $S \subseteq V$.
Let \[N_G(x) = \{ y \in V : \{x,y\} \in E \}\]
denote the neighborhood of a vertex $x \in V$.

Extend $\deg_G$ to a measure on subsets $S\subseteq V(G)$ in the
obvious way:  $\deg_G(S) \seteq \sum_{x \in S} \deg_G(x)$.
Define a measure $\pi_S$ on $S$ by
\begin{equation}\label{eq:piS}
   \pi_S(x) \seteq \frac{\deg_G(x)}{\deg_G(S)} \1_{S}(x)\,.
\end{equation}
We define {\em the random walk restricted to $S$} as the following process $\{Z_t\}$: For $t \geq 0$, put \[ \Pr(Z_{t+1} = y \mid Z_t = x) = \begin{cases} \frac{|N_G(x) \setminus S|}{\deg_G(x)} & y = x \\ \frac{1}{\deg_G(x)} & y \in N_G(x) \cap S \\ 0 & \textrm{otherwise.} \end{cases} \] It is straightforward to check that $\{Z_t\}$ is a reversible Markov chain on $S$ with stationary measure $\pi_S$. If $Z_0$ has law $\pi_S$, we say that $\{Z_t\}$ is the {\em stationary random walk restricted to $S$.}
\end{definition}

We now prove the following theorem; it immediately yields \pref{thm:unimodular-markov-type-intro}
(since the assumption of the latter theorem implies that $\deg_G(\rho)$
is essentially bounded).

\begin{theorem}\label{thm:unimodular-markov-type}
   Suppose that $(G,\rho)$ is a
   random rooted graph
   and for some constants $C,q \geq 1$,
   it holds that
   \begin{equation}\label{eq:bgrow}
      \E |B_G(\rho,r)| \leq C r^q\qquad \forall r \geq 1\,.
   \end{equation}
   Then:
   \begin{enumerate}
      \item
   If $(G,\rho)$ is reversible, then
   for any conformal metric $\omega$ on $(G,\rho)$:
   For any $q' \geq 1$ and $T \geq 2$,
   \begin{equation}\label{eq:umt2}
      \E\left[T^{2q'} \wedge \dist_{\omega}(X_0,X_T)^2 \mid X_0=\rho\right]
      \leq C' T(\log T)^2  \cdot \E[\omega(\rho)^2] + 1\,,
   \end{equation}
   where $C'=C'(C,q,q')$ is number depending only on $C,q,q'$.
\item
   If $(G,\rho)$ is unimodular and
   \begin{equation}\label{eq:exp-tails}
      \Pr[\deg_G(\rho) > \lambda] \leq e^{\lambda/C} \qquad \forall \lambda \geq 1\,,
   \end{equation}
   then for any conformal metric $\omega$ on $(G,\rho)$:
   For any $q' \geq 1$ and $T \geq 2$,
   \begin{equation}\label{eq:umt3}
      \E\left[T^{2q'} \wedge \dist_{\omega}(X_0,X_T)^2 \mid X_0=\rho\right]
      \leq C' T(\log T)^4  \cdot \E[\omega(\rho)^2] + 1\,,
   \end{equation}
   where $C'=C'(C,q,q')$.
   
   If, additionally, $G$ is almost surely planar, then the bound improves to
   \begin{equation}\label{eq:umt4}
      \E\left[T^{2q'} \wedge \dist_{\omega}(X_0,X_T)^2 \mid X_0=\rho\right]
      \leq C' T (\log T)^2 \cdot \E[\omega(\rho)^2] + 1.
   \end{equation}
   \end{enumerate}
\end{theorem}

   \begin{proof}
      Let $\langle \xi_j : j \geq 1\rangle$ denote the sequence of bond percolations
      guaranteed by \pref{lem:bdd-folner}.
      Note that \pref{lem:bdd-folner}(1) assures that each $\xi_j$ is almost surely
      finitary.
      Assume first that $(G,\rho)$ is reversible.
      Then the (degree-biased)
      mass-transport principle shows that if we choose $\hat{\rho}$
      according to the measure $\pi_{K_{\xi_j}(\rho)}$ (recall \eqref{eq:piS}), then
      $(G,\rho)$ and $(G,\hat{\rho})$
      have the same law.

      Let $\{Z^j_t\}$ denote the stationary random walk restricted to $K_{\xi_j}(\rho)$
      (conditioned on $(G,\rho), \xi_j)$.
      Therefore \pref{thm:mt-finite} yields, for any $T,j \geq 1$:
      \begin{equation}\label{eq:fdb}
      \E\left[\dist_{\omega}(Z^j_0, Z^j_T)^2 \mid (G,\rho),\xi_j\right] 
      \lesssim T (\log |K_{\xi_j}(\rho)|)^2 \E\left[\dist_{\omega}(Z^j_0,Z^j_1)^2 \mid (G,\rho),\xi_j\right].
   \end{equation}

   Define the events
   \begin{align*}
      \cE_T &\seteq \left\{ |K_{\xi_j}(\rho)| \leq 2 T^{2 q'} \E |B_G(\rho,j)| \right\}, \\
      \cB_T &\seteq \left\{ B_G(\hat{\rho},T) \subseteq K_{\xi_j}(\rho) \right\},
   \end{align*}
   and let $\{X_t\}$ denote the random walk on $G$.
   Observe now that for $j \geq 2$:
   \begin{align*}
      \E&\left[\dist_{\omega}(X_0,X_T)^2 \1_{\cB_T} \1_{\cE_T}
      \mid X_0=\hat{\rho}, (G,\rho),\xi_j\right]
      \\ &=
      \E\left[\dist_{\omega}(Z^j_0,Z^j_T)^2 \1_{\cB_T} \1_{\cE_T}
      \mid Z^j_0=\hat{\rho}, (G,\rho),\xi_j\right] \\
      &\leq
      \E\left[\dist_{\omega}(Z^j_0,Z^j_T)^2 \1_{\cE_T}
      \mid Z^j_0=\hat{\rho}, (G,\rho),\xi_j\right] \\
      &\lesssim
      T (q' \log T + \log (\E |B_G(\rho,j)|))^2 \E[\omega(\hat{\rho})^2 \mid (G,\rho),\xi_j]\,.
   \end{align*}
   Taking expectations and using \eqref{eq:bgrow} yields, for $T \geq 2$:
   \[
      \E\left[\dist_{\omega}(X_0,X_T)^2 \1_{\cB_T}\1_{\cE_T}
      \mid X_0=\rho\right]
      \lesssim
      C q T (q'\log T + \log j)^2 \E[\omega(\rho)^2]\,.
   \]

   Using assumptions \eqref{eq:bgrow} and \pref{lem:bdd-folner}(2),
   we can choose $j \leq T^{O(q')}$ so that
   \[
      \Pr\left[\cB_T \textrm{ and } \cE_T\right]
      \geq 1 - \frac{1}{T^{2q'}}\,,
   \]
   verifying \eqref{eq:umt2}.

   Let us now assume that $(G,\rho)$ is unimodular and verify \eqref{eq:umt3}.
   In this case, if we choose $\hat{\rho} \in K_{\xi_j}(\rho)$ uniformly
   at random, then the mass-transport principle shows that
   $(G,\rho)$ and $(G,\hat{\rho})$ have the same law.

   Return momentarily to \eqref{eq:fdb} and observe that we
   can replace the stationary measure with the uniform measure on
   both the left and right, losing two factors of $\dmax(K_{\xi_j}(\rho))$:
   \begin{equation*}\label{eq:fdb2}
      \E\left[\dist_{\omega}(Z^j_0, Z^j_T)^2 \mid Z^j_0=\hat{\rho}, (G,\rho),\xi_j\right] 
      \lesssim (\dmax(K_{\xi_j}(\rho)))^2 (\log |K_{\xi_j}(\rho)|)^2 \cdot
      T \E\left[\omega(\hat{\rho})^2 \mid (G,\rho),\xi_j\right].
   \end{equation*}
   Define instead
   \[
      \cE_T \seteq \left\{ |K_{\xi_j}(\rho)| \leq 2 T^{2q'} \E |B_G(\rho,j)| \wedge
      \dmax(K_{\xi_j}(\rho)) \leq c' \log T\right\}\,,
   \]
   Using \eqref{eq:exp-tails}, along with 
   assumptions \eqref{eq:bgrow} and \pref{lem:bdd-folner}(2),
   we can choose $c' \leq O(q')$ and $j \leq T^{O(q')}$ so that
   \[
      \Pr[\cB_T \textrm{ and } \cE_T] \geq 1-\frac{1}{T^{2q'}}\,,
   \]
   and the proof is finished as before.

   Finally, to verify \eqref{eq:umt4}, note that when $G$ is planar,
   for any conformal metric $\omega : V(G) \to \R_+$,
   it holds that $M_2(V(G),\dist_{\omega}) \leq O(1)$ by the results
   of \cite{DLP13} which establishes that planar graphs metrics
   have Markov type 2 with a uniform constant.
   Appealing to this fact instead of \pref{thm:mt-finite}
   yields the desired improvement.
\end{proof} 

\subsection{Weights from separators}
\label{sec:separators}

We now turn to the proofs of \pref{lem:barriers-intro} and \pref{thm:uipq-speed}.
For a graph $G$, $x \in V(G)$, and $r' > r > 0$, define
\[
   q_G(x; r, r') \seteq \max \left\{ \sum_{y \in S} \frac{1}{|B_G(y,r)|}
   : S \subseteq B_G(x,r'), |S|=\kappa_G(\rho; r, r') \right\}\,.
\]
Note that under the assumptions of \pref{thm:subd-intro-3}, we have
almost surely:
\[
   q_G(\rho;r,h(r) r) \leq r^{k-1-d+o(1)} \qquad \forall r\geq 1\,.
\]
For a given $s \geq 1$, define
\[
   r(s) \seteq \sup \{ r : h(r) r \leq s \}.
\]
Since $h(r) \leq r^{o(1)}$ by assumption, we have $s \leq r(s)^{1+o(1)}$.
Applying the following lemma with $r=r(s/2)$ and $r'=s/2$ yields \pref{lem:barriers-intro}.

\begin{lemma}\label{lem:separators-gen}
   Suppose $(G,\rho)$ is a unimodular random graph.
   Then for every $r' > r > 0$, there is a subset $W_{r,r'} \subseteq V(G)$
   such that $(G,\rho,W_{r,r'})$ is unimodular as a marked network, and
   the following hold:
   \begin{enumerate}
      \item $\Pr[\rho \in W_{r,r'}] \leq \E[q_G(\rho; r,r')]$.
      \item Every connected component of $G[V(G) \setminus W_{r,r'}]$ has
         diameter at most $2r'$ in $\dist_G$.
   \end{enumerate}
\end{lemma}

\begin{proof}
   Fix $r' > r \geq 1$.
   For each $x \in V(G)$, let $U_x \subseteq B_G(x,r') \setminus B_G(x,r)$
   denote a separator achieving $\kappa_G(x;r,r')$.
   Let $\{\beta_x \in [0,1] : x \in V(G)\}$ be a collection of i.i.d.
   uniform random variables.
   Define, for every $x \in V(G)$, the set
   \[
      \hat{U}_x \seteq U_x \setminus \bigcup_{y : \beta_y < \beta_x} B_G(y,r)\,,
   \]
   and
   \[
      W_{r,r'} \seteq \bigcup_{x \in V(G)} \hat{U}_x\,.
   \]

   To see that almost surely every connected component of $G[V(G) \setminus W_{r,r'}]$ has
   diameter at most $2r'$ in $\dist_G$, consider $x,y \in V(G)$ with $\dist_G(x,y) > 2r'$.
   Let $\gamma$ be a simple path from $x$ to $y$ in $G$.
   We will show that almost surely $\gamma \cap W_{r,r'} \neq \emptyset$.
   Let $z \in V(G)$ be the vertex in the
   finite set $\{z : \gamma \cap B_G(z,r) \neq \emptyset\}$ with $\beta_z$ minimal.
   Since $\dist_G(x,y) > 2r'$, it cannot be that both $x,y \in B_G(z,r')$,
   hence it must be that $\gamma \cap U_z \neq \emptyset$.
   By minimality of $\beta_z$, we have
   $\gamma \cap \hat{U}_z \neq \emptyset$ as well.

   Now note that for any $x \in V(G)$,
   \[
      \E\left[|\hat{U}_x| \mid (G,\rho)\right] \leq \sum_{y \in U_x} \frac{1}{|B_G(y,r)|}\,.
   \]
   If we define a flow $F(G,x,y)\seteq \1_{\hat{U}_x}(y)$, then
   combining the preceding inequality with the mass-transport principle yields
   \[
      \Pr[\rho \in W_{r,r'}] \leq \E\left[\sum_{y \in U_{\rho}} \frac{1}{|B_G(y,r)|}\right]\,.\qedhere
   \]
\end{proof}

Finally, we move on to the proof of \pref{thm:uipq-speed}.

\begin{proof}[Proof of \pref{thm:uipq-speed}]
   We begin by showing that assumption (A) implies assumption (B).

\begin{lemma}[\cite{BP11}]
   \label{lem:BP}
   Let $H$ be a planar graph.  Consider $x \in V(H)$ and $\tau \geq 1$.
   If $B_H(x,4\tau)$ can be covered by $\lambda$
   balls of radius $\tau$,
   then there is a subset $W \subseteq B_H(x,6\tau) \setminus B_H(x,\tau)$
   whose removal separates $B_H(x,\tau)$ and $V(H) \setminus B_H(x,6\tau)$, and
   furthermore, $|W| \leq (\lambda+1)(2\tau+1)$.
\end{lemma}

Define
\[
   \nu_G(x,r) \seteq \min \left\{ |B_G(y,r)| : y \in B_G(x,6r)\right\}\,.
\]

By a straightforward packing/covering argument,
\pref{lem:BP} yields the following if $G$ is
almost surely planar:
For any $r \geq 2$, it holds that
\[
   \kappa_G(\rho; r, 6r) \lesssim r \frac{|B_G(\rho,6r)|}{\nu_G(x,r)}\,,
\]
and thus
\[
   q_G(\rho;r,6r) \lesssim r \frac{|B_G(\rho,6r)|}{\nu_G(x,r)^2}\,.
\]

Combining this with \pref{lem:separators-gen} and the next lemma
shows that assumption (A) implies assumption (B).

\begin{lemma}
   Suppose $(G,\rho)$ is a unimodular random graph and the following
   two conditions hold for some $C \geq 1$, $q > 1$, $\beta,\gamma > 0$, and $r \geq 2$:
         \begin{equation}\label{eq:qmom}
            \left(\E |B_G(\rho,6 r)|^q\right)^{1/q} \leq C r^d\,.
         \end{equation}
         \begin{equation}\label{eq:lowertail}
            \Pr\left[|B_G(\rho,r)| < \frac{\e}{C(\log r)^{\gamma} } r^d\right] \leq \exp(-1/\e^{\beta}) \qquad \forall \e > 0\,.
         \end{equation}
         Then:
   \[
      \E\left[\frac{|B_G(\rho,6r)|}{\nu_G(\rho,r)^2}\right] \leq C' r^{-d} (\log r)^{2/\beta+2\gamma}\,.
   \]
   for some $C'=C'(C,\beta,\gamma,d,q)$.
\end{lemma}

\begin{proof}
   Fix $\e > 0$.
   Let us define a flow
   \[
      F(G,x,y) \seteq \1_{\{|B_G(x,r)| < \e/(C (\log r)^{\gamma}) r^d\}} \1_{\{\dist_G(x,y) \leq 6r\}}\,.
   \]
   Then we can apply H\"older's inequality with $q'$ such that $1/q+1/q'=1$ to conclude that
   \begin{align*}
      \E\left[\sum_{x \in V(G)} F(G,\rho,x)\right] &= \E\left[|B_G(\rho,6r)| \1_{\{|B_G(\rho,r)| < \e/(C(\log r)^{\gamma}) r^d\}}\right]
      \\
      &\leq \left(\E |B_G(\rho,6r)|^q\right)^{1/q} \Pr[|B_G(\rho,r)| < \e/(C(\log r)^{\gamma}) r^d]^{1/q'} \\
      &\leq C r^d \exp(-1/(q' \e^{\beta}))\,.
   \end{align*}
   By the Mass-Transport Principle, this bounds $\E[\sum_{x \in V(G)} F(G,x,\rho)]=\Pr[\nu_r(\rho) < \e/(C(\log r)^{\gamma}) r^d]$, 
   hence for every $\e > 0$:
   \[
      \Pr\left[\nu_r(\rho) < \frac{\e}{C(\log r)^{\gamma} (q'd \log r)^{1/\beta}} r^d\right] \leq C r^{d(1-\e^{-\beta})}\,.
   \]
   Thus we can bound:
   \begin{equation}\label{eq:holder2}
      \left(\E\left[\frac{1}{\nu_r(\rho)^{2q'}}\right]\right)^{1/q'}
      \leq
      C_1 (\log r)^{2/\beta+2\gamma} r^{-2d}\,,
   \end{equation}
   where $C_1=C_1(C,q,\beta,\gamma,d)$.

   Now apply H\"older's inequality to \eqref{eq:qmom} and \eqref{eq:holder2} yields
   \[
      \E\left[\frac{|B_G(\rho,6r)|}{\nu_r(\rho)^2}\right] \leq C r^d \cdot C_1 (\log r)^{2/\beta+2\gamma} r^{-2d}\,,
   \]
   which gives the desired bound.
\end{proof}

Therefore it suffices to
prove the desired conclusion under assumption (B).
Let us now apply (B) to obtain, for every $r \geq 2$,
a marked unimodular network $(G,\rho,W_r)$ satisfying
properties (B)(i) and (B)(ii).

Define normalized conformal metrics $\{\omega_j : j \in \N\}$ on $(G,\rho)$ by
\[
   \omega_j \seteq \frac{\1_{W_{2^j}}}{\sqrt{\Pr[\rho \in W_{2^j}]}}\,,
\]
and note that from (B)(i), we have
\[
   \omega_j \geq j^{-\alpha/2} 2^{j(d-1)/2} \1_{W_{2^j}}\qquad \forall j \geq 1\,.
\]
Combining this with (B)(ii) gives for every $x,y\in V(G)$:
\begin{equation}\label{eq:distcomp}
   \dist_G(x,y) \geq 6 \cdot 2^j \implies \dist_{\omega_j}(x,y) \geq j^{-\alpha/2} 2^{j(d-1)/2}\,.
\end{equation}
Consider now the unnormalized metric
\[
   \omega^{(T)} \seteq \sqrt{\sum_{j=1}^{\lceil \log_2 T\rceil} \omega_{j}^2}\,,
\]
and observe the bound:  For $T \geq 2$,
\begin{equation}\label{eq:normb}
   \E [\omega^{(T)}(\rho)^2] \lesssim \log T\,.
\end{equation}

Finally, we apply \pref{thm:unimodular-markov-type}\eqref{eq:umt3} to give:
For every $T \geq 2$,
\begin{align*}
   \E\left[\dist_G(X_0,X_T)^{d-1}\right] &= \E\left[\min\{T^{d-1},\dist_G(X_0,X_T)^{d-1}\}\right] \\
                                   &\stackrel{\mathclap{\eqref{eq:distcomp}}}{\lesssim}
   (\log T)^{\alpha(d-1)} \E\left[\min \{T^{(d-1)/2},\dist_{\omega^{(T)}}(X_0,X_T)\}^{2}\right] \\
   &\lesssim
   T (\log T)^{\alpha(d-1)+5}\,,
\end{align*}
where our application of \pref{thm:unimodular-markov-type}\eqref{eq:umt3}
in the final inequality uses \eqref{eq:normb}.
\end{proof}

\subsection{Appendix: Growth bounds for UIPT/UIPQ}
\label{sec:appendix-curien}

We now sketch a justification
for why \pref{thm:uipq-speed} applies to UIPT/UIPQ with $d=4$.

\medskip
\noindent
{\bf Upper tail \eqref{eq:qmom}.}  It is known that $q$-th moment bounds of the form \eqref{eq:qmom}
hold for the size of the ``$r$-hull'' (which contains $B_G(\rho,r)$)
for $d=4$ and every $q < 3/2$.  See \cite{Krikun05} for UIPQ and \cite{Menard16} for UIPT.

\medskip
\noindent
{\bf Lower tail \eqref{eq:lowertail}.}  The stretched exponential lower tail is more delicate.
We sketch here the verification for UIPQ.
One can use Schaeffer's correspondence to establish that 
the law of $|B_G(\rho,r)|$ is the law of the
number of nodes $N_r$ of label at most $r$
in a random labeled tree (see \cite[\S 5]{CD06} and \cite[\S 2]{LM10}).
The spine of the tree is goverened by
a Markov process
$\{X_n \in \N : n \geq 0\}$, and
conditioned on $\{X_n\}$,
one has
\[
   N_{r} = \sum_{n : X_n \leq r} Y^{(n)}_{X_n}\,,
\]
where for each $i \geq 1$, $\{Y^{(n)}_i : n \geq 0\}$ are independent copies
of a nonnegative random variable $Y_i$.

The law of $Y_i$ is sampled as follows:  Let $T_i$ be a critical Galton-Watson tree
where the number of offspring is geometric with parameter $1/2$.
Inductively label the vertices of $T_i$ by a map $\ell~:~V(T_i)~\to~\Z$ as follows:  The root is labeled $i$.
If a node has label $\ell$, the labels of its offspring are independent
and uniform in the set $\{\ell-1,\ell,\ell+1\}$.
Let $\hat{T}_i$ have the law of $T_i$ conditioned on the event 
$\{\ell(x) > 0\ \forall x \in V(T_i)\}$.
Then $Y_i$ has the law
of $|\{x \in V(\hat{T}_i) : \ell(x) \leq r\}|$.
See, e.g., \cite[\S 2.3]{LM10}.

The tree $T_i$ satisifes, for all $N \geq 1$:
\[
   \Pr(|V(T_i)|=N) \gtrsim N^{-3/2}
\]
Let $W(h)$ denote the number of nodes in $T_i$ of height $h$.
Then there is a constant $c_1 > 0$ such that
\[
   \Pr\left[\sum_{h=1}^{N} W(h) \geq c_1 N^2 \bigmid |V(T_i)| \geq N^2\right] \geq c_1\,.
\]
See, for instance, \cite{Drmota09}.

From this, it is elementary to establish that
if $i \in [r/4,r/2]$, then $Y_i$ satisfies
\begin{equation}\label{eq:babyuni}
   \Pr[Y_i \geq \varepsilon r^4] \gtrsim c_2 \frac{r^{-2}}{\sqrt{\e}} \qquad \forall \e > 0, r \geq \e^{-4}\,.
\end{equation}

Therefore if $Z = \# \left\{ n \geq 1 : X_n \in [r/4,r/2]\right\}$, then
we have
\begin{equation}\label{eq:uipq1}
   \Pr\left[N_{r} < \e r^4 \mid Z\right] \leq
   \left(1-c_2\frac{r^{-2}}{\sqrt{\e}}\right)^{Z}
   \leq \exp\left(-c_2 \frac{Z r^{-2}}{\sqrt{\e}}\right)\,.
\end{equation}

The Markov process $\{X_n\}$ is a birth and death chain whose
transition probabilities converge to a limiting distribution as $X_n \to \infty$.
When scaled appropriately, $\{X_n\}$ converges to a Bessel process \cite[Prop. 5]{Menard10},
and satisfies, for some $c > 0$ and $r$ sufficiently large:
\[
   \Pr\left[\vphantom{\bigoplus}
   Z < c r^2 \right] \lesssim e^{-c r^2}\,.
\]
Combining this with \eqref{eq:uipq1} yields
\[
   \Pr\left[N_{r} < \e r^4\right] \leq
   \Pr[Z_{n} < c r^2] + \exp\left(-c_2 \frac{c}{\sqrt{\e}}\right)
   \lesssim
   e^{-c r^2}  + \exp\left(-c_2 \frac{c}{\sqrt{\e}}\right)\,.
\]
This verifies \eqref{eq:lowertail} for $\beta=1/2$ and $\gamma=0$.

\subsection*{Acknowledgements}

I would like to thank Itai Benjamini for his support and for 
sharing his many insights, along with a stream of
inspiring open questions.  My thanks to Asaf Nachmias
for reading many drafts of this manuscript
and sharing his wisdom on circle packings,
and to Omer Angel and Nicolas Curien for
invaluable
discussions on the geometry of random planar maps.
In particular, Nicolas offered crucial guidance
for the references and
calculations in \pref{sec:appendix-curien}.

I am grateful to Jian Ding, Russ Lyons,
Steffen Rohde, and Lior Silberman for comments
at various stages, and
to Tom Hutchcroft who pointed
out an error in an initial draft of this manuscript
and explained the construction found 
at the beginning of \pref{sec:homo}, as well
as a simpler proof of \pref{lem:is-amenable}.
Lastly, I am grateful to the anonymous referees for
their detailed, insightful feedback.

\bibliographystyle{alpha}
\bibliography{diffusive}

\end{document}